\documentclass[11pt,oneside,english]{amsart}
\RequirePackage{amsthm,amsmath,mathtools}

\usepackage[T1]{fontenc}
\usepackage{geometry}
\usepackage{subcaption}
\usepackage{babel,verbatim}
\usepackage{float}
\usepackage{amstext}
\usepackage{amssymb}
\usepackage{graphicx}
\usepackage[numbers]{natbib}
\RequirePackage[colorlinks=true,linkcolor=black,
citecolor=blue,urlcolor=blue]{hyperref}

\theoremstyle{plain}

\theoremstyle{plain}
\newtheorem{thm}{\protect\theoremname}
\theoremstyle{remark}
\newtheorem*{rem*}{\protect\remarkname}
\theoremstyle{remark}
\newtheorem{rem}{\protect\remarkname}
\theoremstyle{plain}
\newtheorem*{assumption*}{\protect\assumptionname}
\theoremstyle{plain}
\newtheorem{lem}{\protect\lemmaname}
\theoremstyle{plain}
\newtheorem{cor}{\protect\corollaryname}
\theoremstyle{plain}
\newtheorem{prop}{\protect\propositionname}
\theoremstyle{definition}
\newtheorem*{example*}{\protect\examplename}
\newtheorem{ex}{\protect\examplename}
\theoremstyle{definition}
\newtheorem*{algorithm}{\protect\algorithmname}
\theoremstyle{definition}
\newtheorem{asm}{\protect\assumptionname}

\numberwithin{equation}{section}
\numberwithin{figure}{section}

\PassOptionsToPackage{subsection=false,section=false}{beamerouterthememiniframes}
\usepackage{tikz}
\usetikzlibrary{shapes.multipart,shapes,arrows,decorations.pathreplacing}

\usepackage{mathtools,caption,tabularx}
\usepackage{algorithm,algpseudocode,multirow,bbm}
\newcommand{\1}[1]{\mathbbm{1}_{\{#1\}}}
\renewcommand{\O}{\mathcal{O}}
\newcommand{\C}{\mathcal{C}}
\newcommand{\E}{\mathbb{E}}
\newcommand{\V}{\mathbb{V}}
\renewcommand{\P}{\mathbb{P}}
\newcommand{\Z}{\mathbb{Z}}
\newcommand{\mS}{\mathcal{S}}
\newcommand{\R}{\mathbb{R}}
\newcommand{\N}{\mathbb{N}}
\newcommand{\ov}[1]{\overline{#1}}
\newcommand{\U}{\mathrm{U}}

\newcommand{\RW}{\text{\normalfont RW}}
\newcommand{\SB}{\text{\normalfont SB}}

\newcommand{\STE}{\text{\normalfont ST}}
\newcommand{\ISE}{\text{\normalfont IS}}

\DeclareRobustCommand{\stirling}{\genfrac\{\}{0pt}{}}

\usepackage{lineno}
\usepackage{pgfplots} 

\newcommand{\BrownianMotion}[7]{
\draw[#6] (#1,#2)
\foreach \x in {1,...,#3}
{   -- ++(#4,rand*#5)
}
node[below right] {#7};
}

\usepackage{caption}
\usepackage{nameref}
\DeclareCaptionLabelFormat{algnonumber}{}
\makeatother
  \providecommand{\algorithmname}{Algorithm}
\providecommand{\assumptionname}{Assumption}
  \providecommand{\examplename}{Example}
  \providecommand{\lemmaname}{Lemma}
  \providecommand{\propositionname}{Proposition}
  \providecommand{\remarkname}{Remark}
\providecommand{\corollaryname}{Corollary}
\providecommand{\theoremname}{Theorem}

\begin{document}

\title[Simulation of L\'evy Extrema]{Geometrically Convergent Simulation of the Extrema of L\'evy Processes}

\author{Jorge Gonz\'{a}lez C\'{a}zares, Aleksandar Mijatovi\'{c}, \and
	Ger\'{o}nimo Uribe Bravo}

\address{Department of Statistics, University of Warwick, \& The Alan Turing Institute, UK}

\email{jorge.gonzalez-cazares@warwick.ac.uk}

\address{Department of Statistics, University of Warwick, \& The Alan Turing Institute, UK}

\email{a.mijatovic@warwick.ac.uk}

\address{Universidad Nacional Aut\'{o}noma de M\'{e}xico, M\'{e}xico}

\email{geronimo@matem.unam.mx}

\begin{abstract}
We develop a novel approximate simulation algorithm for the joint 
law of the position, the running supremum and the time of the supremum of 
a general L\'evy process at an arbitrary finite time. We identify the law 
of the error in simple terms. We prove that the error decays geometrically in 
$L^p$ (for any $p\geq 1$) as a function of the computational cost, 
in contrast with the polynomial decay for the approximations available in the literature. 
We establish a central limit theorem and construct non-asymptotic and asymptotic confidence intervals
for the corresponding Monte Carlo estimator. 
We prove that the 
multilevel Monte Carlo 
estimator has optimal computational complexity (i.e. of order 
$\epsilon^{-2}$ if the mean squared error is at most $\epsilon^2$) for locally 
Lipschitz and barrier-type functions of the triplet 
and develop an unbiased version of the estimator. 
We illustrate the  performance of the algorithm with numerical examples. 
\end{abstract}



\maketitle

\section{Introduction}
\label{sec:intro}
Consider a L\'evy processes $X=(X_t)_{t\geq0}$ over the time interval $[0,T]$ 
for a given constant time $T>0$. The triplet $\chi=(X_T,\ov{X}_T,\tau_T)$, 
consisting of the position $X_T$, the supremum $\ov{X}_T$ of $X$ over the 
interval $[0,T]$ and the first time $\tau_T$ at which $X$ attains its supremum, 
plays a key role in numerous areas of applied probability (e.g. ruin 
probabilities in insurance mathematics~\cite{MR2099651}, 
barrier and lookback options and technical trading in mathematical 
finance~\cite{MR1936593,MR1932381,MijatovicDrawdown}, 
buffer size in queuing theory~\cite{MR1978607,MR3343283} and the 
prediction of the ultimate supremum and its time in optimal 
stopping~\cite{MR2932671,MR3273683}, to name a few). 
However, the information about the law of $\ov{X}_T$ (let alone of $\chi$) is 
very difficult to extract from the characteristics of the L\'evy process 
$X$~\cite{MR3098676}. Moreover, the known properties of the law of $\ov{X}_T$ 
are typically not explicit in the characteristics~\cite{MR3531705}, making its 
exact simulation challenging (e.g. the first exact simulation algorithm 
for the supremum of a stable process was developed recently~\cite{MR4032169}).

The central importance of $\chi$ in applied probability, combined with its 
intractability when $X$ is \emph{not} compound Poisson with drift, has lead to 
an abundance of works on its approximation over the last quarter of the 
century~\cite{MR1384357,MR1482707,MR1805321,MR2867949,MR2851060,MR2996014,
	MR2802466,MR2759203,MR2895413,MR3138603,MR3723380,MR3784492,ZoomIn}. 
These methodologies naturally yield Monte Carlo (MC) and Multilevel Monte Carlo 
(MLMC) algorithms for $\chi$. Without exception, the errors of these algorithms 
achieve \emph{polynomial} decay in the computational cost. The following 
natural question arises: \emph{does there exist an algorithm whose error decays 
geometrically in the cost?} The simple and general algorithm~\nameref{alg:SBA} below 
answers this question affirmatively. Subsection~\ref{subsec:Contribution} gives 
an intuitive introduction to the algorithm and describes its main properites, 
while Subsection~\ref{subsec:Related_Lit} 
compares~\nameref{alg:SBA} with the existing literature, 
cited at the beginning of this paragraph. (See also the YouTube  \href{https://youtu.be/P3vHmJUCFbU}{presentation}~\cite{Presentation_AM} for an overview of the paper.)

\subsection{Contribution\label{subsec:Contribution}}
The present paper has two main contributions: (I) a novel 
stick-breaking approximation (SBA) for $\chi$, given 
in Equation~\eqref{eq:SBA} and sampled by the 
algorithm~\nameref{alg:SBA} 
below, and an explicit characterisation of the law of its error 
(see Theorem~\ref{thm:error} below); (II) an analysis of the SBA 
as a Monte Carlo 
algorithm for functions of interest in applied probability. 
Contribution~(II),  described in Section~\ref{sec:main} below, includes the 
geometric decay of the strong error and the central limit theorem for the MC 
estimator based on~\nameref{alg:SBA} for various classes of functions 
of $\chi$ arising in applications (e.g. locally Lipschitz and barrier-type). 
Moreover, Section~\ref{sec:main} develops the MLMC and unbiased extensions 
of the SBA, both of which have optimal computational complexity. 
In the present subsection we describe Contribution~(I). 

The SBA is based on the stick-breaking representation of the triplet 
$\chi$, derived from the description of the law of the concave 
majorant of a L\'evy process given in~\cite{MR2978134}.
More precisely, the stick-breaking representation of $\chi$ states 
that the following a.s. equality holds  
\begin{equation}
\label{eq:chi}
\chi=(X_T, \ov X_T,\tau_T)=\sum_{k=1}^\infty
\left(Y_{L_{k-1}}-Y_{L_k},(Y_{L_{k-1}}-Y_{L_k})^+,\ell_k\cdot \1{Y_{L_{k-1}}-Y_{L_k}>0}\right),
\end{equation} 
where the L\'evy process $Y$ has the same law as $X$ and is 
independent of the stick-breaking process, $\ell=(\ell_n)_{n\in\N}$ 
on $[0,T]$, based on the uniform law $\U(0,1)$, 
i.e., $L_0=T$, $\ell_n=V_nL_{n-1}$ and $L_n=L_{n-1}-\ell_n$ for 
$n\in\N$ where $(V_n)_{n\in\N}$ is a $\U(0,1)$-iid sequence, 
see Figure~\ref{fig:stick}. 
(In~\eqref{eq:chi} and throughout the paper we denote 
$x^+=\max\{x,0\}$ for any $x\in\R$.) 
The coupling $(X,\ell,Y)$ satisfying the almost sure equality 
in~\eqref{eq:chi} is constructed in Subsection~\ref{subsec:coupling} 
below. Note that, in particular, it satisfies $Y_T=X_T$ a.s. 

\begin{figure}[ht]
	\begin{tikzpicture}
	\draw [thin, draw=gray, ->] (0,0) -- (15.5,0);
	\node[circle, fill=black, scale=.5, label=left:{$0$}] at (0,0){}{};
	\node[circle, fill=black, scale=.5, label=above:{$L_0=T$}] at (15.5,0){}{};
	\node[circle, fill=black, scale=.5, label=above:{$L_1$}] at (7.612,0){}{};
	\node[circle, fill=black, scale=.5, label=above:{$L_2$}] at (5.285,0){}{};
	\node[circle, fill=black, scale=.5, label=above:{$L_3$}] at (2.030,0){}{};
	\node[circle, fill=black, scale=.5, label=above:{$L_4$}] at (0.603,0){}{};
	
	\draw [
	decoration={
		brace,
		mirror,
		raise=0.15cm
	}, decorate] (7.622,0) -- (15.5,0) node 
	[pos=0.5,anchor=north,yshift=-.3cm] {$\ell_1$}; 
	
	\draw [
	decoration={
		brace,
		mirror,
		raise=0.15cm
	}, decorate] (5.295,0) -- (7.602,0) node 
	[pos=0.5,anchor=north,yshift=-.3cm] {$\ell_2$}; 
	
	\draw [
	decoration={
		brace,
		mirror,
		raise=0.15cm
	},
	decorate] (2.040,0) -- (5.275,0) node 
	[pos=0.5,anchor=north,yshift=-.3cm] {$\ell_3$};
	
	\draw [
	decoration={
		brace,
		mirror,
		raise=0.15cm
	}, decorate] (0.603,0) -- (2.020,0) node 
	[pos=0.5,anchor=north,yshift=-.3cm] {$\ell_4$};
	\end{tikzpicture}
	\caption{\footnotesize 
		The figure illustrates the first $n=4$ sticks of a 
		stick-breaking process. The increments of $Y$ in~\eqref{eq:chi} are taken 
		over the intervals $[L_k,L_{k-1}]$ of length $\ell_k$. Crucially, the time  
		$L_n$ featuring in the vector $(Y_{L_n},\ov{Y}_{L_n},\ov\tau_{L_n}(Y))$
		in~\eqref{eq:error-law} of Theorem~\ref{thm:error} is exponentially
		small in $n$ and independent of $Y$.\label{fig:stick}}
\end{figure}

Given the representation in~\eqref{eq:chi}, the SBA is defined as follows:
\begin{equation}
\label{eq:SBA}
\begin{split}
\chi_n =\sum_{k=1}^n& 
\left(Y_{L_{k-1}}-Y_{L_k},(Y_{L_{k-1}}-Y_{L_k})^+,\ell_k\cdot \1{Y_{L_{k-1}}-Y_{L_k}>0}\right) \\
+&
\left(Y_{L_n},Y_{L_n}^+,L_{n}\cdot \1{Y_{L_n}>0}\right).
\end{split}
\end{equation}
Since the residual sum $\sum_{k=n+1}^\infty \left(Y_{L_{k-1}}-Y_{L_k}\right)$ equals 
$Y_{L_n}$ for any $n\in\N$, the first component of $\chi_n$ coincides with 
that of $\chi$, while, as we shall see in Theorem~\ref{thm:error} below, 
$Y_{L_n}^+$ and $L_n\cdot \1{Y_{L_n}>0}$ reduce the errors of the 
corresponding partial sums in~\eqref{eq:SBA}. The coupling 
$(X,\ell,Y)$ makes it possible to compare $\chi$ and $\chi_n$ 
on the same probability space and analyse the strong error 
$\chi-\chi_n$. 

Denote the distribution of $X_t$ by $F(t,x)=\P(X_t\leq x)$, where  $x\in\R$ 
and $t>0$. Then an algorithm that simulates exactly from the law of the SBA 
$\chi_n$ is given as follows:

\begin{algorithm} 
	\caption{SB-Alg}
	\label{alg:SBA}
	\begin{algorithmic}[1]
		\Require{$n\in\N$, fixed time horizon $T>0$} 
		\State{Set $\Lambda_0=T$, $\mathcal{X}_0=(0,0,0)$}
		\For{$k=1,\ldots,n$}
		\State{Sample $\upsilon_k\sim \U(0,1)$ and put $\lambda_k=\upsilon_k\Lambda_{k-1}$ 
			and $\Lambda_k=\Lambda_{k-1}-\lambda_{k}$}
		\State{Sample $\xi_k\sim F(\lambda_k,\cdot)$ and put
			$\mathcal{X}_k= 
			\mathcal{X}_{k-1}+ 
			(\xi_k,\xi^+_k,\lambda_k\cdot \1{\xi_k>0})$} 
		\EndFor
		\State{Sample $\varsigma_n\sim F(\Lambda_n,\cdot)$ and \Return $\mathcal{X}_n+(\varsigma_n,\varsigma^+_n,\Lambda_n\cdot \1{\varsigma_n>0})$}
	\end{algorithmic}
\end{algorithm}

\nameref{alg:SBA} clearly outputs a random vector with the 
same law as $\chi_n$ in~\eqref{eq:SBA}, using a total of $n+1$ 
sampling steps. Theorem~\ref{thm:error} and Section~\ref{sec:main} 
below show that $\chi_n$ in~\eqref{eq:SBA} is an increasingly 
accurate approximation of $\chi$ as $n$ grows. 
Intuitively this is because, by~\eqref{eq:levy-minorant}, 
the sum in the definition of $\chi_n$ consists of the first $n$ terms 
taken in a size-biased order (see Subsection~\ref{subsec:coupling} 
below) making the remainder very small. It will become clear 
from Theorem~\ref{thm:error} that the last step 
in~\nameref{alg:SBA} reduces the error further. The computational 
cost of the algorithm is proportional to $n$ if we can sample any 
increment of $X$ in constant time. We stress 
that~\nameref{alg:SBA} is \textit{not} a version of the random walk 
approximation (see Equation~\eqref{eq:RWA} below) on a 
randomised grid as it does not require the computation of either $\max$ 
or $\arg\max$ of a discretisation of $X$. Instead, the approximation for the 
supremum and its time are obtained by summing non-negative numbers, 
making~\nameref{alg:SBA} numerically very stable. The convergence 
analysis of~\nameref{alg:SBA} relies on the following result, 
which describes explicitly the law of its error.

The following notation, needed to state Theorem~\ref{thm:error}, 
will be used throughout the paper: for a right-continuous 
function $f:[0,\infty)\to\R$ with left-hand limits, we denote by 
$\ov f_t=\sup\{f_s:s\in[0,t]\}$ its supremum over the interval 
$[0,t]$ and by $\tau_t(f)=\inf\{s\in[0,t]:\ov{f}_s=\ov{f}_t\}$ the 
first time the supremum $\ov f_t$ is attained. 

\begin{thm}\label{thm:error}
Assume the L\'evy process $X$ is not compound Poisson with 
drift and let $(X,\ell,Y)$ be the coupling constructed in 
Subsection~\ref{subsec:coupling} below, satisfying~\eqref{eq:chi}. 
For any $n\in\N$, define the vector of errors of the SBA by 
\begin{equation}
\label{def:Delta}
\begin{split}
\chi-\chi_n &=\big(0,\Delta^\SB_n,\delta^\SB_n\big)=
(0,\Delta_n-Y_{L_n}^+,\delta_n-L_n\cdot \1{Y_{L_n}>0}),
\quad\text{ where}\\ 
\Delta_n &=\ov X_T-\sum_{k=1}^n(Y_{L_{k-1}}-Y_{L_k})^+
\quad\text{and}\quad
\delta_n=\tau_T-\sum_{k=1}^n \ell_k\cdot \1{Y_{L_{k-1}}-Y_{L_k}>0}.
\end{split}
\end{equation} 
Then, conditionally on $L_n$,
\begin{equation}
\label{eq:error-law}
\begin{split}
(Y_{L_n},\Delta_n,\delta_n) 
& \overset{d}{=}\big(Y_{L_n},\ov Y_{L_n},
\tau_{L_n}\big(Y\big)\big),\quad\text{and hence}\\
\big(\Delta^\SB_n,\delta^\SB_n\big) 
& \overset{d}{=}\big(\ov Y_{L_n}-Y_{L_n}^+,
\tau_{L_n}\big(Y\big)-L_n\cdot \1{Y_{L_n}>0}\big).
\end{split}
\end{equation}
Moreover, the inequalities $0\leq\Delta^\SB_{n+1}\leq\Delta^\SB_n\leq\Delta_n$, 
$0\leq\delta_n\leq L_n$ and $|\delta^\SB_n|\leq L_n$ hold a.s. 
\end{thm}

Non-asymptotic (i.e. for fixed $n$) explicit descriptions of the law of the 
error, such as~\eqref{eq:error-law} in Theorem~\ref{thm:error}, are not common 
among the simulation algorithms for the supremum and related functionals of the 
path. Since $L_n$ and $Y$ are independent, the representation 
in~\eqref{eq:error-law} is easy to work with and provides a cornerstone for the 
results of Section~\ref{sec:main}. 
Note that, by Theorem~\ref{thm:error},
the sequences $(\Delta_n^\SB)_{n\in\N}$, $(\Delta_n)_{n\in\N}$ and 
$(\delta_n)_{n\in\N}$ are nonincreasing almost surely and converge to $0$. 
Furthermore, the following observations based on Theorem~\ref{thm:error} 
motivate the final step in~\nameref{alg:SBA} (i.e. the inclusion of the last 
summand in the definition in~\eqref{eq:SBA}): (I) the tail of the error $\Delta^\SB_n$ may 
be strictly lighter than that of $\Delta_n$ (as $\ov X_t-X_t^+=\min\{\ov X_t,
\ov X_t-X_t\}$ and $\ov X_t-X_t\overset{d}{=}\sup_{s\in[0,t]}(-X_s)$ for all 
$t>0$~\cite[Prop.~VI.3]{MR1406564}); (II) for a large class of L\'evy processes, 
$\delta^\SB_n$ is \emph{asymptotically centred} at $0$, i.e. 
$\E[\delta_n^\SB/L_n]\to 0$ as $n\to\infty$, while $\E[\delta_n/L_n]$ converges 
to a strictly positive constant (see Proposition~\ref{prop:E_tau} below for 
details). Theorem~\ref{thm:error} is proved in 
Subsection~\ref{subsec:Proof_of_Thm1}.

Since $\E L_n=T2^{-n}$ and $L_n$ is independent of $Y$, 
the convergence of~\nameref{alg:SBA} is geometric 
(see also Section~\ref{sec:main}). 
Indeed, the error $(\Delta_n^\SB,\delta_n^\SB)$, satisfies the 
following weak limit. 

\begin{cor}
\label{cor:zoom}
If weak convergence $X_{t}/a\left(t\right)\overset{d}{\to}Z_{1}$ (as $t\searrow0$)
holds for some (necessarily) $\alpha$-stable process $Z$ and
a  function $a$, which is necessarily $1/\alpha$-regularly varying at zero, then
\begin{equation}
\left(\frac{Y_{L_n}}{a(L_n)},\frac{\Delta_{n}}{a(L_{n})},\frac{\Delta_n^\SB}{a(L_{n})},
	\frac{\delta_{n}}{L_{n}},\frac{\delta_{n}^\SB}{L_{n}}\right) 
\overset{d}{\to} \left(Z_1,\ov{Z}_{1},\ov{Z}_{1}-Z_{1}^+,
	\tau_1(Z),\tau_1(Z)-\1{Z_1>0}\right)
\enskip\text{as}\enskip n\to\infty.
\label{eq:zoom-in}
\end{equation}
\end{cor}

The assumption in Corollary~\ref{cor:zoom} essentially amounts to both tails of the 
L\'evy measure of $X$ being regularly varying at zero with index $-1/\alpha$
(see~\cite[Thm~2]{MR3784492}). This is a rather weak requirement, typically 
satisfied by L\'evy based models in applied probability, which allows an arbitrary 
modification of the  L\'evy measure away from zero 
(see discussion in~\cite[Sec.~4]{MR3784492}). Moreover, the index $\alpha$ 
is given by~\eqref{eq:alpha} and the function $a(t)$ 
is typically of the form $a(t)\sim C_0 t^{1/\alpha}$ for some constant $C_0>0$.
The scaling in the limit~\eqref{eq:zoom-in} is stochastic; however, 
since $\E L_{n}=T2^{-n}$, the rate of decay of the error is clearly geometric. 
Corollary~\ref{cor:zoom} is proved in Subsection~\ref{subsec:Proof_of_Thm1} 
by applying Theorem~\ref{thm:error} to the small-time weak limit of $X$. 

\subsection{Connections with existing literature}\label{subsec:Related_Lit}
In the present subsection we discuss briefly the literature on the 
approximations of $\chi$ and compare it with~\nameref{alg:SBA}.

The random walk approximation (RWA) (defined in~\eqref{eq:RWA} below) is based 
on $(X_{kT/n})_{k\in\{1,\ldots,n\}}$, the skeleton of the L\'evy process $X$. It 
is a widely used method for approximating $\chi$ with computational cost 
proportional to the discretisation parameter $n$. In the case of Brownian motion, 
the asymptotic law of the error was studied in~\cite{MR1384357}. The 
papers~\cite{MR1482707,MR1805321} (resp.~\cite{MR2867949,MR2851060})
identified the dominant error term of the RWA for barrier and lookback options 
under the exponential L\'evy models when $X$ is a Brownian motion with drift 
(resp. jump diffusion). Based on Spitzer's identity,~\cite{MR2996014} developed 
bounds on the decay of the error in $L^1$ for general L\'evy processes, extending the results of~\cite{MR2867949}. 
Ideas from~\cite{MR3784492} were employed in~\cite{ZoomIn} to obtain sharper 
bounds on the convergence of the error of the RWA in $L^p$ for general L\'evy 
processes and any $p>0$. Such results are useful for the analysis of MC and MLMC 
schemes based on the RWA, see~\cite{MR3723380} for  
the case of certain parametric L\'evy models. We will describe in more detail 
these contributions in Section~\ref{sec:main} as we contrast them with the 
analogous results for~\nameref{alg:SBA}.

Exploiting the the Wiener-Hopf factorisation,~\cite{MR2895413} 
introduced the Wiener-Hopf approximation (WHA) of $(X_T,\ov X_T)$. This 
approximation is given by $(X_{G_n},\ov X_{G_n})$, where $G_n$ is the sum of 
$n$ independent exponential random variables with mean $T/n$, so that $\E G_n=T$ 
with variance $T^2/n$. Implementing the WHA requires the ability to sample the 
supremum at an independent exponential time, which is only done approximately 
for a specific parametric class of L\'evy processes with exponential moments and 
arbitrary path variation~\cite{MR2895413}. The computational cost of the WHA is 
proportional to $n$. The decay  of the bias and the MLMC version of the WHA were 
later studied in~\cite{MR3138603}. As observed in~\cite[Sec.~1]{MR3723380}, the 
WHA currently cannot be directly applied to various parametric models used in 
practice which possess  increments that can be simulated exactly (e.g. the 
variance gamma process).

The jump-adapted Gaussian approximation (JAGA) was introduced 
in~\cite{MR2759203,MR2802466} to approximate Lipschitz 
functions in the supremum norm of L\'evy-driven stochastic 
differential equations with Lipschitz coefficients. 
The algorithm is based on an approximation of the skeleton $\{X_{t_k}\}_{k=1}^n$ 
where the time grid includes the times of the jumps of $X$ whose magnitude is 
larger than some cutoff level $\kappa$ and the small-jump component of $X$ 
is approximated by an additional Brownian motion. Typically, the cost and 
bias of the JAGA are proportional to $n+\kappa^{-\beta}$ and 
$(n^{-1/2}+n^{1/4}\kappa)\sqrt{\log n}$, respectively, where $\beta$ is the 
Blumenthal-Getoor index, see~\eqref{def:I0_beta}. 
The complexity of the MLMC 
version of the JAGA for Lipschitz functions of $(X_T,\ov X_T)$ 
is compared with that of~\nameref{alg:SBA} in 
Subsection~\ref{subsec:Complexity_SBA} below. 

In contrast with Theorem~\ref{thm:error} for the SBA, the laws of the errors 
of all the other algorithms discussed in the present subsection are intractable.
The error of the SBA $\chi_n$ in~\eqref{eq:SBA} decays geometrically in $L^p$ 
(see Theorem~\ref{thm:Lp} below) as opposed to the polynomial decay for the 
other algorithms (see Subsection~\ref{subsec:Lit_Lp} below). The error in $L^p$ 
of the SBA applied to locally Lipschitz and barrier-type functions arising in 
applications also decays geometrically (see Propositions~\ref{prop:LocLip} 
\&~\ref{prop:barrier} below). To the best of our knowledge, such errors have 
not been analysed for algorithms other than the RWA, which has polynomial 
decay (see Subsection~\ref{subsec:Lit_functionals} for details). The rate of 
the decay of the bias is directly linked to the computational complexity of MC 
and MLMC estimates. Indeed, if the mean squared error is to be at most 
$\epsilon>0$, the MC algorithm based on the SBA has (near optimal) complexity of 
order $\O(\epsilon^{-2}\log\epsilon)$ (see Appendix~\ref{sec:O_o} 
below for the definition of $\O$). The MLMC scheme based on~\nameref{alg:SBA} has (optimal) complexity of order $\O(\epsilon^{-2})$, 
which is in general neither the case for the RWA~\cite{MR3723380} nor the 
WHA~\cite{MR3138603} 
(see details in Subsection~\ref{subsec:Lit_complexity}). 

\vspace{-5mm}

\subsection{Organisation\label{subsec:organisation}}
The remainder of this paper is organised as follows. We develop the theory for 
the SBA as a Monte Carlo algorithm in Section~\ref{sec:main}. Each result is 
compared with its analogue (if it exists) for the algorithms discussed in
Subsection~\ref{subsec:Related_Lit} above. 
In Section~\ref{sec:Numerical-Examples} we provide numerical examples 
illustrating the performance of~\nameref{alg:SBA}. The proofs of the
results in Sections~\ref{sec:intro} and~\ref{sec:main} are presented in 
Section~\ref{sec:proofs}.

\section{SBA Monte Carlo: theory and applications\label{sec:main}}

The present section describes the geometric convergence of~\nameref{alg:SBA} and analyses the Monte Carlo estimation of the 
functions of interest in applied probability. 
In Subsection~\ref{subsec:Lp} we establish the geometric decay of the error in 
$L^p$. In Subsection~\ref{subsec:Functionals} we show that the error in $L^p$ (and hence the bias) of~\nameref{alg:SBA} applied to the aforementioned functions also decays 
geometrically. 
In Subsection~\ref{subsec:CLT}
we study the error of the MC estimator based on~\nameref{alg:SBA}
for  the expected value of those functions via a central limit theorem 
and provide the corresponding asymptotic and non-asymptotic 
confidence intervals. 
Subsection~\ref{subsec:Complexity_SBA} gives the computational complexity of 
the MC and MLMC estimators based on~\nameref{alg:SBA}. 
Subsection~\ref{subsec:UnbiasedEstimators} describes 
the unbiased estimator. 

\subsection{Geometric decay in $L^p$ of the error of the SBA}
\label{subsec:Lp}
In the present subsection we study the decay in $L^p$ of the error 
$(\Delta_n^\SB,\delta_n^\SB)$ of the SBA $\chi_n$ given in~\eqref{def:Delta}. 
Let $(\sigma^2,\nu,b)$ be the generating triplet of $X$ associated with the 
cutoff function $x\mapsto \1{|x|<1}$ (see~\cite[Ch.~2, Def.~8.2]{MR3185174}).
The existence of the moments of $X_T$ and $\ov X_T$, necessary for 
the following result, can be characterised~\cite[Thm~25.3]{MR3185174} in 
terms of the integrals
\begin{equation}\label{def:I_pm}
I_+^p=\int_{[1,\infty)}x^p\nu(dx),
\quad I_-^p=\int_{(-\infty,-1]}|x|^p\nu(dx),\quad p\geq 0.
\end{equation}
Throughout we use the standard $\O$ notation, see
Appendix~\ref{sec:O_o} below for definition.

\begin{thm}\label{thm:Lp} 
Under the assumptions of Theorem~\ref{thm:error}, the following holds for 
any $p\geq1$.
\item[{\normalfont(a)}] The inequality 
$\max\{\E\big[|\delta_n^\SB|^p\big],\E[\delta_n^p]\}\leq T^p(1+p)^{-n}$ 
holds for any $n\in\N$.
\item[{\normalfont(b)}] If $\min\{I_+^p,I_-^p\}<\infty$ (resp. $I_+^p<\infty$),
then $\E[(\Delta_n^\SB)^p]$ (resp. $\E[\Delta_n^p]$) is bounded above by 
$\O(\eta_p^{-n})$ as $n\to\infty$, where $\eta_p$ lies in the interval 
$[3/2,2]$ for any L\'evy process $X$. Both $\eta_p$, defined 
in~\eqref{eq:eta}, and the constants in $\O(\eta_p^{-n})$ are explicit in the 
characteristics $(\sigma^2,\nu,b)$ of 
$X$ (see~\eqref{eq:explicit_bounds}). 
\end{thm}

By Theorem~\ref{thm:error},
the error $\Delta_n^\SB$
is  bounded above by the supremum of the L\'evy process over the stochastic interval $[0,L_n]$
with average length equal to $\E L_n = T2^{-n}$.
The key step in the proof of Theorem~\ref{thm:Lp}, 
given in Lemma~\ref{lem:LevyMomBound} below, consists of controlling 
the expectation of the supremum of $X$ over short time intervals
(see Subsection~\ref{subsec:proofs-Lp} below for details).

Since $\eta_2=2$ (see definition in~\eqref{eq:eta}), an application 
of Theorem~\ref{thm:Lp}(b) for $p\in\{1,2\}$ yields $\E\Delta^\SB_n=
\O((3/2)^{-n})$ and $\E\big[\big(\Delta^\SB_n\big)^2\big]=\O(2^{-n})$. 
These two moments are used in the analysis of the MLMC estimator based on~\nameref{alg:SBA} (see Subsection~\ref{subsec:Complexity_SBA} below).
A further application of Theorem~\ref{thm:Lp} yields a geometric bound on the $L^p$-Wasserstein distance $\mathcal{W}_p(\mathcal{L}(\chi),
\mathcal{L}(\chi_n))$ between the laws $\mathcal{L}(\chi)$ and 
$\mathcal{L}(\chi_n)$ of the corresponding random vectors (see~\eqref{eq:Lp_Wd}  below for the definition of the Wasserstein distance 
and Subsection~\ref{subsec:proofs-Lp} for the proof of 
Corollary~\ref{cor:Wp}).

\begin{cor}\label{cor:Wp}
Assume $\min\{I_+^p,I_-^p\}<\infty$ for some $p\geq 1$. 
Under the assumptions of Theorem~\ref{thm:error} we have
$\mathcal{W}_p(\mathcal{L}(\chi),\mathcal{L}(\chi_n))=\O(\eta_p^{-n/p})$ as 
$n\to\infty$. As in Theorem~\ref{thm:Lp}(b) above, $\eta_p$ lies in the interval 
$[3/2,2]$ and the constant in $\O(\eta_p^{-n/p})$, given in 
Equation~\eqref{eq:Wp}, is explicit. 
\end{cor}

\subsubsection{Comparison}\label{subsec:Lit_Lp}
The algorithm based on the RWA with time-step $T/n$
outputs 
\begin{equation}\label{eq:RWA}
\left(X_T,\max_{k\in\{1,\ldots,n\}}X_{kT/n},
\frac{T}{n}\arg\max_{k\in\{1,\ldots,n\}}X_{kT/n}\right).
\end{equation}
The $L^1$ bounds on the error $\Delta_n^\RW = 
\ov X_T - \max_{k\in\{1,\ldots,n\}}X_{kT/n}$ 
have a long history. Using the weak limit of the error of the RWA, 
the $L^1$ bound $\E\Delta^\RW_n=\O(n^{-1/2})$ is established for the 
Brownian motion with drift in~\cite{MR1384357,MR1805321}. 
The same bound holds when the jumps of $X$ have finite activity (i.e. 
$\nu(\R)<\infty$ and $\sigma\neq0$)~\cite{MR2867949}. 
The approach of~\cite{MR2867949}, based on Spitzer's identity, was extended 
in~\cite[Thm~5.2.1]{MR2996014} to the case without a Brownian component. If $X$ 
has paths of finite variation, these bounds were further improved via a different 
methodology in~\cite{ZoomIn}. In particular, by~\cite[Thm~4.1]{ZoomIn}, we have: 
$\E\Delta^\RW_n=\O(n^{-1/2})$ if $X$ has a Brownian component (i.e.
$\sigma\neq0$), $\E\Delta^\RW_n=\O(n^{-1})$ if $X$ has paths of finite 
variation (i.e. $\int_{(-1,1)}|x|\nu(dx)<\infty$ and $\sigma=0$) and 
$\E\Delta^\RW_n=\O(n^{\delta-1/\beta})$ otherwise, 
for any small $\delta>0$ and $\beta\in[1,2]$ defined 
in~\eqref{def:I0_beta} below. 

Bounds for $\E\big[\big(\Delta_n^\RW\big)^p\big]$, $p>0$, analysed 
in~\cite{MR2867949,ZoomIn}, are as follows. By~\cite[Thm~4.1]{ZoomIn}, for 
$\alpha\in[0,2]$ given in~\eqref{eq:alpha} below, the decay is $\O(n^{-1})$ for 
$p>\alpha$ and $\O(n^{\delta-p/\alpha})$ for $0<p\leq\alpha$ and any small 
$\delta>0$ (we may take $\delta=0$ if either $\alpha=1$ and $X$ is of finite 
variation or $\alpha=2$). If $X$ is spectrally negative (i.e. $\nu((0,\infty))=0$) and 
has jumps of finite variation (i.e. $\int_{(-1,0)}|x|\nu(dx)<\infty$), then for 
$p>1$ the decay is of order $\O(n^{-p})$ (resp. $\O(n^{-p/2}\log(n)^p)$) if 
$\sigma=0$ (resp. $\sigma\neq0$)~\cite[Lem.~6.5]{MR2867949}. Interestingly, as 
noted in~\cite[Rem.~4.4]{ZoomIn}, if $X$ has jumps of both signs, then for any 
$p>0$, the error of the RWA satisfies 
$\liminf_{n\to\infty}n\E\big[\big(\Delta_n^\RW\big)^p\big]>0$. Put differently, 
the error cannot be of order $o(n^{-1})$ (see Appendix~\ref{sec:O_o} 
below for the definition of $o$). 

Intuitively, the error committed by the RWA is due to the skeleton missing the 
fluctuations of the process over the interval of length $1/n$ where the process 
attained its supremum. Since these fluctuations can be substantial in the presence 
of high jump activity and heavy tails, the decay of the resulting error is polynomial 
in $n$. In contrast, the error of the SBA is by 
Theorem~\ref{thm:Lp}(b) bounded by  $\O(\eta_p^{-n})$ with $\eta_p\in[3/2,2]$, as 
it commits the same error as the RWA but over the interval $[0,L_n]$ with 
average length of $T/2^n$. Numerical results show that the biases of the RWA and 
the SBA over $2^n$ and $n$ steps, respectively, are comparable 
(Figure~\ref{fig:lookback_TSBM} below).

Recall that the WHA, applicable to a specific parametric class of L\'evy 
processes~\cite{MR2895413}, is given by $(X_{G_n},\ov X_{G_n})$, where $G_n$ is 
an independent gamma random variable with mean $\E G_n=T$ and variance $T^2/n$. 
Since $\ov X_{s+t}-\ov X_s$ is stochastically dominated by $\ov X_t$ and 
$X_{t+s}-X_s\overset{d}{=}X_t$, the $L^p$ norm of the error is linked to both, 
the small time behaviour of $t\mapsto(X_t,\ov X_t)$ and the deviations of $G_n$ 
from $T$. Therefore, the moments of the errors depend on those of $|G_n-T|$ and 
satisfy $\E[|X_T-X_{G_n}|^p]=\O(n^{-1/q})$ and $\E[|\ov X_T-\ov X_{G_n}|^p]=
\O(n^{-1/q})$ for $p\in\{1,2\}$, where $q=4$ if $p=1$ and $X$ is of infinite 
variation and $q=2$ otherwise~\cite[Prop.~4.5]{MR3138603}. 
These bounds are based on a martingale decomposition of 
the L\'evy process $X$ (see~\cite[Lem.~4.4]{MR3138603}), 
while the analogous results in our paper use the L\'evy-It\^o 
decomposition, see Lemma~\ref{lem:LevyMomBound} below.

Intuitively, the error in the WHA is due to the censored fluctuations of $X$
over a stochastic interval of length $|G_n-T|$. This is analogous to the error 
of the SBA over a stochastic interval of length $L_n$. However, since
$\E[|G_n-T|]$ is asymptotically equal to $T\sqrt{2/(n\pi)}$ (by the central 
limit theorem and~\cite[Thm~5.4]{MR1700749}) and $\E[L_n]=T2^{-n}$, the 
speed of convergence is polynomial in the WHA and geometric in the SBA.

The first two moments of the error of the JAGA with cost $n$ were 
analysed in~\cite{MR2802466,MR2759203}, resulting in the bound 
$\O(n^{-\min\{1,1/\beta_+\}}+n^{1/4-1/\beta_+}\sqrt{\log n})$ if $X$ 
has no Brownian component (i.e. $\sigma=0$) and 
$\O(n^{1/4-\min\{3/4, 1/\beta_+\}}\sqrt{\log n})$ otherwise, 
where $\beta_+$, given in~\eqref{eq:BG+}, is \emph{slightly} larger 
than the Blumenthal-Getoor index $\beta\in[0,2]$ 
in~\eqref{def:I0_beta}. 
Intuitively, this error is the result of missing the fluctuations of $X$ between 
consecutive points on the random grid and the error incurred from approximating 
the small-jump component with an additional Brownian motion. 

\subsection{SBA for certain functions of $\chi$: geometric decay of the strong error\label{subsec:Functionals}}
Throughout the paper we consider a measurable function 
$g:\R\times\R_+\times[0,T]\to\R$ satisfying $\E|g(\chi)|<\infty$, 
where $\R_+=[0,\infty)$. We focus our 
attention on the classes of functions that arise in application areas such as 
financial mathematics~\cite{schoutens2003levy,MR2042661}, risk 
theory~\cite{schoutens2010levy,MR2766220} and 
insurance~\cite{KyprianouCoCos}. More specifically, we study the following 
three classes of functions: 
(I) Lipschitz in Proposition~\ref{prop:Lipschitz}, 
(II) locally Lipschitz in Proposition~\ref{prop:LocLip} and 
(III) barrier-type in Proposition~\ref{prop:barrier}.
These results are a consequence of the representation of the law of the error in 
Theorem~\ref{thm:error}, bounds from Theorem~\ref{thm:Lp} and a tail estimate 
(without integrability assumptions) for the error $\Delta_n$ 
in Lemma~\ref{lem:bound}.


Lipschitz functions arise in applications, for example, in the pricing of 
hindsight~\cite{MR1805321,MR2023872,MR2867949,MR3723380} and perpetual 
American~\cite{MR1932381} puts under exponential L\'evy models. Indeed, for 
fixed $S_0,K_0>0$, these two examples require computing the expectations of 
$(K_0-S_0e^{X_T-\ov{X}_T})^+$ and $e^{X_T-\ov{X}_T}$, both of which are bounded 
and Lipschitz in $(X_T,\ov{X}_T)$ since $\ov{X}_T\geq X_T$.  
The next result, proved in Subsection~\ref{proof:Lip} below, shows that the 
convergence of~\nameref{alg:SBA} is also geometric for these functions.

\begin{prop}\label{prop:Lipschitz}
Assume $|g(x,y,t)-g(x,y',t')|\leq K(|y-y'|+|t-t'|)$ 
for some $K>0$ and all $x\in\R$, $y,y'\in\R_+$, $t,t'\in[0,T]$. 
Suppose $p\geq 1$ satisfies $\min\{\|g\|_\infty,I_+^p,I_-^p\}<\infty$, where 
$\|g\|_\infty=\sup\{|g(x,y,t)|:(x,y,t)\in\R\times\R_+\times[0,T]\}$, and 
let $\eta_p\in[3/2,2]$ be as in~\eqref{eq:eta}. Then, under the assumptions 
of Theorem~\ref{thm:error}, we have \[\E[|g(\chi)-g(\chi_n)|^p]=\O(\eta_p^{-n})
\quad\text{as }n\to\infty.\] Moreover, the constant in $\O(\eta_p^{-n})$, given in 
Equation~\eqref{eq:LipBound} below, is explicit in $K$, $\|g\|_\infty$ and the 
characteristics $(\sigma^2,\nu,b)$ of the L\'evy process $X$.
\end{prop}

The pricing of lookback puts, hindsight calls~\cite{MR1805321,MR2867949,
MR3723380} and perpetual American calls~\cite{MR1932381} involve expectations 
of continuous functions of $\chi$, such as $(S_0e^{\ov X_T}-K_0)^+$ and 
$e^{\ov X_T}$, which are only locally Lipschitz. By Proposition~\ref{prop:LocLip}, 
under appropriate assumptions on large positive jumps, the error of~\nameref{alg:SBA} decays geometrically for such functions. 

\begin{prop}\label{prop:LocLip} 
Assume that $|g(x,y,t)-g(x,y',t')|\leq K
(|y-y'|+|t-t'|)e^{\lambda\max\{y,y'\}}$ for some $K,\lambda>0$ 
and all $(x,y,y^{\prime},t,t^{\prime})\in\R\times\R_+^2\times[0,T]^2$. 
Let $p\geq1$ and $q>1$ satisfy $\int_{[1,\infty)}e^{\lambda pqx}\nu(dx)<\infty$ 
and let $\eta_{pq'}\in[3/2,2]$ be as in~\eqref{eq:eta}, where $q'=(1-1/q)^{-1}$. 
Then, under the assumptions of Theorem~\ref{thm:error}, 
\[\E[|g(\chi)-g(\chi_n)|^p] = \O\big(\eta_{pq'}^{-n/q'}\big)\quad\text{as }n\to\infty.\]
Moreover, the constant in $\O\big(\eta_{pq'}^{-n/q'}\big)$, given in 
Equation~\eqref{eq:LocLipBound} below, is explicit in $p,q,K,\lambda$ and 
the characteristics $(\sigma^2,\nu,b)$ of the L\'evy process $X$.
\end{prop}

In order to obtain the smallest value $\eta_{pq'}^{-1/q'}$ in Proposition~\ref{prop:LocLip}, 
one needs to take the largest possible $q$ allowed by the assumptions (see Remark~\ref{rem:LocLip} below for 
details). Hence, the rate of decay is determined by the exponential moments of the 
L\'evy measure $\nu|_{[1,\infty)}$. In the context of financial mathematics, it is 
natural to assume that the returns in the exponential L\'evy model have finite 
variance, i.e. $\E e^{2X_t}<\infty$. This is equivalent to 
$\int_{[1,\infty)}e^{2x}\nu(dx)<\infty$~\cite[Thm~25.3]{MR3185174}, implying 
for example $q=2$ (for $\lambda=1$ and $p=1$) with the bound $\O(2^{-n/2})$.
The proof of Proposition~\ref{prop:LocLip} is in Subsection~\ref{proof:LocLip}. 
A numerical example is in Subsection~\ref{subsec:ComparisonSBA_RWA}.

Barrier-type functions of $\chi$, which are discontinuous in the trajectory of the 
L\'evy process, arise in the pricing of contingent convertibles~\cite{KyprianouCoCos}, 
the evaluation of ruin probabilities~\cite{MR2099651} and as payoffs of barrier 
options~\cite{MR1482707,MR1805321,MR2023872}. 
By Theorem~\ref{thm:error}, 
the error $\Delta_n^\SB$ in~\eqref{def:Delta} of the second 
coordinate $\ov X_T-\Delta_n^\SB$ of the SBA $\chi_n$ satisfies 
$0\leq\Delta_n^\SB\searrow 0$ a.s. as $n\to\infty$. Hence, the limit 
$\P(\ov X_T-\Delta_n^\SB\leq x)\searrow\P(\ov{X}_T\leq x)$ as $n\to\infty$ holds for 
any fixed $x>0$. The rate of convergence in this limit is both crucial for the control 
of the bias of barrier-type functions and intimately linked to the quality of the 
right-continuity of  the distribution function $x\mapsto\P(\ov{X}_T\leq x)$ of 
$\ov X_T$. We will thus need the following assumption.

\begin{asm}\label{asm:H} 
Given $M,K,\gamma>0$, the inequality $\P(\ov X_T\leq M+x)-\P(\ov X_T\leq M)
\leq Kx^\gamma$ holds for all $x\geq0$. 
\end{asm}

\begin{prop}\label{prop:barrier} 
Define $g(\chi)=h(X_T)\1{\ov{X}_T\leq M}$, where $h:\R\to\R$ is bounded and 
measurable and $M>0$. Let Assumption~\ref{asm:H} hold for $M$ and 
some $K,\gamma>0$. Fix any $p,q\geq1$ and let $\eta_q\in[3/2,2]$ be as 
in~\eqref{eq:eta}. Then, under the assumptions of Theorem~\ref{thm:error}, we have
\[\E[|g(\chi)-g(\chi_n)|^p] =\O\big(\eta_q^{-n\gamma/(\gamma+q)}\big),
\quad\text{as }n\to\infty.
\]
Moreover, the constant in $\O\big(\eta_q^{-n\gamma/(\gamma+q)}\big)$, given in 
Equation~\eqref{eq:barrierBound} below, is explicit in $K$, $\gamma$, $p$, $q$, 
$\|h\|_\infty$ and the characteristics $(\sigma^2,\nu,b)$ of the L\'evy process $X$.
\end{prop}

The proof of Proposition~\ref{prop:barrier} is in Subsection~\ref{proof:barrier} 
below. Minimising $\eta_q^{-\gamma/(\gamma+q)}$ as a function of $q$ is not 
trivial (see Remark~\ref{rem:optimal_q} below for the optimal choice of $q$). 
In the special case when 
$\gamma=1$ (i.e. the distribution function of $\ov X_T$ is Lipschitz from the 
right at $M$) we have: (a) if $X$ has paths of finite variation, then $\eta_1=2$ 
and the optimal choice $q=1$ gives the bound $\O\big(2^{-n/2}\big)$; (b)  if 
$\sigma\neq0$, then the optimal choice $q=2$ yields the bound $\O(2^{-n/3})$.

The rate of decay in Proposition~\ref{prop:barrier} is essentially 
controlled by the rate of convergence in the Kolmogorov distance of 
$\ov X_T-\Delta_n^\SB$ to $\ov X_T$. 
In general, as mentioned above,
$\ov X_T-\Delta_n^\SB$ is known to converge to $\ov X_T$ weakly.
As the Kolmogorov distance does not 
metrise the topology of weak convergence (cf.~\cite[Ex.~1.8.32,~p.43]{MR1353441}), 
we require an additional assumption, such as~\ref{asm:H}, to 
obtain a rate in Proposition~\ref{prop:barrier}. 

Assumption~\ref{asm:H} holds for a wide class of L\'evy processes. 
By the Lebesgue differentiation 
theorem~\cite[Thm~6.3.3]{MR3098996}, the function 
$x\mapsto\P(\ov X_T\leq x)$ is differentiable a.e. and 
Assumption~\ref{asm:H} holds 
for $\gamma=1$ and Lebesgue almost every $M$. 
If the density of $\ov X_T$ exists and is 
bounded around $M$, then $x\mapsto\P(\ov X_T\leq x)$ is locally Lipschitz at $M$, 
again satisfying Assumption~\ref{asm:H} with $\gamma=1$. 
This is the case if the density of 
$\ov X_T$ is continuous at $M$, which holds  for stable processes or if 
$\sigma\neq0$~\cite{MR3531705}, and, more generally, if $X$ converges weakly 
under the zooming-in procedure and $\alpha>1$ in~\eqref{eq:alpha}, 
see~\cite[Lem.~5.7]{ZoomIn}. Moreover, 
by~\cite[Prop.~2]{MR3531705} and~\cite[Sec.~VI.4, Thm~19]{MR1406564}, the density of $\ov{X}_T$ is continuous at 
$M$ if the ascending ladder height process of $X$ has 
positive drift (e.g. if $X$ is spectrally negative of infinite variation) 
or if $X$ is in a certain class 
of subordinated Brownian motions~\cite[Prop.~4.5]{MR3098066}.  However, the 
continuity of the density of $\ov X_T$ is known to fail if $X$ is of bounded variation with no 
negative jumps and has a L\'evy measure with atoms~\cite[Lem.~2.4]{MR3014147}. 
Furthermore, for any $\gamma\in(0,1)$, the function 
$x\mapsto\P(\ov X_T\leq x)$ may be continuous at $M$ but not 
locally $\gamma$-H\"older continuous (see example in 
Appendix~\ref{sec:Reg} below) even if the L\'evy measure has no 
atoms, demonstrating again the necessity of an condition such 
as Assumption~\ref{asm:H} in Proposition~\ref{prop:barrier}.

We stress that, even if the density is locally bounded at $M$, it appears to be 
very difficult to give bounds (based on the L\'evy characteristics) on the value 
it takes at $M$. This means that, unlike in the case of a (locally)-Lipschitz 
function $g(\chi)$, in the context of barrier options we cannot provide 
non-asymptotic confidence intervals based on Proposition~\ref{prop:barrier}, 
cf. Subsection~\ref{subsec:CLT} below.

\subsubsection{Comparison\label{subsec:Lit_functionals}}
The results in~\cite{MR2802466,MR2759203,MR2867949,MR3138603,ZoomIn}, 
discussed in Subsection~\ref{subsec:Lit_Lp} above, yield bounds in $L^p$ on the error of a 
Lipschitz function of $(X_T,\ov X_T)$. The orders of decay are the same as those 
reported in Subsection~\ref{subsec:Lit_Lp} above for the respective approximations. The error 
of the time of the supremum $\tau_T$, geometrically convergent for the SBA by 
Theorem~\ref{thm:Lp}(a) and Proposition~\ref{prop:Lipschitz}, appears not to have been 
studied for the other algorithms.

In the case of locally Lipschitz functions, only the decay of the error in 
$L^1$ for the RWA seems to have been analysed. Define for any $q>0$ the integral
\begin{equation}\label{def:E_p}
E_+^q=\int_{[1,\infty)}e^{qx}\nu(dx).
\end{equation}
If $X$ has finite activity (i.e. $\nu(\R)<\infty$), then the bias equals 
$\O(n^{-1/2})$ if $\sigma\neq0$ and $E_+^q<\infty$ for some 
$q>2$~\cite[Prop.~5.1]{MR2867949} and $o(n^{-(q-1)/q})$ if $\sigma=0$ 
and $E_+^q<\infty$ for some $q>1$~\cite[Rem.~5.3]{MR2867949}. 
In the case $\sigma=0$ and $\nu(\R)=\infty$, for any $q>1$ satisfying 
$E_+^q<\infty$ and any arbitrarily small $\delta>0$, the bias decays as 
follows: $\O((n/\log(n))^{\delta-(q-1)/q})$ if the process is of finite variation 
(i.e. $\int_{(-1,1)}|x|\nu(dx)<\infty$), $\O(n^{\delta-(q-1)/q})$ if 
$\int_{(-1,1)}|x|\log|x|\nu(dx)<\infty$ and $\O(n^{\delta-(q-1)/(2q)})$ 
otherwise~\cite[Thm~6.2]{MR2867949}. If the L\'evy process $X$ is 
spectrally negative with jumps of finite variation (i.e. $\nu((0,\infty))=0$ and 
$\int_{(-1,0)}|x|\nu(dx)<\infty$) and if $E_+^q<\infty$ for some $q>1$, 
the error decays as $\O(n^{-1})$ (resp. $\O(n^{-1/2}\log(n))$) 
if $\sigma=0$ (resp. $\sigma\neq0$)~\cite[Prop.~6.4]{MR2867949}.

The discontinuous payoffs under variance gamma (VG), normal inverse 
Gaussian (NIG) and spectrally negative $\alpha$-stable (with $\alpha>1$) 
processes are considered in~\cite{MR3723380}. Under the assumption that 
the density of the supremum is bounded around the barrier level in all three 
models, the errors in $L^p$ of the RWA decay as $\O(n^{\delta-1})$, 
$\O(n^{\delta-1/2})$ and $\O(n^{\delta-1/\alpha})$ for arbitrarily small 
$\delta>0$, respectively~\cite[Prop.~5.5]{MR3723380}. In the case 
$\nu(\R)<\infty$ and $\sigma\neq0$, the error decays as $\O(1/\sqrt{n})$, 
see~\cite[Prop.~2.2 \& Rem.~2.3]{MR2851060}. This result was first 
established in~\cite{MR1482707} for the Brownian motion with drift. 

As noted in~\cite[Sec.~5.3]{ZoomIn}, if $X$ has a jointly continuous density 
$(t,x)\mapsto\frac{\partial}{\partial x}\P(X_t\leq x)$ bounded for $(t,x)$ away 
from the origin $(0,0)$ (e.g. if Orey's condition holds for 
$\gamma>1$~\cite[Prop.~28.3]{MR3185174} or $\sigma>0$, see also the 
paragraphs following Proposition~\ref{prop:barrier}), $\nu(\R)=\infty$ 
and $\alpha\geq 1$ (defined in~\eqref{eq:alpha}), then the error in $L^p$ of 
the RWA for a barrier option decays as $\O(n^{\delta-1/\alpha})$ for any small 
$\delta>0$. Moreover, by~\cite[Lem.~5.8]{ZoomIn}, 
$\liminf_{n\to\infty}n\P(\ov X_T>x\geq \max_{k\in\{1,\ldots,n\}} X_{kT/n})>0$ 
if $X$ has jumps of both signs. 
Put differently, the error in $L^p$ of the RWA for a general barrier option 
cannot be of order $o(n^{-1})$. As far as we know, such results for the 
WHA~\cite{MR2895413} are currently unavailable. 

\subsection{The central limit theorem (CLT) and the confidence intervals (CIs)\label{subsec:CLT}}
Let $(\chi^i_n)_{i\in\{1,\ldots,N\}}$ be the output produced by $N\in\N$ 
independent runs of~\nameref{alg:SBA} using $n$ steps. 
The \textit{Monte Carlo estimator} $\sum_{i=1}^N g(\chi^i_n)/N$ of $\E g(\chi)$, 
where $g:\R\times\R_+\times[0,T]\to\R$ is a measurable function of interest in 
applied probability (e.g. in one of the classes from 
Subsection~\ref{subsec:Functionals} above), has an error
\begin{equation}\label{eq:MC_error}
\Delta_{n,N}^g = \frac{1}{N}\sum_{i=1}^N g(\chi_n^i)-\E g(\chi).
\end{equation}
Our aim is to understand the rate of convergence of the error in~\eqref{eq:MC_error}
as the number of samples $N$ tends to infinity.

\begin{thm}[CLT]\label{thm:CLT} 
Assume $\P[\chi\in D_g]=0$, where $D_g$ is the discontinuity set of $g$, and
\begin{itemize}
\item[\normalfont(a)] 
there is a measurable function $G:\R\times\R_+\times[0,T]\to\R_+$ such that: 
\begin{itemize}
\item[\normalfont(i)] 
$|g(x,y,t)|\leq G(x,y,t)$ for all $(x,y,t)\in\R\times\R_+\times[0,T]$,
\item[\normalfont(ii)] 
for all $x\in\R$, $(y,t)\mapsto G(x,y,t)$ is nondecreasing in both coordinates,
\item[\normalfont(iii)] 
$\E [G(X_T,\ov{X}_T,T)^2]<\infty$,
\end{itemize}
\item[\normalfont(b)] 
$\E g(\chi)=\E g(\chi_n) + o(\eta_g^{-n})$ for some $\eta_g>1$.
\end{itemize}
Denote $\V [g(\chi)]=\E[(g(\chi)-\E[g(\chi)])^2]$ and set
$n_N=\lceil \log N/ \log(\eta_g^{2})\rceil$ for every $N\in\N$, where we denote 
$\lceil x\rceil =\inf\{n\in\N:n\geq x\}$ for $x\in\R$. Then the following weak 
convergence holds
\begin{equation}\label{eq:CLT}
\sqrt{N}\Delta_{n_N,N}^g \overset{d}{\to}N(0,\V [g(\chi)]), \qquad\text{as $N\to\infty$. }
\end{equation}
\end{thm}

Theorem~\ref{thm:CLT} is not an iid CLT since the bias of the MC estimator
forces the increase in the 
number of steps taken by~\nameref{alg:SBA} as the number of samples $N\to\infty$. 
 Its proof (see Subsection~\ref{proof:CLT} below) establishes Lindeberg's 
condition and then applies the CLT for triangular arrays. 
The condition $\P [\chi\in D_g]=0$ is satisfied 
if e.g. the Lebesgue measure of $D_g$ is zero and $0$ is regular for $X$ for both half-lines~\cite[Thm~3]{MR3098676}. This assumption is important as 
it allows us to construct asymptotic confidence intervals for barrier options using 
the limit in~\eqref{eq:CLT}. 
Assumption (a) ensures the convergence of $\V[g(\chi_n)]$ to $\V[g(\chi)]$ 
and might seem restrictive at first sight. However, the function $G$ is very easy to 
identify (see Remark~\ref{rem:CLT-G} below) in the contexts of 
Propositions~\ref{prop:Lipschitz},~\ref{prop:LocLip} and~\ref{prop:barrier}, 
where Assumption~(b) also clearly holds. 

Since $|\Delta_{n,N}^g|\leq |\E g(\chi)-\E g(\chi_n)| 
+ |\Delta_{n,N}^g-\E \Delta_{n,N}^g|$, we may construct a confidence interval 
for the MC estimator $\sum_{i=1}^N g(\chi^i_n)/N$
at level $1-\epsilon\in(0,1)$ using the implication:
\begin{equation}\label{eq:CIs_rule}
\begin{rcases}
|\E g(\chi)-\E g(\chi_n)|<r_1, \\
\P (|\Delta_{n,N}^g-\E \Delta_{n,N}^g|< r_2) \geq 1-\epsilon,
\end{rcases}\implies \P (|\Delta_{n,N}^g|<r_1+r_2)\geq 1-\epsilon.
\end{equation}
In~\eqref{eq:CIs_rule}, $r_1$ may be chosen as a function of the number $n$ 
of steps in~\nameref{alg:SBA} in various ways depending on the properties 
of $g$ (see Propositions~\ref{prop:Lipschitz} and~\ref{prop:LocLip} of 
Subsection~\ref{subsec:Functionals}). Note that this requires the explicit 
dependence of the constant on the model characteristics. 

Having fixed $n$, pick $r_2$ in~\eqref{eq:CIs_rule} as a function of 
$\epsilon$ either via concentration inequalities (not relying on 
Theorem~\ref{thm:CLT}) or the CLT in Theorem~\ref{thm:CLT}: \\
\noindent (i) Non-asymptotic CI: 
by Chebyshev's inequality 
$\P\big(|\Delta_{n,N}^g-\E \Delta_{n,N}^g|>r\big)\leq\V[g(\chi_n)]/(r^2N)$, 
we only need to bound the variance $\V[g(\chi_n)]$ (e.g. by the function $G$ in 
Remark~\ref{rem:CLT-G}).  See e.g.~\cite[Thm~1]{ConfIntBoundedRV} for a 
sharper choice of $r_2$.\\ 
\noindent (ii) Asymptotic CI: since $\Delta_{n,N}^g-\E \Delta_{n,N}^g=
(1/N)\sum_{i=1}^N g(\chi^i_n)-\E g(\chi_n)$, we may use the CLT for fixed $n$ 
in Remark~\ref{rem:CLT_fixed_n} below (as in~(i) we bound $\V[g(\chi_n)]$ by 
elementary methods).

In the case we do not have access to the constants in the bound on the bias 
in~\eqref{eq:CIs_rule} in terms of the model parameters (e.g. barrier options 
in Proposition~\ref{prop:barrier}), we apply the CLT result in
Theorem~\ref{thm:CLT} to the estimator $\Delta_{n_N,N}^g$ directly, to obtain 
an asymptotic CI. See Subsection~\ref{subsec:CIs_numerics} below for the 
numerical examples of asymptotic and non-asymptotic CIs. 

\subsection{Computational complexity of~\nameref{alg:SBA} and the multilevel Monte Carlo}
\label{subsec:Complexity_SBA}

Assume that the expected computational cost of drawing a sample from the distribution 
$F(t,\cdot)$ 
in~\nameref{alg:SBA}
is bounded above by a constant that does not depend on 
$t\in[0,T]$. 
Then the expected computational cost of 
a single draw from the law of $\chi_n$ via~\nameref{alg:SBA} is bounded by $\O(n)$. 
The CLT in Theorem~\ref{thm:CLT} (applicable to (locally) Lipschitz and barrier-type functions, cf. Subsection~\ref{subsec:CLT} above) implies 
that the $L^2$-norm of the error in~\eqref{eq:MC_error} of the MC estimator 
can be made smaller than $\epsilon$,  i.e. 
$\E[(\Delta_{n,N}^g)^2]\leq\epsilon^2$, at a computational cost of 
$\O(\epsilon^{-2}\log\epsilon)$ as $\epsilon\to0$. 
The cost of the Monte Carlo estimator based on~\nameref{alg:SBA}
is thus only a log-factor away from the optimal Monte Carlo cost of 
$\O(\epsilon^{-2})$, arising when one has access to exact 
simulation with finite expected running time.

The main aim of MLMC, introduced in~\cite{Heinrich_MLMC, MR2436856}, is to reduce the computational cost of an 
MC algorithm for a given level of accuracy. 
We will apply a general MLMC result~\cite[Thm~1]{MR2835612}, stated in our 
setting for ease of reference as Theorem~\ref{thm:MLMC} in 
Appendix~\ref{sec:MLMC_1} below. Let $P=g(\chi)$ and $P_n=g(\chi_n)$, $n\in\N$, 
for any function $g$ that satisfies the assumptions of Theorem~\ref{thm:CLT} 
(see also Remark~\ref{rem:CLT-G} below). Note that the expected computational cost 
of a single draw in Theorem~\ref{thm:MLMC} is allowed to grow geometrically in
$n$. Since in the context of the present section sampling $P_n$ has a cost of 
$\O(n)$, we may choose an arbitrarily small rate $q_3>0$ in Theorem~\ref{thm:MLMC}. 

A key component of any MLMC scheme is the coupling $(P_n,P_{n+1})$. 
In the case of~\nameref{alg:SBA} (and the notation therein), this consists of using 
the same sequence of sticks $(\lambda_k)_{k\in\{1,\ldots,n\}}$ and increments 
$(\xi_k)_{k\in\{1,\ldots,n\}}$
in the consecutive  levels and setting $\varsigma_{n}=\xi_{n+1}+\varsigma_{n+1}$, cf. 
the coupling of Subsection~\ref{subsec:coupling}. Since 
\begin{equation}\label{eq:Var_L2}
\V[P_{n+1}-P_n]\leq\E[(P_{n+1}-P_n)^2]\leq 2(\E[(P_{n+1}-P)^2]+\E[(P-P_n)^2]),
\end{equation}
Assumption~(b) in Theorem~\ref{thm:MLMC} follows easily from the bound 
$\E[(P-P_n)^2]=\O(2^{-nq_2})$ for all functions $g$ of interest 
(see Propositions~\ref{prop:Lipschitz},~\ref{prop:LocLip} 
and~\ref{prop:barrier} above for the corresponding $q_2>0$). 
These observations imply that the computational complexity of 
the MLMC estimator in~\eqref{eq:MLMC_estim}
is bounded above by 
$\O(\epsilon^{-2})$
(take $q_3=q_2/2$ for all choices of $g$ in the
propositions above).
The implementation 
of the MLMC estimator based on~\nameref{alg:SBA} 
for a barrier-type function $g$ under the NIG model 
numerically confirms this bound, see 
Subsection~\ref{subsec:MLMC_NIG_barrier_ex} below. 

\subsubsection{Comparison\label{subsec:Lit_complexity}}
The computational complexity of 
MC and  MLMC procedures based on 
the~\nameref{alg:SBA} is given by  
$\O(\epsilon^{-2}|\log\epsilon|)$ and $\O(\epsilon^{-2})$, respectively, 
for a function $g(\chi)$, which is Lipschitz, locally Lipschitz or barrier-type.
This makes~\nameref{alg:SBA} robust, as its performance does not depend on the 
structure of the problem. In particular, minor changes in model parameters 
will not result in major differences in the computational complexity. 
We compare this to the extant MC and MLMC algorithms in the literature.

\underline{Lipschitz function $g$.} We first review the results for Lipschitz functions 
of $(X_T,\ov X_T)$. 
For the RWA, $\alpha$ as in~\eqref{eq:alpha} below and a small $\delta>0$ 
($\delta=0$ if $\alpha\in\{1,2\}$),~\cite[Thm~4.1]{ZoomIn} implies that the 
cost of an MC estimator is $\O(\epsilon^{-2-\max\{1,\alpha+\delta\}})$. In 
particular, if $\sigma\neq0$, the complexity of the RWA is $\O(\epsilon^{-4})$ 
(see also~\cite{MR2867949,MR2996014,MR3723380}). Their MLMC counterparts, 
derived following the procedure of~\cite{MR3723380}, together with the bounds 
in~\cite[Thm~4.1]{ZoomIn}
and~\eqref{eq:Var_L2},
have a complexity of 
$\O(\epsilon^{-2}\log^2(\epsilon))$. Moreover, if the process is spectrally negative 
without a Brownian component and either an infinite variation stable 
process~\cite[Prop.~5.5]{MR3723380} or of finite
variation~\cite[Lem.~6.5]{MR2867949}, then the MLMC estimator 
for a Lipschitz function of  $(X_T,\ov X_T)$
has optimal cost 
$\O(\epsilon^{-2})$. 
For the WHA 
(see Subsection~\ref{subsec:Related_Lit} above),
the MC (resp. MLMC) estimator 
for a Lipschitz function of $(X_T,\ov X_T)$
has a complexity of 
$\O(\epsilon^{-4})$ (resp. $\O(\epsilon^{-3})$) if the process is of finite 
variation and of $\O(\epsilon^{-6})$ (resp. $\O(\epsilon^{-4})$) 
otherwise~\cite[Thm~4.6]{MR3138603}.
For the JAGA, the complexity of the MC estimator is 
$\O(\epsilon^{-2}\max\{\epsilon^{-\max\{1,\beta_+\}},
	\epsilon^{-4\beta_+/(4-\beta_+)}\log(1/\epsilon)^{2\beta_+/(4-\beta_+)}\})$ 
	if $\sigma=0$ and  
$\O(\epsilon^{-2-\max\{2,4\beta_+/(4-\beta_+)\}})$ otherwise 
(see~\eqref{eq:BG+} for the definition of $\beta_+\in(0,2]$). 
The complexity of the MLMC estimator is 
$\O(\epsilon^{-2}\log(1/\epsilon)^{3\cdot\1{\sigma\ne 0}})$ if $\beta_+< 1$,
$\O(\epsilon^{-2}\log(1/\epsilon)^{2+\1{\sigma\ne 0}})$ if $\beta_+=1$, 
$\O(\epsilon^{-2-4(1-1/\beta_+)}\log(1/\epsilon)^{2-2/\beta_+})$ 
if $\beta_+\in(1,4/3]$ and $\sigma\ne0$, and 
$\O(\epsilon^{-2-8(\beta_+-1)/(4-\beta_+)}\log(1/\epsilon)^{4(\beta_+-1)/(4-\beta_+)})$ 
otherwise. In the worst case $\beta_+=2$, the MLMC estimator based on the JAGA 
has a complexity of $\O(\epsilon^{-6})$. 

\underline{Locally Lipschitz function $g$.}
In the case of locally Lipschitz functions, only the MC analysis 
of the RWA appears to be available in the literature. The error in this case is at 
best $\O(\epsilon^{-3})$, attained only when the L\'evy process is spectrally 
negative, with jumps of finite variation and no Brownian component (i.e.
$\nu(\R_+)=0$, $\int_{(-1,0)}|x|\nu(dx)<\infty$ and $\sigma=0$) and 
the inequality $E_+^q<\infty$ holds for some $q>1$~\cite[Prop.~6.4]{MR2867949} (recall the
definition of $E_+^q$ in~\eqref{def:E_p} above). If $X$ has a Brownian component 
(i.e. $\sigma\neq0$), then the cost is either $\O(\epsilon^{-4})$ if $\nu(\R)<\infty$ and 
$E_+^q<\infty$ for some $q>2$~\cite[Prop.~6.4]{MR2867949} or 
$\O(\epsilon^{-4}\log^2(\epsilon))$ if $X$ is spectrally negative with jumps of finite 
variation and $E_+^q<\infty$ for some $q>1$~\cite[Prop.~5.1]{MR2867949}. If 
$\sigma=0$ and $X$ has infinite activity, then for any arbitrarily small $\delta>0$, 
the condition $E_+^q<\infty$ (for some $q>1$) implies an MC complexity of 
$\O(\epsilon^{-2-2q/(q-1)-\delta})$. In the last case, the decay may be improved to 
$\O(\epsilon^{-2-q/(q-1)-\delta}|\log(\epsilon)|)$ (resp. 
$\O(\epsilon^{-2-q/(q-1)-\delta})$) if $\int_{(-1,1)}|x|\nu(dx)<\infty$ (resp. 
$\int_{(-1,1)}|x|\log|x|\nu(dx)<\infty$)~\cite[Thm~6.2]{MR2867949}.

\underline{Barrier-type function $g$.}
To the best of our knowledge, there are no non-parametric MLMC results in the 
literature for barrier options under the RWA. Recently the MLMC for the RWA 
under VG, NIG and spectrally negative $\alpha$-stable (with $\alpha>1$) 
processes has been shown in~\cite{MR3723380} to have the computational cost of 
$\O(\epsilon^{-2-\delta})$, $\O(\epsilon^{-3-\delta})$ and 
$\O(\epsilon^{-1-\alpha-\delta})$ for small $\delta>0$, respectively. 
We are not aware of any results for WHA, introduced in~\cite{MR2895413}, for 
barrier options.

\subsection{Unbiased estimators\label{subsec:UnbiasedEstimators}}
Randomising the number of levels and samples at each level in the MLMC estimator 
from the previous section yields an unbiased estimator~\eqref{eq:Matti_estimator} 
below, see e.g.~\cite{MR3422533,MR3782809}. There are numerous ways of 
implementing such a debiasing technique, typically based on a random variable $R$ 
on the integers satisfying $\P[R=n]>0$ for all $n\in\N$, with the tail of the law of $R$ in some 
way linked to the asymptotic decay of the level variances in the MLMC. While other 
estimators from~\cite{MR3782809} could be considered, here we focus on the 
\emph{single term estimator} (STE) and the \emph{independent sum estimator} (ISE). 
For these two estimators, a sequence 
$(R_j)_{j\in\{1,\ldots,N\}}$ of independent random variables specifies 
 the number of samples $N_k$ at level $k\in\N$ as follows: 
$N_k=\sum_{j=1}^N \1{R_j=k}$ for STE and $N_k=\sum_{j=1}^N \1{R_j\geq k}$ for ISE. 
For both estimators, 
we use the uniform stratified sampling of the sequence
$(R_i)_{i\in\{1,\ldots,N\}}$:
each $R_j$  is drawn independently
and distributed as $R$ conditioned to be between its $(j-1)/N$ and $j/N$ quantiles.

The probabilities $(\P[R=n])_{n\in\N}$
that
maximise the asymptotic inverse relative efficiencies (see Appendix~\ref{sec:MLMC_2} below for definition)
for the STE and ISE, denoted by
by $(p^\STE_n)_{n\in\N}$ and $(p^\ISE_n)_{n\in\N}$, respectively, are in general given by the formulae in~\eqref{eq:optimal_prob_unbiased}. 
In the case of the unbiased estimator for $\E P$, where $P=g(\chi)$, the optimal probabilities take the form:
\begin{itemize}
\item (Lipschitz) If $g$ is as in Proposition~\ref{prop:Lipschitz}, we set
\[p_n^\STE=\frac{2^{-n/2}/\sqrt{n}}{\sum_{k=1}^\infty2^{-k/2}/\sqrt{k}},
\qquad p^\ISE_n=\frac{2^{-(n-1)/2}}{\sqrt{n}}-\frac{2^{-n/2}}{\sqrt{n+1}}.\]
\item (Locally Lipschitz) If $g$, $q$ and $q'=(1-1/q)^{-1}$ are as in 
Proposition~\ref{prop:LocLip}, we set
\[p_n^\STE=\frac{2^{-n/(2q')}/\sqrt{n}}
{\sum_{k=1}^\infty2^{-k/(2q')}/\sqrt{k}},
\qquad p^\ISE_n=\frac{2^{-(n-1)/(2q')}}{\sqrt{n}}-
\frac{2^{-n/(2q')}}{\sqrt{n+1}}.\]
\item (Barrier-type) If $g$, $\gamma$ and $q$ are as in 
Proposition~\ref{prop:barrier}, we set
\[p_n^\STE=\frac{\eta_q^{-n\gamma/(2\gamma+2q)}/\sqrt{n}}
{\sum_{k=1}^\infty\eta_q^{-k\gamma/(2\gamma+2q)}/\sqrt{k}},
\qquad p^\ISE_n=\frac{\eta_q^{-(n-1)\gamma/(2\gamma+2q)}}{\sqrt{n}}-
\frac{\eta_q^{-n\gamma/(2\gamma+2q)}}{\sqrt{n+1}}.\]
\end{itemize}

It is interesting to note that the choices in the Lipschitz (resp. locally 
Lipschitz) case is independent of the structure of the L\'evy process $X$
(resp. dependent only through its exponential moments). This invariance 
reinforces the idea that~\nameref{alg:SBA} is robust.
It is a consequence of the fact that $\eta_p$ (defined in~\eqref{eq:eta}) equals $2$ for 
$p\geq 2$. 

\section{\label{sec:Numerical-Examples}Numerical examples}

The implementation of~\nameref{alg:SBA} above can be found in the 
repository~\cite{Jorge_GitHub} together with a simple algorithm for the 
simulation of the increments of the VG, NIG and weakly stable processes.
This implementation was used in Sections~\ref{subsec:ComparisonSBA_RWA} below.

\subsection{Numerical comparison: SBA and RWA}
\label{subsec:ComparisonSBA_RWA}
Let $X=(X_t)_{t\geq0}$ be given by $X_t= B_{Z_t}+bt$, where $Z$ is a 
subordinator with L\'evy measure $\nu_Z(dx)=\1{x>0}\gamma 
x^{-\alpha-1}e^{-\lambda x}dx$ 
($\alpha\in[0,1)$, $\gamma,\lambda>0$) and drift 
$\sigma_Z\geq0$, $B$ is a standard Brownian motion and $b\in\R$. 
The L\'evy measure of $X$ by~\cite[Thm~30.1]{MR3185174} equals 
$\nu(dx)/dx=\frac{\gamma}{\sqrt{2\pi}} |x|^{-2\alpha-1} 
	\int_0^\infty s^{-\alpha-3/2}e^{-\lambda sx^2-s^{-1}/2}ds$, 
implying that the Blumenthal-Getoor index of $X$ is 
$\beta=2\alpha\in[0,2)$, and its Brownian component equals 
$\sigma^2=\sigma_Z^2$. Moreover, the increment $X_t$ can be 
simulated in constant expected computational time for any $t>0$.

We consider the estimator $\sum_{i=1}^N g(\chi^i_n)/N$, where 
$(\chi^i_n)_{i\in\{1,\ldots,N\}}$ are $N$ iid samples produced by 
running the~\nameref{alg:SBA} over $n$ steps. We compare the 
results with the output of the RWA in~\eqref{eq:RWA}, based on a 
time step of size $T/2^n$ and the same number $N$ of iid samples. 
The function $g(\chi)$ corresponds to either a lookback put or an 
up-and-out call under the exponential L\'evy model $S=S_0\exp(X)$. 
Figure~\ref{fig:lookback_TSBM} shows that the accuracy of the 
two algorithms is comparable as suggested by 
Propositions~\ref{prop:LocLip} and~\ref{prop:barrier} above 
(note $E_\pm^q<\infty$ if and only if $q^2<2\lambda$, 
since $\E\big[e^{qX_t}\big]=e^{bt}\E\big[e^{q^2Z_t/2}\big]$).

\begin{figure}[ht]
	\begin{tikzpicture} 
			\begin{axis} 
			[
			title={Lookback put: $g(\chi)=\ov{S}_T-S_T$},
			ymin=1.35,
			ymax=1.66,
			xmin=4.95,
			xmax=20.05,
			xlabel={$n$},
			width=7.85cm,
			height=4.5cm,
			axis on top=true,
			axis x line=bottom, 
			axis y line=middle,
			legend style={at={(1,.025)},anchor=south east}]			
			\addplot[
			densely dotted,
			color=black,
			style = thick,
			]
			coordinates {(5,1.4230)(6,1.4814)(7,1.5284)(8,1.5629)(9,1.5901)(10,1.6090)(11,1.6186)(12,1.6253)(13,1.6340)(14,1.6377)(15,1.6395)
			};
			\addplot[
			dashed,
			color=black,
			style = thick,
			]
			coordinates {(5,1.4489)(6,1.5153)(7,1.5609)(8,1.5889)(9,1.6097)(10,1.6245)(11,1.6335)(12,1.6359)(13,1.6397)(14,1.6437)(15,1.6471)(16,1.6458)(17,1.6472)(18,1.6485)(19,1.6462)(20,1.6488)
			};
			\addplot[
			solid,
			color=black,
			]
			coordinates {(5,1.6480829339511918)(20,1.6480829339511918)
			};
			\legend {RWA with time step $T/2^n$, 
				SBA after $n$ steps, 
				$\E g(\chi)$};
			\end{axis}
	\end{tikzpicture}
	\begin{tikzpicture} 
			\begin{axis} 
			[
			title={Up-and-out call: $g(\chi)= (S_T-K_0)^+\cdot
				\1{\ov{S}_T\leq M}$},
			ymin=.407,
			ymax=.44,
			xmin=4.95,
			xmax=20.05,
			xlabel={$n$},
			width=7.85cm,
			height=4.5cm,
			axis on top=true,
			axis x line=bottom, 
			axis y line=middle,
			legend style={at={(1,.4)},anchor=south east}]
			\addplot[
			densely dotted,
			color=black,
			style = thick,
			]
			coordinates {(5,0.4325)(6,0.4231)(7,0.4200)(8,0.4140)(9,0.4151)(10,0.4118)(11,0.4134)(12,0.4097)(13,0.4122)(14,0.4114)(15,0.4107)
			};
			\addplot[
			dashed,
			color=black,
			style = thick,
			]
			coordinates {(5,0.4187)(6,0.4159)(7,0.4146)(8,0.4149)(9,0.4121)(10,0.4110)(11,0.4100)(12,0.4119)(13,0.4088)(14,0.4122)(15,0.4118)(16,0.4110)(17,0.4106)(18,0.4104)(19,0.4109)(20,0.4108)
			};
			\addplot[
			solid,
			color=black,
			]
			coordinates {(5,0.4108943884278457)(20,0.4108943884278457)
			};
			\legend {RWA with time step $T/2^n$, 
				SBA after $n$ steps, 
				$\E g(\chi)$};
			\end{axis}
	\end{tikzpicture}
\caption{\label{fig:lookback_TSBM}
We take $\alpha=0.75$, $\gamma=0.1$, $\lambda=4$, $\sigma_Z=0.05$, 
$b=-0.05$ and $S_0=2$, $K_0=3$, $M=5$, $T=1$ and $N=10^7$. The value 
$\E g(\chi)$ is obtained by running~\nameref{alg:SBA} for 
$n=100$ steps and using $N=10^8$ samples. The RWA is approximately 
$(2^n/n)$-times slower than the SBA for the same amount of bias, 
making it infeasible for  $n>15$ as at least $60 000<2^n$ steps are 
needed in the time interval $[0,1]$.}
\end{figure}

\subsection{Asymptotic and non-asymptotic CIs\label{subsec:CIs_numerics}}

Let $X$ be a Normal Inverse Gaussian process (NIG) with parameters
$(b,\kappa,\sigma,\theta)$, i.e. a L\'evy process with characteristic
function $\E\big[e^{iuX_t}\big] =\exp(t(b+1/\kappa)-(t/\kappa)
\sqrt{1-2iu\theta\kappa+\kappa\sigma^2u^2})$, whose L\'evy measure is given by 
\[\frac{\nu(dx)}{dx} =\frac{C}{|x|}e^{Ax}K_1(B|x|),\quad 
\text{with}\quad A =\frac{\theta}{\sigma^{2}}, \quad 
B =\frac{\sqrt{\theta^2+\sigma^{2}/\kappa}}{\sigma^{2}}, \quad 
C =\frac{\sqrt{\theta^2 + 2\sigma^{2}/\kappa}}{2\pi\sigma\kappa^{3/2}},\]
where $K_1$ is the modified Bessel function of the second kind, which satisfies
\[K_1(x)=\frac{1}{x}+\O(1),\text{ as }x\to0,\quad K_1(x)=e^{-x}\sqrt{\frac{\pi}{2|x|}}( 1+\O(1/|x|) ),\text{ as }x\to\infty.\]
We simulate the increments of the NIG process by~\cite[Alg.~6.12]{MR2042661}.
Figure~\ref{fig:NIG-ConfInt} presents confidence intervals at level 
$1-\epsilon=99\%$ for the prices of hindsight put and barrier up-and-out call
under the NIG model $S=S_0\exp(X)$.

The non-asymptotic CI for the hindsight put is constructed via Chebyshev's 
inequality as discussed in Subsection~\ref{subsec:CLT} above. In particular, 
note that the payoff of the hindsight put $g:(x,y,t)\mapsto(K_0-S_0e^y)^+$ is 
non-increasing in $y$ and does not depend on $x$ and $t$. Since $\ov{X}_T$  
dominates the second coordinate $\ov X_T-\Delta_n^\SB$ of the SBA $\chi_n$
in~\eqref{eq:SBA}, we apply $\E g(\chi_n)\geq \E g(\chi)$ and find
\begin{equation*}
\begin{rcases}
0\leq \E g(\chi_n) - \E g(\chi) < r_1, \\
\P (|\Delta_{n,N}^g-\E \Delta_{n,N}^g|< r_2)  \geq 1-\epsilon
\end{rcases}\implies \P (-r_1-r_2 < \Delta_{n,N}^g < r_2)\geq 1-\epsilon,
\end{equation*}
where $\Delta_{n,N}^g$ is defined in~\eqref{eq:MC_error}, reducing the upper 
bound of the CI to the error $r_2$, which depends on the bound on $g$
and the number of samples $N$ but not on $n$.

As explained in Section~\ref{subsec:CLT} above, if explicit constants in 
the bounds on the bias are not available in terms of the model parameters, 
as is the case with an up-and-out call option (see 
Proposition~\ref{prop:barrier} above and remarks following it),
we resort to the CLT in Theorem~\ref{thm:CLT} above. The plot on the right in 
Figure~\ref{fig:NIG-ConfInt} depicts the asymptotic CI for an up-and-out call 
as a function of $\log_2 N$, where $N$ is the number of samples used to 
estimate $\E g(\chi)$ and the asymptotic variance in~\eqref{eq:CLT} of 
Theorem~\ref{thm:CLT} is estimated using the sample. 

\begin{figure}[ht]
	\begin{tikzpicture} 
		\begin{axis} 
		[
		title={Non-asymptotic CI for 
			$g(\chi)= (K_0-\ov{S}_T)^+$},
		ymin=-.3,
		ymax=0.45,
		xmin=1,
		xmax=15,
		xlabel={$n$},
		width=7.85cm,
		height=5.5cm,
		axis on top=true,
		axis x line=bottom,
		axis y line=middle,
		legend style = {at={(1,0.02)}, anchor=south east},
		x label style={at={(axis description cs:1.025,0.175)},anchor=north},
		]
		\addplot[
		dashed,
		mark=+,
		color=black,
		]
		coordinates {(1,0.4069)(2,0.3225)(3,0.2782)(4,0.2549)(5,0.2422)(6,0.2362)(7,0.2327)(8,0.2307)(9,0.2300)(10,0.2293)(11,0.2293)(12,0.2289)(13,0.2289)(14,0.2289)(15,0.2291)(16,0.2289)(17,0.2286)(18,0.2288)(19,0.2288)(20,0.2290)
		};
		\addplot[
		densely dotted,
		color=black,
		style = thick,
		]
		coordinates {(1,0.4116)(2,0.3273)(3,0.2830)(4,0.2596)(5,0.2470)(6,0.2409)(7,0.2374)(8,0.2355)(9,0.2348)(10,0.2340)(11,0.2340)(12,0.2336)(13,0.2337)(14,0.2337)(15,0.2338)(16,0.2337)(17,0.2334)(18,0.2336)(19,0.2336)(20,0.2338)
		};
		\addplot[
		densely dotted,
		color=black,
		style = thick,
		]
		coordinates {(1,-8.661)(2,-4.236)(3,-2.014)(4,-0.899)(5,-0.340)(6,-0.058)(7,0.0821)(8,0.1527)(9,0.1884)(10,0.2060)(11,0.2152)(12,0.2195)(13,0.2218)(14,0.2230)(15,0.2237)(16,0.2239)(17,0.2238)(18,0.2240)(19,0.2241)(20,0.2243)
		};
		\addplot[
		solid,
		color=black,
		]
		coordinates {(1,0.2290)(20,0.2290)
		};
		\legend {$\Delta_{n,N}^g+\E g(\chi)$,
			Upper bound, Lower bound, 
			$\E g(\chi)$};
		\end{axis}
	\end{tikzpicture}
	\begin{tikzpicture} 
		\begin{axis} 
		[
		title={Asymptotic CI for $g(\chi)= (S_T-K_0)^+\cdot \1{\ov{S}_T\leq M}$},
		ymin=0,
		ymax=1.1,
		xmin=4,
		xmax=24,
		xlabel={$\log_2 N$},
		width=7.85cm,
		height=5.5cm,
		axis on top=true,
		axis x line=bottom,
		axis y line=middle,
		legend style = {at={(1,.4)}, anchor=south east},
		x label style={at={(axis description cs:0.93,0.3)},anchor=north},
		]
		\addplot[
		dashed,
		mark=+,
		color=black,
		]
		coordinates {(5,0.4567)(6,0.2343)(7,0.2949)(8,0.2750)(9,0.3440)(10,0.2960)(11,0.3391)(12,0.3181)(13,0.3030)(14,0.3038)(15,0.3069)(16,0.3041)(17,0.3046)(18,0.3061)(19,0.3067)(20,0.3038)(21,0.3053)(22,0.3054)(23,0.3058)(24,0.3054)
		};
		\addplot[
		densely dotted,
		color=black,
		style = thick,
		]
		coordinates {(5,1.0735)(6,0.6704)(7,0.6033)(8,0.4931)(9,0.4982)(10,0.4051)(11,0.4162)(12,0.3726)(13,0.3415)(14,0.3310)(15,0.3261)(16,0.3177)(17,0.3143)(18,0.3129)(19,0.3115)(20,0.3073)(21,0.3077)(22,0.3071)(23,0.3070)(24,0.3062)
		};
		\addplot[
		densely dotted,
		color=black,
		style = thick,
		]
		coordinates {(5,-0.160)(6,-0.201)(7,-0.013)(8,0.0569)(9,0.1898)(10,0.1870)(11,0.2619)(12,0.2635)(13,0.2644)(14,0.2765)(15,0.2876)(16,0.2905)(17,0.2950)(18,0.2992)(19,0.3019)(20,0.3004)(21,0.3029)(22,0.3037)(23,0.3046)(24,0.3045)
		};
		\addplot[
		solid,
		color=black,
		]
		coordinates {(1,0.3054)(20,0.3054)
		};
		\legend {$\Delta_{n_N,N}^g+\E g(\chi)$,
			Upper bound,Lower bound, 
			$\E g(\chi)$};
		\end{axis}
		\end{tikzpicture}
\caption{\label{fig:NIG-ConfInt}
The pictures show the point estimation and CIs for the hinsight put 
(left) and the up-and-out call (right) under the NIG model. 
NIG parameters: $\sigma=1$, $\theta=0.1$, $\kappa=0.1$ and 
$b=-0.05$. Option parameters: $S_0=2$, $K_0=3$, $M=8$ and 
$T=1$. The number of samples in the plot on the left equals 
$N=10^7$. The confidence level of $1-\epsilon=99\%$ applies to 
both plots.}
\end{figure}

\subsection{MLMC for a barrier payoff under NIG\label{subsec:MLMC_NIG_barrier_ex}}
We apply the MLMC algorithm for the SBA to the up-and-out call option 
in~\cite[Sec.~6.3]{MR3723380} (with payoff $g(\chi)=(S_T-K_0)^+\cdot
\1{\ov{S}_T\leq M}$, where $S_T=S_0\exp(X_T)$) under the NIG model. The top 
left (resp. right) plot in Figure~\ref{fig:NIG-MLMC} graphs the estimated and 
theoretically predicted mean (resp. variance) of the difference of two 
consecutive levels (as a function of $n$). 

It is common practice in MLMC to estimate the bias and level variances (rather 
than use the theoretical bounds such as those in Theorem~\ref{thm:MLMC}) first
and then compute the numbers of samples $(N_k)_{k\in\{1,\ldots,n\}}$ at each 
level by solving a simple optimisation problem. This often improves the overall 
performance of the algorithm but requires an initial computational investment. 
The fact that $(N_k)_{k\in\{1,\ldots,n\}}$ are based on estimates gives rise to 
some oscillation in their behaviour and, consequently, in that of the 
computational cost. However, as expected from~\eqref{eq:MLMC_const}, the bottom 
left plot in Figure~\ref{fig:NIG-MLMC} shows that $(N_k)_{k\in\{1,\ldots,n\}}$ 
constitute approximately straight lines for various levels of accuracy. The 
bottom right plot in Figure~\ref{fig:NIG-MLMC} shows that the computational 
complexity is approximately constant, as expected from the analysis in
Section~\ref{subsec:Complexity_SBA} above. Moreover, the difference in the 
complexity between the MC and MLMC is numerically seen to be small. This is not 
surprising since, as explained in Section~\ref{subsec:Complexity_SBA} above, 
the two differ by a log-factor. The analogous figure for the MLMC based on the 
RWA for the identical model parameters and option is given 
in~\cite[Fig.~7]{MR3723380}.

\begin{figure}[ht]
	\begin{tikzpicture}
		\pgfplotsset{
			scale only axis,
			xmin=2, xmax=5
		} 
		\begin{axis}[
		width=6.75cm,
		height=3cm,
		axis on top=true,
		axis x line=bottom,
		axis y line*=left,
		xlabel = {$n$},
		xmin = 2, xmax = 20,
		ymin = -19.5, ymax = -2.2,
		title={Bias decay $\log_2|\E P_n-\E P_{n-1}|$},
		legend style={at={(0.4,.05)},anchor=south east},
		]
		\addplot[
		solid,
		mark=+,
		color=black,
		]
		coordinates {(1,1.2025)(2,-2.216)(3,-2.862)(4,-3.681)(5,-4.597)(6,-5.514)(7,-6.479)(8,-7.399)(9,-8.408)(10,-9.387)(11,-10.35)(12,-11.27)(13,-12.15)(14,-13.26)(15,-14.30)(16,-15.31)(17,-15.32)(18,-16.88)(19,-17.90)(20,-19.50)
		};
		\addlegendentry{\scriptsize Observed}
		\addplot[
		dashed,
		mark=+,
		color=black,
		]
		coordinates {(1,-1.716)(2,-2.216)(3,-2.716)(4,-3.216)(5,-3.716)(6,-4.216)(7,-4.716)(8,-5.216)(9,-5.716)(10,-6.216)(11,-6.716)(12,-7.216)(13,-7.716)(14,-8.216)(15,-8.716)(16,-9.216)(17,-9.716)(18,-10.21)(19,-10.71)(20,-11.21)
		};
		\addlegendentry{\scriptsize Bound}
		\end{axis}
	\end{tikzpicture}
	\begin{tikzpicture}
		\pgfplotsset{
			scale only axis,
			xmin=2, xmax=5
		} 
		\begin{axis}[
		width=6.75cm,
		height=3cm,
		axis on top=true,
		axis x line=bottom,
		axis y line*=left,
		xlabel = {$n$},
		xmin = 2, xmax = 20,
		ymin = -15.75, ymax = 1.2,
		title={ Variance decay $\log_2\V[P_n-P_{n-1}]$},
		legend style={at={(0.4,.05)},anchor=south east},
		]
		\addplot[
		solid,
		mark=+,
		color=black,
		]
		coordinates {(1,4.0510)(2,1.2156)(3,0.5843)(4,-0.225)(5,-1.137)(6,-2.054)(7,-3.011)(8,-3.925)(9,-4.926)(10,-5.897)(11,-6.866)(12,-7.774)(13,-8.656)(14,-9.774)(15,-10.86)(16,-11.80)(17,-11.81)(18,-13.48)(19,-14.73)(20,-15.75)
		};
		\addlegendentry{\scriptsize Observed}
		\addplot[
		dashed,
		mark=+,
		color=black,
		]
		coordinates {(1,1.7156)(2,1.2156)(3,0.7156)(4,0.2156)(5,-0.284)(6,-0.784)(7,-1.284)(8,-1.784)(9,-2.284)(10,-2.784)(11,-3.284)(12,-3.784)(13,-4.284)(14,-4.784)(15,-5.284)(16,-5.784)(17,-6.284)(18,-6.784)(19,-7.284)(20,-7.784)
		};
		\addlegendentry{\scriptsize Bound}
		\end{axis}
	\end{tikzpicture}
	\begin{tikzpicture}
		\pgfplotsset{
			scale only axis,
			xmin=2, xmax=5
		} 
		\begin{axis}[
		width=6.75cm,
		height=3cm,
		axis on top=true,
		axis x line=bottom,
		axis y line*=left,
		xlabel = {$k$},
		ylabel = {$\log_2 N_k$},
		xmin = 1, xmax = 18,
		ymin = 13, ymax = 30,
		title={ Number of levels and samples per level},
		legend style={at={(1,.6)},anchor=south east},
		]
		\addplot[
		solid,
		mark=*,
		color=black,
		]
		coordinates {(1,20.977)(2,19.042)(3,18.415)(4,17.818)(5,17.080)(6,16.548)(7,15.889)(8,15.788)(9,14.958)(10,13.287)
		};
		\addlegendentry{\scriptsize $\epsilon=2^{-7}$}
		\addplot[
		solid,
		mark=o,
		color=black,
		]
		coordinates {(1,25.040)(2,23.081)(3,22.560)(4,21.920)(5,21.218)(6,20.737)(7,20.089)(8,19.413)(9,19.276)(10,18.537)(11,17.607)(12,17.270)(13,16.948)(14,16.609)
		};
		\addlegendentry{\scriptsize $\epsilon=2^{-9}$}
		\addplot[
		dashed,
		mark=+,
		color=black,
		]
		coordinates {(1,29.068)(2,27.154)(3,26.548)(4,25.931)(5,25.397)(6,24.792)(7,24.196)(8,23.545)(9,22.832)(10,22.539)(11,22.099)(12,20.946)(13,20.713)(14,20.151)(15,16.752)(16,16.254)
		};
		\addlegendentry{\scriptsize $\epsilon=2^{-11}$}
		\end{axis}
	\end{tikzpicture}
	\begin{tikzpicture} 
		\pgfplotsset{
			scale only axis,
			xmin=5, xmax=13
		}
		\begin{axis}[
		width=6.75cm,
		height=3cm,
		axis on top=true,
		axis x line=bottom,	
		axis y line*=left,
		xlabel={$\log_2(1/\epsilon)$},
		ymin=10, ymax=40,
		title={Logarithm of the computational cost},
		legend style={at={(1,.01)},anchor=south east},
		]
		\addplot[
		solid,
		mark=+,
		mark options={scale=.8},
		color=black,
		]
		coordinates {(5.0,23.815)(6.0,25.569)(7.0,27.114)(8.0,29.138)(9.0,31.115)(10.0,33.423)(11.0,35.430)(12.0,37.544)(13.0,39.761)
		}; 
		\addlegendentry{\scriptsize Observed (MLMC)}
		\addplot[
		dashed,
		mark=+,
		mark options={scale=.8},
		color=black,
		]
		coordinates {(5.0,23.545)(6.0,25.545)(7.0,27.545)(8.0,29.545)(9.0,31.545)(10.,33.545)(11.,35.545)(12.,37.545)(13.,39.545)
		}; 
		\addlegendentry{\scriptsize Predicted (MLMC)}
		\addplot[
		solid,
		mark=*,
		mark options={scale=.8},
		color=black,
		]
		coordinates {(5.0,21.315)(6.0,23.455)(7.0,25.619)(8.0,27.672)(9.0,29.741)(10.0,31.900)(11.0,34.040)(12.0,36.214)(13.0,38.211)
		};
		\addlegendentry{\scriptsize Observed (MC)}
		\addplot[
		dashed,
		mark=*,
		mark options={scale=.8},
		color=black,
		]
		coordinates {(5.0,21.315)(6.0,23.495)(7.0,25.656)(8.0,27.800)(9.0,29.931)(10.,32.052)(11.,34.163)(12.,36.266)(13.,38.362)
		};
		\addlegendentry{\scriptsize Predicted (MC)}
		\end{axis}
	\end{tikzpicture}
\caption{\label{fig:NIG-MLMC}
The pictures show the level bias decay, level variance decay, samples per level 
and complexities of MC and MLMC implementations for the up-and-out call 
$g(\chi)= e^{-rT}(S_T-K)^+\1{\ov{S}_T<M}$ and the NIG process. 
NIG parameters: 
$\sigma=0.1836$, $\theta=-0.1313$, $\kappa=1.2819$ and $b=0.1571$ 
(see~\cite[Sec.~3]{MR3723380} and the reference therein). Option 
parameters: $S_0=100$, $K_0=100$, $M=115$, $T=1$ and $r=0.05$. 
The bounds in the top two graphs are based on Proposition~\ref{prop:barrier} 
(with $\gamma=q=1$) and synchronous coupling. See 
Subsection~\ref{subsec:Complexity_SBA} for the computational complexity 
of MC and MLMC in the bottom right.}
\end{figure}

The computational complexity of MLMC in Figure~\ref{fig:NIG-MLMC} is greater 
than that of the MC (for $\epsilon>1/8000$) due to the size of the leading 
constant. Overall, the performance of both MC and MLMC in this examples is 
good, with the actual decay rates of the bias and level variances being better 
than the theoretical bounds by a factor of $2$.

\section{Proofs and technical results\label{sec:proofs}}

Let $X=(X_t)_{t\geq0}$ be a L\'evy process, which we assume not to be 
compound Poisson with drift. By Doeblin's diffuseness 
lemma~\cite[Lem.~13.22]{MR1876169}, this is equivalent to the following 
requirement, which we assume throughout the remainder of the paper. 
\begin{asm}
\label{asm:D} 
$\P(X_t=x)=0$ for all $x\in\R$ and for some (and hence all) $t>0$.
\end{asm}

\subsection{The concave majorant of $X$ and its coupling with $(\ell,Y)$}
\label{subsec:coupling}
Given a countable set $\mS$ and a function $\phi:\mS\to(0,\infty)$ such that 
$\sum_{s\in\mS}\phi(s)<\infty$, \emph{size-biased sampling} of $\mS$ based 
on the function $\phi$ produces a random enumeration $(s_n)_{n\in\N}$ of 
$\mS$ using the following sequential construction: let $Z_0=\emptyset$ and 
assume we have already sampled the points in $Z_{n-1}=\{s_1,\ldots,s_{n-1}\}$ 
for some $n\in\N$; then, conditional on $Z_{n-1}$, the  random element $s_n$ 
in $\mS\setminus Z_{n-1}$ follows the law $\P(s_n=s|Z_{n-1})=\phi(s)/\sum_{s'\in\mS\setminus Z_{n-1}}\phi(s')$, $s\in\mS\setminus Z_{n-1}$.

The concave majorant of a path of $(X_t)_{t\in[0,T]}$ is the point-wise smallest 
concave function $C:[0,T]\to\R$ satisfying $C_t\geq X_t$ for all $t\in[0,T]$. 
Since $X$ is not compound Poisson with drift, it is possible to obtain a 
complete description of the law of $C$ (see~\cite{MR2978134} for details), 
which we now recall. Note that $t\mapsto C_t$ is  a piecewise linear function 
comprising of infinitely many line segments known as \textit{faces}. 
Each face has a positive length and a height, which is a real number. If the 
faces are ordered chronologically (i.e. as they arise with increasing $t$), 
the concavity of $C$ implies that the sequence of the corresponding slopes is 
strictly decreasing (see Figure~\ref{fig:FacesCM}(a) below). The lengths of the 
faces constitute a countable set of positive numbers with a finite sum clearly 
equal to $T$. 

We may thus order randomly the faces of $C$ using size-biased 
sampling on lengths, see Figure~\ref{fig:FacesCM} below. 
This random ordering almost surely differs from the chronological one, with 
longer faces much more likely to appear near the beginning of the sequence. 
For any $n\in\N=\{1,2,\ldots\}$, let $g_n$ (resp. $d_n$) be the left (resp. right) 
end point of the $n$-th face of $C$ in the size-biased enumeration. 
The size-biased sequence of lengths and heights of the faces of $C$ satisfies 
the following equality in law~\cite[Thm~1]{MR2978134}: 
\begin{equation}
\label{eq:levy-minorant}
	((d_n-g_n,C_{d_n}-C_{g_n}))_{n\in\N}\overset{d}{=}
	((\ell_n,Y_{L_{n-1}}-Y_{L_n}))_{n\in\N},
\end{equation}
where $Y$ is a copy of $X$, independent of the stick-breaking 
process $\ell=(\ell_n)_{n\in\N}$ on $[0,T]$ based on the uniform law $\U(0,1)$. 
We stress that the equality in law~\eqref{eq:levy-minorant} holds in the sense 
of random processes indexed by $\N$. Surprisingly, 
by~\eqref{eq:levy-minorant}, the law of the sequence of lengths 
$(d_n-g_n)_{n\in\N}$ does not depend on $X$. This fact is the basis for a 
coupling of $(\ell,Y)$ and $X$ such that~\eqref{eq:levy-minorant} holds a.s. 

This coupling, constructed below, is crucial for the analysis of the error in 
the~\nameref{alg:SBA} above and will be used throughout the paper. 
Indeed, under such a coupling,~\eqref{eq:chi} holds a.s. since the location 
(resp. time) of the supremum of $X$ over $[0,T]$ equals the sum of all the 
heights (resp. lengths) of the faces of $C$ with positive slope. (Note that 
the function $t\mapsto C_t$ is concave and thus of finite variation, 
making the sequence of heights $(C_{d_n}-C_{g_n})_{n\in\N}=
(Y_{L_{n-1}}-Y_{L_n})_{n\in\N}$ absolutely summable.) In particular, 
it implies $Y_T=X_T$ a.s. 

Consider the countable set of faces of the concave majorant $C$ of $X$. 
Each face consists of a pair $(x,y)$, where $x>0$ is the length and $y\in\R$ is 
the height of the face. Since the lengths of the faces are positive and summable 
with sum $T$, it is possible to perform size-biased sampling of the faces based 
on the function $\phi:(x,y)\mapsto x$, which then yields the random 
enumeration $((d_n-g_n,C_{d_n}-C_{g_n}))_{n\in\N}$ of the faces of $C$. 
This enumeration, by~\cite[Thm~1]{MR2978134}, satisfies the distributional 
equality~\eqref{eq:levy-minorant}. Furthermore, in this case, the size-biased 
sampling has a geometric interpretation as illustrated by 
Figure~\ref{fig:FacesCM} below, wherein $(g_n)_{n\in\N}$ and $(d_n)_{n\in\N}$ 
are the left and right endpoints of the $n$-th face, respectively. Note that 
Assumption~\ref{asm:D} and~\eqref{eq:levy-minorant} imply that there is 
no face of $C$ that is horizontal. Hence the time at which the supremum is 
attained is a.s. unique.

\begin{figure}[ht]
	\begin{tikzpicture}
		\pgfmathsetseed{101101}
		\BrownianMotion{0}{0}{100}{0.003}{0.2*.7}{black}{};
		\BrownianMotion{101*.00285}{2.5*.7}{123}{0.003}{0.2*.7}{black}{};
		\BrownianMotion{225*.00285}{2.9*.7}{54}{0.003}{0.2*.7}{black}{};
		\BrownianMotion{280*.00285}{2.4*.7}{172}{0.003}{0.2*.7}{black}{};
		\BrownianMotion{453*.00285}{0.8*.7}{27}{0.003}{0.2*.7}{black}{};
		\BrownianMotion{481*.00285}{2.0*.7}{80}{0.003}{0.2*.7}{black}{};
		\BrownianMotion{562*.00285}{1.3*.7}{69}{0.003}{0.2*.7}{black}{};
		\BrownianMotion{632*.00285}{4.6*.7}{126}{0.003}{0.2*.7}{black}{};
		\BrownianMotion{759*.00285}{0.4*.7}{208}{0.003}{0.2*.7}{black}{};
		\BrownianMotion{968*.00285}{0.8*.7}{67}{0.003}{0.2*.7}{black}{};
		\BrownianMotion{1036*.00285}{1.8*.7}{132}{0.003}{0.2*.7}{black}{};
		\BrownianMotion{1169*.00285}{0.1*.7}{164}{0.003}{0.2*.7}{black}{};
		\BrownianMotion{1324*.00285}{1.0*.7}{45}{0.003}{0.2*.7}{black}{};
		\draw[red] (0,0) -- (6*.00285,.455*.7) 
		-- (18*.00285,.95*.7)
		-- (180*.00285,4.6*.7)
		-- (199*.00285,4.96*.7)
		-- (635*.00285,5.28*.7)
		-- (668*.00285,5.31*.7)
		-- (745*.00285,4.99*.7)
		-- (1148*.00285,3.0*.7)
		-- (1275*.00285,2.08*.7)
		-- (1337*.00285,1.4*.7)
		-- (1362*.00285,.89*.7)
		-- (1369*.00285,.68*.7)
		node[right] { $C$};
		\node[circle, fill=black, scale=0.5] at (933*.00285,188*3*.7/403+215*5.03*.7/403) {}{};
		\node[circle, fill=black, scale=0.5, label=above right:{ $U_1$}] at (933*.00285,0) {}{};
		\draw[dashed] (933*.00285,0) -- (933*.00285,188*3*.7/403+215*5.03*.7/403);
		\node[circle, fill=blue, scale=0.5] at (745*.00285,5.03*.7) {}{};
		\node[circle, fill=blue, scale=0.5, label=below: {\color{blue} $g_1$}] at (745*.00285,0) {}{};
		\node[circle, fill=blue, scale=0.5, label=left: {\color{blue} $C_{g_1}$}] at (0,5.03*.7) {}{};
		\draw[blue,dashed] (745*.00285,0) -- (745*.00285,5.03*.7);
		\draw[blue,dashed] (0,5.03*.7) -- (745*.00285,5.03*.7);
		\node[circle, fill=blue, scale=0.5] at (1148*.00285,3*.7) {}{};
		\node[circle, fill=blue, scale=0.5, label=below: {\color{blue} $d_1$}] at (1148*.00285,0) {}{};
		\node[circle, fill=blue, scale=0.5, label=left: {\color{blue} $C_{d_1}$}] at (0,3*.7) {}{};
		\draw[blue,dashed] (1148*.00285,0) -- (1148*.00285,3*.7);
		\draw[blue,dashed] (0,3*.7) -- (1148*.00285,3*.7);
		\draw [thin, draw=gray, ->] (0,0) -- (1400*.00285,0);
		\draw [thin, draw=gray, ->] (0,0) -- (0,5.5*.7);
		\node[circle, fill=black, scale=0.5, label=below: $T$] at (1369*.00285,0) {}{};
		\draw[very thick, draw=blue] (0,0) -- (1369*.00285,0);
	\end{tikzpicture}
	\begin{tikzpicture}
		\pgfmathsetseed{101101}
		\BrownianMotion{0}{0}{100}{0.003}{0.2*.7}{black}{};
		\BrownianMotion{101*.00285}{2.5*.7}{123}{0.003}{0.2*.7}{black}{};
		\BrownianMotion{225*.00285}{2.9*.7}{54}{0.003}{0.2*.7}{black}{};
		\BrownianMotion{280*.00285}{2.4*.7}{172}{0.003}{0.2*.7}{black}{};
		\BrownianMotion{453*.00285}{0.8*.7}{27}{0.003}{0.2*.7}{black}{};
		\BrownianMotion{481*.00285}{2.0*.7}{80}{0.003}{0.2*.7}{black}{};
		\BrownianMotion{562*.00285}{1.3*.7}{69}{0.003}{0.2*.7}{black}{};
		\BrownianMotion{632*.00285}{4.6*.7}{113}{0.003}{0.2*.7}{black}{};
		\phantom{
			\BrownianMotion{632*.00285}{4.6*.7}{13}{0.003}{0.2*.7}{black}{};
			\BrownianMotion{759*.00285}{0.4*.7}{208}{0.003}{0.2*.7}{black}{};
			\BrownianMotion{968*.00285}{0.8*.7}{67}{0.003}{0.2*.7}{black}{};
			\BrownianMotion{1036*.00285}{1.8*.7}{113}{0.003}{0.2*.7}{black}{};
		}
		\BrownianMotion{1149*.00285}{3.0*.7}{19}{0.003}{0.2*.7}{black}{};
		\BrownianMotion{1169*.00285}{0.1*.7}{164}{0.003}{0.2*.7}{black}{};
		\BrownianMotion{1324*.00285}{1.0*.7}{45}{0.003}{0.2*.7}{black}{};
		\draw[red] (0,0) -- (6*.00285,.455*.7) 
		-- (18*.00285,.95*.7)
		-- (180*.00285,4.6*.7)
		-- (199*.00285,4.96*.7)
		-- (635*.00285,5.28*.7)
		-- (668*.00285,5.31*.7)
		--(745*.00285,4.99*.7){};
		\draw[red] (1148*.00285,3.0*.7)
		-- (1275*.00285,2.08*.7)
		-- (1337*.00285,1.4*.7)
		-- (1362*.00285,.89*.7)
		-- (1369*.00285,.68*.7)
		node[right] { $C$};
		\node[circle, fill=black, scale=0.5] at (433*.00285,202*5.33*.7/436+234*5.02*.7/436) {}{};
		\node[circle, fill=black, scale=0.5, label=below:{ $U_2$}] at (433*.00285,0) {}{};
		\draw[dashed] (433*.00285,0) -- (433*.00285,202*5.33*.7/436+234*5.02*.7/436);
		\node[circle, fill=blue, scale=0.5] at (199*.00285,5.02*.7) {}{};
		\node[circle, fill=blue, scale=0.5, label=below: {\color{blue} $g_2$}] at (199*.00285,0) {}{};
		\node[circle, fill=blue, scale=0.5, label=below left: {\color{blue} $C_{g_2}$}] at (0,5.02*.7) {}{};
		\draw[blue,dashed] (199*.00285,0) -- (199*.00285,5.02*.7);
		\draw[blue,dashed] (0,5.02*.7) -- (199*.00285,5.02*.7);
		\node[circle, fill=blue, scale=0.5] at (635*.00285,5.33*.7) {}{};
		\node[circle, fill=blue, scale=0.5, label=below: {\color{blue} $d_2$}] at (635*.00285,0) {}{};
		\node[circle, fill=blue, scale=0.5, label=left: {\color{blue} $C_{d_2}$}] at (0,5.33*.7) {}{};
		\draw[blue,dashed] (635*.00285,0) -- (635*.00285,5.33*.7);
		\draw[blue,dashed] (0,5.33*.7) -- (635*.00285,5.33*.7);
		\draw [thin, draw=gray, ->] (0,0) -- (1400*.00285,0);
		\draw [thin, draw=gray, ->] (0,0) -- (0,5.5*.7);
		\draw [very thick, draw=blue] (0,0) -- (745*.00285,0);
		\draw [very thick, draw=blue] (1148*.00285,0) -- (1369*.00285,0);
		\node[circle, fill=black, scale=0.5, label=below: $T$] at (1369*.00285,0) {}{};
	\end{tikzpicture}
	\begin{tikzpicture}
		\pgfmathsetseed{101101}
		\BrownianMotion{0}{0}{100}{0.003}{0.2*.7}{black}{};
		\BrownianMotion{101*.00285}{2.5*.7}{98}{0.003}{0.2*.7}{black}{};
		\phantom{
			\BrownianMotion{101*.00285}{2.5*.7}{25}{0.003}{0.2*.7}{black}{};
			\BrownianMotion{225*.00285}{2.9*.7}{54}{0.003}{0.2*.7}{black}{};
			\BrownianMotion{280*.00285}{2.4*.7}{172}{0.003}{0.2*.7}{black}{};
			\BrownianMotion{453*.00285}{0.8*.7}{27}{0.003}{0.2*.7}{black}{};
			\BrownianMotion{481*.00285}{2.0*.7}{80}{0.003}{0.2*.7}{black}{};
			\BrownianMotion{562*.00285}{1.3*.7}{69}{0.003}{0.2*.7}{black}{};
			\BrownianMotion{632*.00285}{4.6*.7}{2}{0.003}{0.2*.7}{black}{};
		}
		\BrownianMotion{635*.00285}{4.95*.7}{111}{0.003}{0.2*.7}{black}{};
		\phantom{
			\BrownianMotion{632*.00285}{4.6*.7}{13}{0.003}{0.2*.7}{black}{};
			\BrownianMotion{759*.00285}{0.4*.7}{208}{0.003}{0.2*.7}{black}{};
			\BrownianMotion{968*.00285}{0.8*.7}{67}{0.003}{0.2*.7}{black}{};
			\BrownianMotion{1036*.00285}{1.8*.7}{113}{0.003}{0.2*.7}{black}{};
		}
		\BrownianMotion{1149*.00285}{3.0*.7}{19}{0.003}{0.2*.7}{black}{};
		\BrownianMotion{1169*.00285}{0.1*.7}{164}{0.003}{0.2*.7}{black}{};
		\BrownianMotion{1324*.00285}{1.0*.7}{45}{0.003}{0.2*.7}{black}{};
		\draw[red] (0,0) -- (6*.00285,.455*.7) 
		-- (18*.00285,.95*.7)
		-- (180*.00285,4.6*.7)
		-- (199*.00285,4.96*.7){};
		\draw[red] (635*.00285,5.28*.7)
		-- (668*.00285,5.31*.7)
		--(745*.00285,4.99*.7){};
		\draw[red] (1148*.00285,3.0*.7)
		-- (1275*.00285,2.08*.7)
		-- (1337*.00285,1.4*.7)
		-- (1362*.00285,.89*.7)
		-- (1369*.00285,.68*.7)
		node[right] { $C$};
		\node[circle, fill=black, scale=0.5] at (150*.00285,30*.95*.7/162+132*4.6*.7/162) {}{};
		\node[circle, fill=black, scale=0.5, label=below:{ $U_3$}] at (150*.00285,0) {}{};
		\draw[dashed] (150*.00285,0) -- (150*.00285,30*.95*.7/162+132*4.6*.7/162);
		\node[circle, fill=blue, scale=0.5] at (18*.00285,.95*.7) {}{};
		\node[circle, fill=blue, scale=0.5, label=below:{\color{blue} $g_3$}] at (18*.00285,0) {}{};
		\node[circle, fill=blue, scale=0.5, label=above left:{\color{blue} $C_{g_3}$}] at (0,.95*.7) {}{};
		\draw[blue,dashed] (18*.00285,0) -- (18*.00285,.95*.7);
		\draw[blue,dashed] (0,.95) -- (18*.00285,.95*.7);
		\node[circle, fill=blue, scale=0.5] at (180*.00285,4.6*.7) {}{};
		\node[circle, fill=blue, scale=0.5, label=below right:{\color{blue}  $d_3$}] at (180*.00285,0) {}{};
		\node[circle, fill=blue, scale=0.5, label=below left:{\color{blue} $C_{d_3}$}] at (0,4.6*.7) {}{};
		\draw[blue,dashed] (180*.00285,0) -- (180*.00285,4.6*.7);
		\draw[blue,dashed] (0,4.6*.7) -- (180*.00285,4.6*.7);
		\draw [thin, draw=gray, ->] (0,0) -- (1400*.00285,0);
		\draw [thin, draw=gray, ->] (0,0) -- (0,5.5*.7);
		\draw [very thick, draw=blue] (0,0) -- (199*.00285,0);
		\draw [very thick, draw=blue] (635*.00285,0) -- (745*.00285,0);
		\draw [very thick, draw=blue] (1148*.00285,0) -- (1369*.00285,0);
		\node[circle, fill=black, scale=0.5, label=below: $T$] at (1369*.00285,0) {}{};
	\end{tikzpicture}
\caption{\label{fig:FacesCM} 
Selecting the first three faces of the concave majorant: the total length of 
the thick blue segment(s) on the abscissa equal the stick sizes $T$, 
$T-(d_1-g_1)$ and $T-(d_1-g_1)-(d_2-g_2)$, respectively. 
The independent random variables $U_1,U_2,U_3$ are uniform on the sets 
$[0,T]$, $[0,T]\setminus(g_1,d_1)$, $[0,T]\setminus\bigcup_{i=1}^2(g_i,d_i)$, 
respectively. Note that the residual length of unsampled faces after $n$ 
samples is $L_n$.}
\end{figure}

We now explain how to couple $(\ell, Y)$ with $X$ in such a way 
that~\eqref{eq:levy-minorant} (and hence~\eqref{eq:chi}) holds a.s. 
Start by recalling from~\eqref{eq:levy-minorant} that 
$\big(\big(d_n-g_n,C_{d_n}-C_{g_n}\big)\big)_{n\in\N}\overset{d}{=}
\big(\big(\ell'_n,Y'_{L'_{n-1}}-Y'_{L'_n}\big)\big)_{n\in\N}$, where $Y'$ is a 
copy of $X$, independent of the stick-breaking process 
$\ell'=(\ell'_n)_{n\in\N}$ on $[0,T]$ based on the uniform law $\U(0,1)$ and 
$L'_{n-1}=\sum_{k=n}^\infty\ell'_k$. Now recall that the Skorokhod space 
$\mathcal{D}[0,T]$ of right-continuous functions on $[0,T]$ with left-hand 
limits (see~\cite[p.~109]{MR1700749}) is a Polish 
space~\cite[p.~112]{MR1700749} and thus a Borel 
space~\cite[Thm~A1.2]{MR1876169}. By possibly extending the 
original probability space,~\cite[Thm~6.10]{MR1876169} asserts the 
existence of a random element $Y$ in $\mathcal{D}[0,T]$ such that 
\begin{equation}
\label{eq:levy-minorant-2}
\big((d_n-g_n)_{n\in\N},\big(C_{d_n}-C_{g_n}\big)_{n\in\N},Y\big)
\overset{d}{=}
\big((\ell'_n)_{n\in\N},\big(Y'_{L'_{n-1}}-Y'_{L'_n}\big)_{n\in\N},Y'\big).
\end{equation}
Consequently, the process $Y$ has the same law as $Y'\overset{d}{=}X$. 
If we define the sequence $\ell=(\ell_n)_{n\in\N}$ through $\ell_n=d_n-g_n$ 
and $L_{n-1}=\sum_{k=n}^\infty \ell_k$ for each $n\in\N$, 
then~\eqref{eq:levy-minorant-2} implies that $Y$ is independent of $\ell$. 
Again, by~\eqref{eq:levy-minorant-2}, the increment of $Y$ over the interval 
$[L_n,L_{n-1}]$ is equal to $Y_{L_{n-1}}-Y_{L_n}=C_{d_n}-C_{g_n}$ a.s. 
Thus, this coupling between $(\ell,Y)$ and $X$ is the desired one, 
as~\eqref{eq:levy-minorant} holds a.s.

The coupling $(X,\ell,Y)$ can also be obtained without the abstract 
result~\citep[Thm~6.10]{MR1876169} using the `3214' 
transformation from~\cite{MR2978134}, which is explicit in the 
trajectory of $X$. Since the details of the coupling are not important 
in this paper, we use the abstract result for brevity.

\subsection{The law of the error and the proof of Theorem~\ref{thm:error}\label{subsec:Proof_of_Thm1}}

In the present subsection we will prove Theorem~\ref{thm:error}. We also 
state and prove Proposition~\ref{prop:E_tau}, which explains why the error 
$\delta_n^\SB$ of the SBA $\chi_n$ is typically smaller than $\delta_n$.

\begin{proof}[Proof of Theorem~\ref{thm:error}]\label{proof:Proof_of_Thm1}
By the coupling from Subsection~\ref{subsec:coupling}, 
the equality in~\eqref{eq:chi} holds a.s., i.e. we have 
$\chi=(X_T,\ov X_T,\tau_T)
=\sum_{k=1}^\infty(Y_{L_{k-1}}-Y_{L_k},(Y_{L_{k-1}}-Y_{L_k})^+,
\ell_k\cdot\1{Y_{L_{k-1}}-Y_{L_k}>0})$. Hence, from 
the definition in~\eqref{def:Delta}, we clearly obtain 
\[(Y_{L_n},\Delta_n,\delta_n)=\sum_{k=n+1}^\infty
\big(Y_{L_{k-1}}-Y_{L_k},(Y_{L_{k-1}}-Y_{L_k})^+,\ell_k\cdot 
\1{Y_{L_{k-1}}-Y_{L_k}>0}\big).\]
In particular, we have $\delta_n\leq\sum_{k=n+1}^\infty \ell_k=L_n$ and thus 
$|\delta_n^\SB|\leq L_n$. 

We now apply~\eqref{eq:levy-minorant} to conclude that the tail sum in the 
display above has the required law. Note first that, given $L_n$, 
$(\ell_{n+k})_{k\in\N}$ is a stick-breaking process on the interval $[0,L_n]$. 
Thus, since $Y$ and $\ell$ are independent, the law of the sequence 
$((\ell_{n+k},Y_{L_{k+n-1}}-Y_{L_{k+n}}))_{k\in\N}$, given $L_n$, 
is the same law as that of the right-hand side of~\eqref{eq:levy-minorant} 
applied to the interval $[0,L_n]$. Put differently, by~\eqref{eq:levy-minorant}, 
this sequence has the same law as the sequence of the faces of the concave 
majorant of the L\'evy process $Y$ over the interval $[0,L_n]$ in size-biased 
order. Hence, identity~\eqref{eq:chi} applied to the interval $[0,L_n]$ 
(instead of $[0,T]$), together with the independence of $Y$ and $\ell$,
yields the first equality in law in~\eqref{eq:error-law}:
\begin{equation*} 
(Y_{L_n},\ov Y_{L_n},\tau_{L_n}(Y))\overset{d}{=} 
\sum_{k=n+1}^{\infty}\big(Y_{L_{k-1}}-Y_{L_k},(Y_{L_{k-1}}-Y_{L_k})^+, 
\ell_k\cdot \1{Y_{L_{k-1}}-Y_{L_k}>0}\big). 
\end{equation*}
The second distributional 
identity in~\eqref{eq:error-law} follows from the definition of 
$(\Delta_n^\SB,\delta_n^\SB)$ as a measurable transformation of 
$(Y_{L_n},\Delta_n,\delta_n)$. 

For any $n\in\N$, the second identity in~\eqref{eq:error-law} implies
$0\leq \Delta_n^\SB$. The definition of $\Delta_n$ in~\eqref{def:Delta}
and the inequality $Y_{L_n}^+\leq  (Y_{L_n}-Y_{L_{n+1}})^+ + Y_{L_{n+1}}^+$ 
yield the following: 
\[
\Delta_{n+1}^\SB
=\Delta_{n+1}-Y_{L_{n+1}}^+
= \Delta_n-(Y_{L_n}-Y_{L_{n+1}})^+ -Y_{L_{n+1}}^+ 
\leq \Delta_n-Y_{L_n}^+=\Delta_n^\SB\leq \Delta_n.
\] 
This concludes the proof of the theorem. 
\end{proof}

\begin{prop}\label{prop:E_tau}
Let $X$ satisfy Assumption~\ref{asm:D}. 
Then the following statements hold.
\item[{\normalfont(a)}] 
For any $t>0$, we have $\E\tau_t(X)=\int_0^t \P(X_s>0)ds$.
\item[{\normalfont(b)}] If $t^{-1}\int_0^t \P(X_s>0)ds-\P(X_t>0)\to 0$ as 
$t\searrow 0$, then $\E[\delta_n^\SB/L_n]\to 0$ as $n\to\infty$.
\item[{\normalfont(c)}] If $\P(X_t>0)\to\rho_0\in[0,1]$ as $t\searrow 0$, 
then (b) holds and $\E[\delta_n/L_n]\to\rho_0$ as $n\to\infty$.
\item[{\normalfont(d)}] If $\P(X_t>0)=\rho_0\in[0,1]$ for all $t\in(0,T]$, 
then $\E[\delta_n^\SB|L_n]=\E[\delta_n|L_n]-L_n\rho_0=0$ a.s.
\end{prop}
\begin{rem}\label{rem:E_tau}
(i) Note that $\tau_T\in[\tau_T-\delta_n,\tau_T-\delta_n+L_n]$ and, given
$L_n$, SBA $\chi_n$ chooses randomly the endpoints of the interval via a 
Bernoulli random variable with mean $\P(Y_{L_n}>0|L_n)$.\\
(ii) The assumption in (d) holds if e.g. $X$ is a subordinated stable or a 
symmetric L\'evy process. Moreover, it implies that the third coordinate in
$\chi_n$ is unbiased, since the expectation of its error vanishes: $\E[\delta^\SB_n]=0$. In contrast we have  
$\E[\delta_n]=\rho_0T/2^n$.\\
(iii) The bias of the third coordinate of $\chi_n$, conditional on $L_n=t$, 
equals $\int_0^t \P(X_s>0)ds-t\P(X_t>0)$ by~\eqref{eq:E_delta} below.  
This quantity is generally well behaved as $t\to0$. More specifically, 
we have $t^{-1}\int_0^t\P(X_s>0)ds-\P(X_t>0)\to 0$ as $t\searrow 0$ 
(thus satisfying the assumption in (b)) if $t\mapsto\P(X_t>0)$ is slowly 
varying at $0$~\cite[Prop.~1.5.8]{MR1015093}. \\
(iv) Note that the assumption in (c) implies that of (b). This assumption, 
known as Spitzer's condition~\cite[Thm~VI.3.14]{MR1406564}, is satisfied if 
for example $X$ converges weakly under the zooming-in 
procedure~\cite[Sec.~2.2]{ZoomIn}.
\end{rem}

\begin{proof}
Denote $\rho(t)=\P(X_t>0)$ for all $t>0$.\\
(a) Apply~\eqref{eq:error-law} to the interval $[0,t]$ with $n=1$, to get
$\tau_t(X)\overset{d}{=} Ut \1{X_{tU}>0} + \tau_{t(1-U)}(Y)$, where 
$U\sim \U(0,1)$ is independent of $Y$, which itself is a copy of $X$. Hence, 
\[\E\tau_t(X)
= t^{-1}\int_0^t (s\E \1{X_{s}>0} + \E\tau_{t-s}(Y))ds
= t^{-1}\int_0^t (s\rho(s)+\E\tau_{s}(X))ds,\]
where $\rho(s)=\P(X_s>0)$. Since $t\mapsto \tau_t$ is right-continuous and 
nondecreasing, so is $t\mapsto\E\tau_t$. The integral equation in the display 
above, the continuity of $\rho(t)$ for $t>0$ and a bootstrap argument imply 
that $t\mapsto\E\tau_t(X)$ is absolutely continuous with a derivative, say $h$. 
Put differently, we have $\E\tau_t(X)=\int_0^th(s)ds$ for all $t>0$.
Multiplying the equality in the display by $t$ and applying integration by 
parts yields $\int_0^t sh(s)ds=\int_0^ts\rho(s)ds$ for all $t>0$. Hence the 
integrands must agree a.e. with respect to the Lebesgue measure. In particular, 
$\E\tau_t=\int_0^th(s)ds=\int_0^t\rho(s)ds$ as desired.\\
(b) By Theorem~\ref{thm:error}, conditional on $L_n$, we have
$\delta_n^\SB\overset{d}{=}\tau_{L_n}(Y)-L_n\cdot \1{Y_{L_n}>0}$. Hence, by (a),
\begin{equation}\label{eq:E_delta}
\E[\delta_n^\SB|L_n]=\int_0^{L_n}\rho(s)ds-L_n\rho(L_n).
\end{equation} 
Since $L_n\searrow 0$ as $n\to\infty$, the assumption in (b) 
and~\eqref{eq:E_delta} imply that $\E[\delta_n^\SB|L_n]/L_n\to0$ a.s. as 
$n\to\infty$. Jensen's inequality applied to $x\mapsto |x|$ and the inequality 
$|\delta_n^\SB/L_n|\leq 1$ from Theorem~\ref{thm:error} imply that 
$|\E[\delta_n^\SB|L_n]/L_n|\leq\E[|\delta_n^\SB|/L_n|L_n]\leq 1$. Hence,
the dominated convergence theorem~\cite[Thm~1.21]{MR1876169} gives
$\E[\delta_n^\SB/L_n]=\E[\E[\delta_n^\SB|L_n]/L_n]\to0$ as $n\to\infty$.\\
(c) Since the assumption implies that of (b), the conclusion of (b) holds.
Moreover, by (b), \[\lim_{n\to\infty}\E[\delta_n/L_n|L_n]=
\lim_{n\to\infty}\E\big[\delta_n^\SB/L_n+\1{Y_{L_n}>0}\big|L_n\big]
=\lim_{n\to\infty}\rho(L_n)=\rho_0\quad\text{a.s.}\] 
Hence the dominated convergence theorem, applied as in the proof of (b), gives the result.\\
(d) Since $\rho(t)=\rho_0$ for all $t\in[0,T]$, the right-hand side  
in~\eqref{eq:E_delta} equals $0$ a.s., as claimed. Similarly, we have
$\E[\delta_n|L_n]=\E[\delta_n^\SB+L_n\cdot \1{Y_{L_n}>0}|L_n]=L_n\rho_0$ a.s.
\end{proof}

\begin{proof}[Proof of Corollary~\ref{cor:zoom}]

We assumethe existence of a function $a$ on the positive reals, such that 
$(X_{t\delta}/a(\delta))_{t\geq0}$ converges weakly to some process 
$(Z_t)_{t\geq0}$ as $\delta\searrow0$ in the sense of finite-dimensional 
distributions. It is known that the limiting process is then 
self-similar~\cite[Thm~8.5.2]{MR1015093} and thus $\alpha$-stable and the 
function $a$ is regularly varying with index $1/\alpha\in[2,\infty)$. Moreover, 
the convergence extends to the Skorokhod space 
$\mathcal{D}[0,\infty)$~\cite[Cor.~VII.3.6]{MR1943877}. 
(For a detailed description of $a$ and the limit criteria 
see~\cite[Thm~2]{MR3784492}.) 

Note that $Z^{\delta}=(Y_{t\delta}/a(\delta))_{t\in[0,1]}$ converges to 
$Z=(Z_t)_{t\in[0,1]}$ in $\mathcal{D}[0,1]$ and that 
$\tau_1(Z^{\delta})=\tau_{\delta}(Z)/\delta$. It is well known that the supremum 
mapping $x\mapsto\sup_{t\in[0,1]}x_{t}$ and the projection $x\mapsto x_1$ 
are continuous a.s. with respect to the law of $Y$. Next, since the time of the 
maximum of a stable process $(Z_{t}\vee Z_{t-})_{t\in[0,1]}$ is a.s. unique, 
then $\tau_{1}$ is a.s. continuous with respect to the law of $Z$ 
(see e.g.~\cite[Lem. 14.12]{MR1876169}). Thus, as $\delta\searrow0$, 
this yields
\[
\chi^\delta
=(Y_\delta/a(\delta),\ov{Y}_\delta/a(\delta),\tau_\delta(Y)/\delta) 
=(Z_1^\delta,\ov{Z}_1^\delta,\tau_1(Z^\delta))
\overset{d}{\to}
(Z_1,\ov{Z}_1,\tau_1(Z))
=\chi^0.
\]

By the equality in law given in~\eqref{eq:error-law}, we obtain 
\begin{equation}
\label{eq:error_weak_lim}
(Y_{L_{n}}/a(L_{n}),\Delta_{n}/a(L_{n}),\delta_n/L_n) \overset{d}{=}
(Y_{L_{n}}/a(L_{n}),\ov{Y}_{L_{n}}/a(L_n),\tau_{L_n}(Y)/L_n).
\end{equation}
Hence, the result will follow if we prove that 
$\chi^{L_n}\overset{d}{\to}\chi^0$. Recall that the weak convergence is 
equivalent to $\E f(\chi^\delta)\to\E f(\chi^0)$ as $\delta\searrow0$ 
for every bounded and continuous $f$. Since $\ell$ and $Y$ are independent 
and $L_{n}\to0$ a.s., conditional on the sequence $(L_n)_{n\in\N}$ we get 
$\E [f(\chi^{L_n})|L_n]\to\E f(\chi^0)$. The sequence of random variables 
$(\E [f(\chi^{L_n})|L_n])_{n\in\N}$ is bounded (since $f$ is) and converges to 
$\E f(\chi^0)$ a.s. Hence, by the dominated convergence theorem, 
it converges in $L^1$, implying $\chi^{L_n}\overset{d}{\to}\chi^0$. 
Hence, the weak limit holds for the left-hand side of~\eqref{eq:error_weak_lim}, 
which yields Corollary~\ref{cor:zoom}.
\end{proof}

\subsection{Convergence in $L^p$ and the proof of Theorem~\ref{thm:Lp}\label{subsec:proofs-Lp}}
%

Recall that $(\sigma^2,\nu,b)$ is the generating triplet of $X$ associated with 
the cutoff function $x\mapsto \1{|x|<1}$
(see~\cite[Ch.~2,~Def.~8.2]{MR3185174}). 
The moments of the L\'evy measure
$\nu$ at infinity 
are linked with 
the moments of $X^+_t$ and $\ov X_t$ for any $t>0$ as follows. 
By dominating $X$ path-wise with a L\'evy process 
$Z$ equal to $X$ with its jumps in $(-\infty,-1]$ removed and 
applying~\cite[Thm~25.3]{MR3185174} to $Z$, we find that, 
for any $p>0$, the conditions
$I_+^p<\infty$ and $E^p_+<\infty$ (see~\eqref{def:I_pm} and~\eqref{def:E_p} for definition) 
imply $\E\left[(X_{t}^+)^p\right]<\infty$ and 
$\E\exp(p X_t^+)<\infty$, respectively, for all $t>0$. 
Similarly, by applying~\cite[Thm~25.18]{MR3185174} to $Z$ we obtain that 
$I_+^p<\infty$ and $E^p_+<\infty$ imply $\E[\ov{X}_t^p]<\infty$
and $\E\exp(p\ov X_{t})<\infty$, respectively. 


Let $\beta$ be the 
\textit{Blumenthal-Getoor index}~\cite{MR0123362}, defined as 
\begin{equation}\label{def:I0_beta}
\beta=\inf\{p>0:I_0^p<\infty\},\quad\text{where}\quad 
I_0^p=\int_{(-1,1)}|x|^p\nu(dx),\quad\text{for any }p\geq 0,
\end{equation}
and note that $\beta\in[0,2]$ since $I_0^2<\infty$. Moreover, $I_0^1<\infty$ if and 
only if the jumps of $X$ have finite variation, in which case we may define the 
natural drift $b_0=b-\int_{(-1,1)}x\nu(dx)$. Note that $I_0^p<\infty$ for any
$p>\beta$ but $I_0^\beta$ can be either finite or infinite. If $I_0^\beta=\infty$ we 
must have $\beta<2$ and can thus pick $\delta\in(0,2-\beta)$, satisfying 
$\beta+\delta<1$ whenever $\beta<1$, and  define
\begin{equation}\label{eq:BG+}
\beta_+ = \beta+\delta\cdot \1{I_0^\beta=\infty}\in[\beta,2].
\end{equation}
Note that $\beta_+$ is either equal to $\beta$ or arbitrarily close to it.  In either 
case we have $I_0^{\beta_+}<\infty$.

The main aim of the  present subsection is to  prove Theorem~\ref{thm:Lp} and
Propositions~\ref{prop:Lipschitz},~\ref{prop:LocLip}~\&~\ref{prop:barrier}. 
With this in mind, 
we first establish three lemmas and a corollary. 

\begin{lem}\label{lem:BG-bounds}
The L\'evy measure $\nu$ of $X$ satisfies the following for all 
$\kappa\in (0,1]$: 
\begin{equation}\label{eq:BG-bound-1}
	\ov{\nu}(\kappa)=\nu(\R\setminus(-\kappa,\kappa)) \leq \kappa^{-\beta_+}I_{0}^{\beta_+} + \ov{\nu}(1),\qquad
	\ov{\sigma}^2(\kappa)=\int_{(-\kappa,\kappa)}x^2\nu(dx) \leq \kappa^{2-\beta_+}I_{0}^{\beta_+}.
\end{equation}
Moreover the following inequalities hold: 
\begin{eqnarray}
\label{eq:BG-bound-2}
\int_{(-1,-\kappa]\cup[\kappa,1)}|x|^p\nu(dx) & \leq & \kappa^{-(\beta_+-p)^+}I_{0}^{\beta_+}, \quad \text{for $p\in\R$,}\\
\label{eq:BG-bound-3}
\int_{(-\kappa,\kappa)}|x|^p\nu(dx) & \leq & \kappa^{p-\beta_+}I_0^{\beta_+},  
\qquad\text{for $p\geq\beta_+$.}
\end{eqnarray}
\end{lem}

\begin{proof}
Multiplying the integrands by (I) $(|x|/\kappa)^{\beta_+}$, 
(II) $(\kappa/|x|)^{2-\beta_+}$, (III) $(|x|/\kappa)^{\beta_+-p}$ if $p\leq\beta_+$ 
or $|x|^{\beta_+-p}$ otherwise and (IV) $(\kappa/|x|)^{p-\beta_+}$, respectively, 
and extending the integration set to $(-1,1)$ yields the bounds. 
\end{proof}

Recall the definition in~\eqref{def:I_pm} of $I_+^p$ and $I_-^p$ for $p\geq0$. 
Denote $\lceil x\rceil=\inf\{m\in\Z:m\geq x\}$ for any $x\in\R$. Recall that the Stirling 
numbers of the second kind $\stirling{m}{k}$ arise in the formula for the moments of 
a Poisson random variable $H$ with mean $\mu\geq0$: for any $m\in\N$ we have 
\begin{equation} \label{eq:Poi_Strl}
\E \left[H^m \right] = \sum_{k=1}^m \stirling{m}{k} \mu^k, 
\qquad\text{where}\qquad
\stirling{m}{k} = \frac{1}{k!}\sum_{i=0}^k (-1)^i\binom{k}{i} (k-i)^m.
\end{equation}
In particular, we have
$\stirling{m}{0}=0$ for all $m\in\N$.
Throughout, 
we will use the following inequality
\begin{equation} \label{eq:P_mean}
\left(\sum_{k=1}^m x_i\right)^p\leq m^{(p-1)^+}\sum_{k=1}^m x_i^p,\quad \text{where $m\in\N$, $x_1,\ldots,x_m\geq0$ and $p\geq0$.}
\end{equation}
This inequality follows easily from the concavity of $x\mapsto x^p$ when $p<1$ and Jensen's inequality 
when $p\geq 1$.

\begin{lem}\label{lem:LevyMomBound} 
For all $t\in[0,T]$ and $p>0$, the condition $I_+^p<\infty$ implies
\begin{equation}\label{eq:p_mom_bound}
\E[\ov X_t^p]\leq m_X^p(t)= 4^{(p-1)^+}\big(C_{p,1}t^{p/\beta_+} + 
C_{p,2}t^{p/2} + C_{p,3}t^p + C_{p,4}t^{\min\{1,p/\beta_+\}}\big),
\end{equation}
where the constants $\{C_{p,i}\}_{i=1}^4$ are given by
\begin{equation}\label{eq:C_p}\begin{split}
C_{p,1}&=2^{(p-1)^+}T^{p-p/\beta_+}\big(I_0^{\beta_+}\big)^p 
+T^{-p/\beta_+}\Big(2^pT^{p/2}\big(I_0^{\beta_+}\big)^{p/2}\cdot \1{p\leq 2}\\
&+2(p^2/(p-1))^p\exp\big(TI_0^{\beta_+}-p\big)\cdot \1{p>2}\Big),\\
C_{p,2}&=|\sigma|^p\Gamma\Big(\frac{p+1}{2}\Big)\frac{2^{p/2}}{\sqrt{\pi}},\qquad
C_{p,3}=2^{(p-1)^+}\big(b^+\cdot \1{I_0^1=\infty}+b_0^+\cdot\1{I_0^1<\infty}\big)^p,\\
C_{p,4}&=T^{(1-p/\beta_+)^+}\left(I_+^p+ I'\right)
\sum_{k=1}^{\lceil p\rceil}\stirling{\lceil p\rceil}{k}
T^{k-1}\left(I' +\nu([1,\infty))\right)^{k-1},
\end{split}\end{equation}
where 
$I'=\int_{(0,1)}x^{\beta_+}\nu(dx)$
and
$\Gamma(\cdot)$ is the Gamma function.
Moreover,
 if $I_+^1<\infty$, then
\begin{equation}
\label{eq:ChenLaberton}
\E[\ov{X}_t]\leq|\sigma|\sqrt{\frac{2}{\pi}}\sqrt{t}+
\begin{cases}
(b^{+}+I_{+}^{1})t+2\sqrt{I_{0}^{2}}\sqrt{t}, & \beta_+=2,\\
(b^{+}+I_{+}^{1})t+2T^{-1/\beta_+}\Big(\sqrt{TI_{0}^{\beta_+}}+TI_{0}^{\beta_+}
\Big)t^{1/\beta_+}, & \beta_+\in(1,2),\\
\left(b_0^+ +\int_{(0,\infty)}x\nu(dx)\right)t, & \beta_+\leq 1.
\end{cases}
\end{equation}
\end{lem}

\begin{rem}
(i) The formula in~\eqref{eq:ChenLaberton} essentially follows 
from~\cite[Lem.~5.2.2 \&~Eq.~(5.2)]{MR2996014} for $\beta_+\in(1,2]$ and 
from~\cite[Prop.~3.4]{MR2867949} for $\beta_+\leq 1$. 
A new proof of~\eqref{eq:ChenLaberton} given below 
is based on the methodology used to establish a more general inequality in~\eqref{eq:p_mom_bound}. 
Moreover, the dominant powers of $t$ in both bounds~\eqref{eq:p_mom_bound}  and~\eqref{eq:ChenLaberton}
coincide
in the case $p=1$ 
with slightly better constants in~\eqref{eq:ChenLaberton}.
The estimate in~\eqref{eq:p_mom_bound} works for all $p>0$
and is for the reasons of clarity applied 
in the proofs that follow even in the case $p=1$.\\
(ii) Note that $C_{p,2}=0$ if $\sigma=0$ and, if $X$ is spectrally negative, we 
have $C_{p,4}=0$. \\
(iii) The constants in~\eqref{eq:C_p} are well defined even if the assumption 
$I_+^p<\infty$ fails. The inequality in~\eqref{eq:p_mom_bound} holds trivially in 
this case since $C_{p,4}=\infty$.
\end{rem}

Recall that the \textit{L\'evy-It\^o decomposition}~\cite[Thms~19.2~\&~19.3]{MR3185174} of the L\'evy process $X$ 
\label{page:levy_ito_decomp} with generating triplet $(\sigma^2,\nu,b)$ at a 
level $\kappa\in(0,1]$ is given by $X_t=b_\kappa t+\sigma B_t+J^{1,\kappa}_t+
J^{2,\kappa}_t$ for all $t\geq0$, where $b_\kappa = b-\int_{(-1,1)\setminus 
(-\kappa,\kappa)} x\nu(dx)$ and $J^{1,\kappa}=(J^{1,\kappa}_t)_{t\geq0}$ (resp. 
$J^{2,\kappa}=(J^{2,\kappa}_t)_{t\geq0}$) is L\'evy with triplet 
$(0,\nu|_{(-\kappa,\kappa)},0)$ (resp. $(0,\nu|_{\R\setminus(-\kappa,\kappa)},
b-b_\kappa)$ - recall that we are using the cutoff function $x\mapsto
\1{|x|\leq 1}$) and $B=(B_t)_{t\geq0}$ is a standard Brownian motion. Moreover, 
the processes $B, J^{1,\kappa}, J^{2,\kappa}$ are independent, $J^{1,\kappa}$ is 
an $L^2$-bounded martingale with the magnitude of jumps at most $\kappa$ and 
$J^{2,\kappa}$ is a compound Poisson process with intensity $\ov\nu(\kappa)$ (see~\eqref{eq:BG-bound-1} above) and no drift. 

\begin{proof}
By the discussion above we have $\ov{X}_t \leq b_\kappa^+t + 
|\sigma|\ov{B}_t + \ov{J}^{1,\kappa}_t + \ov{J}^{2,\kappa}_t$. Then~\eqref{eq:P_mean} implies
\begin{equation}\label{eq:LevyItoBound}
\E\big[\ov{X}_t^p\big]\leq 4^{(p-1)^+}\big((b_\kappa^+)^pt^p + 
|\sigma|^p\E\big[\ov{B}_t^p\big] + \E\big[\big(\ov{J}^{1,\kappa}_t\big)^p\big] 
+ \E\big[\big(\ov{J}^{2,\kappa}_t\big)^p\big]),
\end{equation}
where $\ov{B}_t\overset{d}{=}|B_t|$ and so $\E\big[\ov{B}_t\big]=t^{p/2}
\Gamma\big(\frac{p+1}{2}\big)2^{p/2}/\sqrt{\pi}$~\cite[Prop.~13.13]{MR1876169},
which yields $C_{p,2}$ in all cases. By Lemma~\ref{lem:BG-bounds} we have
\[b_\kappa^+\leq\begin{cases}
b_0^+ +\int_{(-\kappa,\kappa)}|x|\nu(dx)\leq
b_0^+ +\kappa^{1-\beta_+}I_0^{\beta_+},
&I_0^1<\infty\quad(\text{i.e. }\beta_+\leq 1)\\
b^++\kappa^{1-\beta_+}I_0^{\beta_+},
&I_0^1=\infty\quad(\text{i.e. }\beta_+> 1).
\end{cases}\]
Hence, by~\eqref{eq:P_mean}, we obtain
\begin{equation}\label{eq:bound_bk}\begin{split}
(b_\kappa^+)^p&\leq \big(\kappa^{1-\beta_+}I_0^{\beta_+} + 
\1{I_0^1=\infty}b^++\1{I_0^1<\infty}b_0^+\big)^p\\
&\leq 2^{(p-1)^+}\big(\kappa^{p-p\beta_+}\big(I_0^{\beta_+}\big)^p + 
\1{I_0^1=\infty}(b^+)^p+\1{I_0^1<\infty}(b_0^+)^p\big).
\end{split}\end{equation}
$\ov{J}^{2,\kappa}_t$ is dominated by the sum of the positive jumps of 
$J^{2,\kappa}$ over the interval $[0,t]$, which has the same law as 
$\sum_{k=1}^{N_t}R_k$ for iid random variables $(R_k)_{k\in\N}$ with law 
$\nu|_{[\kappa,\infty)}/\nu([\kappa,\infty))$ and an independent Poisson 
random variable $N_t$ with mean $t\nu([\kappa,\infty))$. Note that since 
$N_t$ is a nonnegative integer, then 
$N_t^{(p-1)^++1}\leq N_t^{\lceil p\rceil}$. Hence, 
the independence between $(R_k)_{k\in\N}$ and $N_t$, the 
inequality $(\sum_{k=1}^{N_t}R_k)^p\leq N_t^{(p-1)^+}\sum_{k=1}^{N_t}R_k^p$ (which follows from~\eqref{eq:P_mean})
and~\eqref{eq:Poi_Strl}
yield
\[\begin{split}
\E\big[\big(\ov{J}^{2,\kappa}_t\big)^p\big] &\leq 
\E\bigg[\bigg(\sum_{k=1}^{N_t}R_k\bigg)^p\bigg] \leq 
\E\bigg[N_t^{(p-1)^+}\sum_{k=1}^{N_t}R_k^p\bigg] = 
\E[R_1^p]\E\big[N_t^{(p-1)^++1}\big] \leq 
\E[R_1^p]\E\big[N_t^{\lceil p\rceil}\big] \\ &=
\bigg(\int_{[\kappa,\infty)}x^p\frac{\nu(dx)}{\nu([\kappa,\infty))}\bigg)
\bigg(\sum_{k=1}^{\lceil p\rceil}
\stirling{\lceil p\rceil}{k}(t\nu([\kappa,\infty)))^k\bigg).
\end{split}\]
Denote 
$I'=\int_{(0,1)}x^{\beta_+}\nu(dx)$.
The first inequality in~\eqref{eq:BG-bound-1} and the bound in~\eqref{eq:BG-bound-2}
of Lemma~\ref{lem:BG-bounds} applied to $\nu|_{(0,\infty)}$ and the facts
$\kappa\leq 1$ and $t\leq T$ yield
\begin{equation}\label{eq:bound_J2}\begin{split}
 \E\big[\big(\ov{J}^{2,\kappa}_t\big)^p\big] 
& \leq
t\bigg(I_+^p+\int_{[\kappa,1)}x^p\nu(dx)\bigg)
\sum_{k=1}^{\lceil p\rceil}\stirling{\lceil p\rceil}{k}
\left(t\kappa^{-\beta_+}
I'+t\nu([1,\infty))\right)^{k-1}\\
&\leq
t\kappa^{-(\beta_+-p)^+}\left(I_+^p+I'\right)
\sum_{k=1}^{\lceil p\rceil}\stirling{\lceil p\rceil}{k}
\left(t\kappa^{-\beta_+}
I'+T\nu([1,\infty))\right)^{k-1}.
\end{split}\end{equation}

Assume $p\leq 2$. Jensen's inequality applied to the function $x\mapsto x^{2/p}$ 
and Doob's martingale inequality~\cite[Prop.~7.6]{MR1876169} applied to 
$J^{1,\kappa}$ yield
\begin{equation}\label{bound_J1_1}
\E\big[\big(\ov{J}^{1,\kappa}_t\big)^p\big] \leq\E\big[\big(\ov{J}^{1,\kappa}_t\big)^2\big]^{p/2}
\leq 2^p\E\big[(J^{1,\kappa}_t)^2\big]^{p/2}
=2^p\left(\ov\sigma(\kappa)\right)^pt^{p/2},\end{equation} 
where $\ov\sigma(\kappa)$ denotes the 
positive square root of $\ov\sigma^2(\kappa)$. Hence~\eqref{eq:LevyItoBound} 
for $p=1$, the first inequality in~\eqref{eq:bound_J2} and the estimate in~\eqref{bound_J1_1} give
\begin{equation}\label{eq:LevyItoBound2}
\E \ov{X}_t \leq \Big(b_\kappa^+ + \int_{[\kappa,1)}x\nu(dx)+ I_+^1\Big)t + \Big(|\sigma|\sqrt{\frac{2}{\pi}} + 2 \ov{\sigma}(\kappa) \Big)\sqrt{t}.
\end{equation}
If $\beta_+=2$, then taking $\kappa=1$ in~\eqref{eq:LevyItoBound2} yields the
first formula in~\eqref{eq:ChenLaberton}. If $\beta_+\leq1$ then
$I_0^1<\infty$. Letting $\kappa\to0$ in~\eqref{eq:LevyItoBound2}
we obtain the third formula in~\eqref{eq:ChenLaberton}. Set 
$\kappa=(t/T)^{1/\beta_+}$ and apply Lemma~\ref{lem:BG-bounds} to get
$t\ov{\sigma}^2(\kappa)\leq t^{2/\beta_+}T^{1-2/\beta_+}I_0^{\beta_+}$. Hence
$t\int_{[\kappa,1)}x\nu(dx)\leq t^{1/\beta_+}T^{1-1/\beta_+}I_0^{\beta_+}$,
and~\eqref{eq:bound_bk}~\&~\eqref{eq:LevyItoBound2} yield the second formula 
in~\eqref{eq:ChenLaberton}, completing the proof of~\eqref{eq:ChenLaberton}. 
To prove~\eqref{eq:p_mom_bound} for general $p\in(0,2]$, we again set 
$\kappa=(t/T)^{1/\beta_+}$ and use the inequalities $t\leq T$  
and~\eqref{eq:bound_bk}--\eqref{bound_J1_1} as before. 
More specifically, (I)~\eqref{eq:bound_bk}, 
(II)~\eqref{eq:bound_J2} and (III)~\eqref{eq:bound_bk} ~\&~\eqref{bound_J1_1} 
establish the values of (I) $C_{p,3}$, (II) $C_{p,4}$ and (III) $C_{p,1}$, respectively. 
This concludes the proof for the case $p\leq 2$.

Assume $p>2$. 
The only bound from the case $p\leq 2$ above that does not apply in this case 
is the one on $\E[(\ov J_t^{1,\kappa})^p]$. Doob's martingale inequality and 
the bound $|x|^p\leq (p/e)^pe^{|x|}$ for all $x\in\R$ yield
\[
\E\big[\big(\ov{J}^{1,\kappa}_t\big)^p\big]
\leq\Big(\frac{p}{p-1}\Big)^p\E\big[|J^{1,\kappa}_t|^p\big]
=\Big(\frac{\kappa p}{p-1}\Big)^p\E\big[(\kappa^{-1}|J^{1,\kappa}_t|)^p\big]
\leq\Big(\frac{\kappa p^2/e}{p-1}\Big)^p
\E\big[e^{\kappa^{-1}|J^{1,\kappa}_t|}\big].
\]
Note $\E\big[e^{\kappa^{-1}|J^{1,\kappa}_t|}\big]\leq \E\big[e^{\kappa^{-1}J^{1,\kappa}_t}+e^{-\kappa^{-1}J^{1,\kappa}_t}\big]
=e^{t\psi_\kappa(\kappa^{-1})}+e^{t\psi_\kappa(-\kappa^{-1})}$, 
where $\psi_\kappa$ is the L\'evy-Khintchine exponent of $J^{1,\kappa}_1$,  
i.e. $\psi_\kappa(u)=\int_{(-\kappa,\kappa)}(e^{ux}-1-ux)\nu(dx)$ for 
$u\in\R$. Using the elementary bound $e^x-1-x\leq x^2$ for all $|x|\leq 1$
and~\eqref{eq:BG-bound-1}, we find that 
$\psi_\kappa(u)\leq u^2\ov\sigma^2(\kappa)
\le u^2\kappa^{2-\beta_+}I_0^{\beta_+}$ for $|u|\leq\kappa^{-1}$. 
By setting $\kappa=(t/T)^{1/\beta_+}$, we obtain
\begin{equation}\label{eq:bound_J1_2}
\E\big[\big(\ov{J}^{1,\kappa}_t\big)^p\big]
\leq 2\Big(\frac{\kappa p^2/e}{p-1}\Big)^p
e^{t\kappa^{-\beta_+}I_0^{\beta_+}}
=2t^{p/\beta_+}T^{-p/\beta_+}
\Big(\frac{p^2}{p-1}\Big)^pe^{TI_0^{\beta_+}-p}.
\end{equation}
As before we obtain~\eqref{eq:p_mom_bound} as follows: (I)~\eqref{eq:bound_bk}, (II)~\eqref{eq:bound_J2} and 
(III)~\eqref{eq:bound_bk}~\&~\eqref{eq:bound_J1_2} establish the values of 
(I) $C_{p,3}$, (II) $C_{p,4}$ and (III) $C_{p,1}$, respectively,
which completes the proof.
\end{proof}

Recall that $\beta$, $I_0^1$ and  
$\beta_+$ are defined in~\eqref{def:I0_beta} and~\eqref{eq:BG+} above.
To describe the dominant power (as $t\downarrow0$) in the preceding results, 
define $\alpha\in[\beta,2]$ and $\alpha_+\in[\beta_+,2]$ by
\begin{equation}\label{eq:alpha}
\alpha = 2\cdot \1{\sigma\neq0} +\1{\sigma=0}\begin{cases}
1, & I_0^1<\infty\text{ and }b_0\neq0\\
\beta, & \text{otherwise},
\end{cases}
\quad\text{and}\qquad
\alpha_+ = \alpha+(\beta_+-\beta)\cdot \1{\alpha=\beta}.
\end{equation}
Note that the index $\alpha$ agrees with the one in~\cite[Eq.~(2.5)]{ZoomIn} and 
$\alpha_+>0$ since, by Assumption~\ref{asm:D}, $X$ is not compound Poisson 
with drift. Define
\begin{equation}\label{eq:eta}
\eta_p = 1+\1{p>\alpha}+\frac{p}{\alpha_+} \cdot \1{p\leq\alpha}\in(1,2],
\quad\text{for any}\quad p>0,
\end{equation}
and note that $\eta_p\geq 3/2$ for $p\geq 1$. 

\begin{rem}\label{rem:Cp_alpha}
(i) In Theorem~\ref{thm:Lp} and Propositions~\ref{prop:Lipschitz}, 
\ref{prop:LocLip} and~\ref{prop:barrier} we assumed that $p\geq 1$ for 
reasons of clarity. This is not a necessary assumption and the proofs can be 
made to work with minor modifications for any $p>0$. However, since $\eta_p
\to1$ as $p\to 0$, the convergence may become arbitrarily slow as $p\to 0$ 
(to be expected since $x^p\to 1$ as $p\to0$ for any $x>0$).\\
(ii) The constants $C_{p,2}$ and $C_{p,3}$ in Lemma~\ref{lem:LevyMomBound} above
satisfy the following: (a)  if $\alpha<2$, then 
$\sigma=0$  and hence $C_{p,2}=0$; (b) 
if $\alpha<1$, then  $I_0^1<\infty$ and $b_0=0$ and hence  $C_{p,3}=0$.
\end{rem}

\begin{cor}\label{cor:LevyMomBound}
Pick $p>0$, let $\{C_{p,i}\}_{i=1}^4$ be as in Lemma~\ref{lem:LevyMomBound} 
and define the constants $C_p(X)$ and $C_p^\ast(X)$ as follows:
\begin{equation}\label{eq:C_pX}\begin{split}
\frac{C_p(X)}{4^{(p-1)^+}}&=\begin{cases}
C_{p,1}T^{\frac{p}{\beta_+}-\frac{p}{\alpha_+}}+C_{p,2}+
C_{p,3}T^{p-\frac{p}{\alpha_+}}
+C_{p,4}T^{\min\{1,\frac{p}{\beta_+}\}-\frac{p}{\alpha_+}},
&p\leq\alpha,\\
C_{p,1}T^{\frac{p}{\beta_+}-1}+C_{p,2}T^{\frac{p}{2}-1}
+C_{p,3}T^{p-1}+C_{p,4},&p>\alpha, 
\end{cases}\\
C_p^\ast(X)&=C_p(X)\cdot \1{I_+^p<\infty}
+C_p(-X)\cdot \1{I_+^p=\infty}.
\end{split}\end{equation}
Then, if  $I_+^p<\infty$ (resp. $\min\{I_+^p,I_-^p\}<\infty$), the inequality 
\[\E[\ov{X}_t^p]\leq C_p(X)t^{\eta_p-1}\quad(\text{resp. }
\E[(\ov{X}_t-X_t^+)^p]\leq C_p^\ast(X)t^{\eta_p-1}).\]
holds for all $t\in[0,T]$.
\end{cor}

\begin{proof}
Since $\ov{X}_t-X_t^+=\min\{\ov{X}_t,\ov{X}_t-X_t\}$ is stochastically 
dominated by both $\ov{X}_t$ and $\ov{(-X)}_t$, then it suffices to prove the 
result for $\ov{X}_t$. (It is critical here, as seen in the definition 
of $C_p^\ast(X)$ in~\eqref{eq:C_pX}, that the definition of $\alpha$ is the
same for $X$ and $-X$.) Since $t^{q+r}\leq t^qT^r$ for $t\in[0,T]$ and 
$r\geq 0$, then it suffices to show that the exponent of $t$ in each term 
of~\eqref{eq:p_mom_bound} is at least $\eta_p-1$. By 
Remark~\ref{rem:Cp_alpha}(ii), this is trivially the case when 
$p\leq\alpha\leq\alpha_+\leq 2$. Recall that $\alpha_+$ is arbitrarily close 
(or equal) to $\alpha$. Hence, in the case $p>\alpha$, we may assume that 
$p>\alpha_+\geq\beta_+$ and use Remark~\ref{rem:Cp_alpha}(ii) to obtain the 
result and conclude the proof.
\end{proof}

\begin{rem}
If $X$ is spectrally negative (i.e. $\nu(\R_+)=0$), then $C_{p,4}=0$ and 
therefore $\E[\ov{X}_t^p]=\O(t^{p/\max\{1,\alpha_+\}})$ as $t\searrow 0$, 
implying the rate in~\cite[Lem.~6.5]{MR2867949}, which is the best in 
the literature to date for the spectrally negative case. 
In certain specific cases, Lemma~\ref{lem:LevyMomBound} implies a rate better 
than the one stated in Corollary~\ref{cor:LevyMomBound}. For example,
if  $\beta<1$ (thus $\beta_+<1$), $\sigma=0$, $I_+^p<\infty$ and the natural drift satisfies $b_0<0$ (thus  $\alpha=1$),
then by 
Lemma~\ref{lem:LevyMomBound} we have 
$\E[\ov{X}_t^p]=\O(t^{p/\beta_+})$ if $p\leq\beta$, 
which is sharper than the bound 
$\E[\ov{X}_t^p]=\O(t^{p})$
implied by
Corollary~\ref{cor:LevyMomBound}. 
Analogous improvements can be stated for $\ov{X}_t-X_t^+$, 
when either 
($I_+^p<\infty$ \& $b_0<0$) or  ($I_-^p<\infty$ \& $b_0>0$).
For the sake of presentation, throughout the paper we work with bounds in 
Corollary~\ref{cor:LevyMomBound}. 
\end{rem}

\begin{lem}\label{lem:small_t_to_Lp} 
Let $X$ be L\'evy process satisfying~\ref{asm:D} and let 
$\Delta_n$ and $\Delta_n^\SB$ be as in Theorem~\ref{thm:error}. 
If $\E[\ov X_t^p]\leq Ct^q$ (resp. $\E[(\ov X_t-X_t^+)^p]\leq Ct^q$) for some 
$C,q,p>0$ and all $t\in[0,T]$, then 
\[\E\big[\Delta_n^p\big]\leq CT^q(1+q)^{-n}\quad(\text{resp. }
\E\big[\big(\Delta_n^\SB\big)^p\big]\leq CT^q(1+q)^{-n})\text{ for all }n\in\N.\]
\end{lem}

\begin{proof}
By assumption and~\eqref{eq:error-law} in Theorem~\ref{thm:error}, 
we have $\E[\Delta_n^p|L_n]=\E[\ov Y_{L_n}^p|L_n]\leq CL_n^q$ and 
thus $\E[\Delta_n^p]\leq\E[CL_n^q]=CT^q(1+q)^{-n}$. The result for 
$\Delta_n^\SB$ is analogously proven.
\end{proof}

\begin{proof}[Proof of Theorem~\ref{thm:Lp}]
(a) By Theorem~\ref{thm:error}, the errors $\delta_n$ and $|\delta_n^\SB|$ are 
both bounded by $L_n$. Since $\E[L_n^p]=T^p(1+p)^{n}$, the claim follows.\\
(b) By Corollary~\ref{cor:LevyMomBound}, we may apply Lemma~\ref{lem:small_t_to_Lp} to obtain
part (b) of the theorem. Indeed, 
\begin{equation}\label{eq:explicit_bounds}
\E[\Delta_n^p]\leq C_p(X)T^{\eta_p-1}\eta_p^{-n}\quad\big(\text{resp. }
\E\big[\big(\Delta_n^\SB\big)^p\big]\leq C_p^\ast(X)T^{\eta_p-1}\eta_p^{-n}\big),
\end{equation}
where $C_p(X)$ (resp. $C_p^\ast(X)$) is as in~\eqref{eq:C_pX} in 
Corollary~\ref{cor:LevyMomBound}.
\end{proof}

For $p\geq 1$, let $\|\cdot\|_p$ denote the $p$-norm on $\R^d$. 
The $L^p$-Wasserstein 
distance between distributions $\mu_x$ and $\mu_y$ on 
$\R^d$
is defined as 
\begin{equation}\label{eq:Lp_Wd}\mathcal{W}_p(\mu_x,\mu_y)=
\inf_{X\sim \mu_x,Y\sim\mu_y}\E[\|\mathcal{X}-\mathcal{Y}\|_p^p]^{1/p},
\end{equation}
where the infimum is taken over all couplings of $(\mathcal{X},\mathcal{Y})$, 
such that $\mathcal{X}$ and $\mathcal{Y}$ follow the laws $\mu_x$ and $\mu_y$, 
respectively.

\begin{proof}[Proof of Corollary~\ref{cor:Wp}]
Recall that 
the coupling of $(\chi,\chi_n)$ in Subsection~\ref{subsec:coupling} 
yields
$\chi-\chi_n=(0,\Delta^\SB_n,\delta^\SB_n)$
(cf. Theorem~\ref{thm:error} above). 
By Theorem~\ref{thm:Lp}(a), Equation~\eqref{eq:explicit_bounds} and the inequality $1+p\geq 2\geq\eta_p$ (since   
$p\geq 1$), 
we have
\[\E[\|\chi-\chi_n\|_p^p]=
\E[|\Delta_n^\SB|^p+|\delta_n^\SB|^p]\leq 
C_p^\ast(X)T^{\eta_p-1}\eta_p^{-n}+T^p(1+p)^{-n}\leq
(C_p^\ast(X)T^{\eta_p-1}+T^p)\eta_p^{-n}.\] 
Since for any coupling of $(\chi,\chi_n)$ we have 
$\mathcal{W}_p(\mathcal{L}(\chi),\mathcal{L}(\chi_n))\leq \E[\|\chi-\chi_n\|_p^p]^{1/p}$,
the $L^p$-Wasserstein distance is bounded by $C'\eta_p^{-n/p}$, 
where the constant takes the form
\begin{equation}\label{eq:Wp}
C'=(C_p^\ast(X)T^{\eta_p-1}+T^p)^{1/p},
\end{equation}
concluding the proof.
\end{proof}

\subsection{Proofs of Propositions~\ref{prop:Lipschitz},~\ref{prop:LocLip} and~\ref{prop:barrier}}
The following result 
about the  tail probabilities of $\Delta_n$
(defined in Theorem~\ref{thm:error})
is key in the proofs below. 

\begin{lem}\label{lem:bound} 
Let $X$ be a L\'evy process satisfying \ref{asm:D}. Fix $p>0$ and $T>0$. 
Let $C_p(Z)$ be the constant in~\eqref{eq:C_pX} of Corollary~\ref{cor:LevyMomBound} for the L\'evy process $Z=X-J^{2,1}$, 
where $J^{2,1}$ is the compound Poisson process in the L\'evy-It\^o 
decomposition of $X$ (see the paragraph preceding the proof of 
Lemma~\ref{lem:LevyMomBound}). Using the notation 
$\ov\nu(1)=\nu(\R\setminus(-1,1))$, for any $r,p>0$, we have 
\begin{eqnarray}
\P\big(\Delta_{n}\geq r\big) & \leq & \ov{\nu}(1)T2^{-n}
	+r^{-p}C_p(Z)T^{\eta_p-1}\eta_p^{-n},\label{eq:ErrorTailBound}\\
\E\big[\min\{\Delta_{n},r\}^p\big] & \leq & r^p\ov{\nu}(1)T2^{-n}
	+C_p(Z)T^{\eta_p-1}\eta_p^{-n}.\label{eq:ErrorTruncE}
\end{eqnarray}
\end{lem}

\begin{proof}
Since 
$ \P\big(\Delta_{n}\geq r\big)=\P\big(\min\{\Delta_{n},r\}^p\geq r^p\big) \leq \E\big[\min\{\Delta_{n},r\}^p\big]/r^p $
by Markov's inequality, 
we only need to prove~\eqref{eq:ErrorTruncE}.

Let $Y$ be as in Theorem~\ref{thm:error}. Pick any $t>0$.
Let $A$ be the event on which $J^{2,1}$ does not have a jump on the interval 
$[0,t]$. Then $\P(A)=e^{-\ov\nu(1)t}\leq 1- \ov\nu(1)t$, or equivalently 
$P(A^c)\leq \ov{\nu}(1)t$. By Corollary~\ref{cor:LevyMomBound} applied to 
$Z$ we have $\E\big[\ov{Z}_t^p\big]\leq C_p(Z)t^{\eta_p-1}$. 
Since $\ov X_t=\ov Z_t$ a.s. on the event $A$ we get $\min\{\ov X_t,r\}^p 
\leq r^p\cdot \mathbbm{1}_{A^c} +\ov Z_t^p \cdot \mathbbm{1}_A\leq r^p\cdot \mathbbm{1}_{A^c} +\ov Z_t^p $,
implying
\[
\E\left[ \min\{\ov X_t,r\}^p\right]
\leq r^p \ov{\nu}(1)t + C_p(Z)t^{\eta_p-1}.
\]
This inequality,  Theorem~\ref{thm:error}, $\E [L_n]=T 2^{-n}$ and the equality in law $X\overset{d}{=}Y$ 
imply~\eqref{eq:ErrorTruncE}:
$\E\big[\min\{\Delta_n,r\}^p\big] =  
	\E \left[\E\big[\min\{\ov Y_{L_n},r\}^p|L_n\big]\right]
\leq \E[r^p \ov{\nu}(1)L_n + C_p(Z)L_n^{\eta_p-1}]$.
\end{proof}

\begin{proof}[Proof of Proposition~\ref{prop:Lipschitz}]
\label{proof:Lip}
Assume first $\|g\|_\infty<\infty$. Since $\min\{a+b,c\}\leq\min\{a,c\}+b$ for 
all $a,b,c\geq0$, we have  
\[|g(x,y,t)-g(x,y^{\prime},t^{\prime})|
\leq\min\{K|y-y^{\prime}|,\|2g\|_\infty\}+K|t-t'|.\]
Recall that the output of~\nameref{alg:SBA} is a copy of $\chi_n$.
Since, by Theorem~\ref{thm:error}, we a.s. have 
$0\leq\Delta^\SB_n\leq\Delta_n$ 
and $|\delta_n^\SB|\leq L_n$, by~\eqref{eq:P_mean} and~\eqref{eq:ErrorTruncE} we obtain
\[\begin{split}
\E[|g(\chi)-g(\chi_n)|^p] & \leq 2^{(p-1)^+}\big(
\E[K^p\min\{\Delta_n,\|2g\|_\infty/K\}^p]+K^p\E[L_n^p]\big)\\
&\leq 2^{(p-1)^+}\big[\|2g\|_\infty^p\ov\nu(1)T2^{-n}+
K^p(C_p(Z)T^{\eta_p-1}\eta_p^{-n} + T^p(1+p)^{-n})\big],
\end{split}
\]
where $Z=X-J^{2,1}$. Now assume that $\min\{I_+^p,I_-^p\}<\infty$. 
Then, again by Theorems~\ref{thm:error}~\&~\ref{thm:Lp} and 
Equation~\eqref{eq:explicit_bounds}, we obtain
\[\begin{split}
\E[|g(\chi)-g(\chi_n)|^p] 
& \leq 2^{(p-1)^+}K^p(\E[\Delta_n^p]+\E[L_{n}^p])\\
& \leq 2^{(p-1)^+}K^p(C_p^\ast(X)T^{\eta_p-1}\eta_p^{-n}+T^p(1+p)^{-n}).
\end{split}\]
Since $\eta_p\leq 2\leq 1+p$ for $p\geq 1$, this yields the result: 
$\E[|g(\chi)-g(\chi_n)|]\leq C'\eta_p^{-n}$ for
\begin{equation}\label{eq:LipBound}
C'=2^{(p-1)^+}
\begin{cases}
\|2g\|_\infty^p\ov\nu(1)T+K^p(C_p(Z)T^{\eta_p-1} + T^p), & \|g\|_\infty<\infty,\\
K^p(C_p^\ast(X)T^{\eta_p-1}+T^p),& \|g\|_\infty=\infty.
\end{cases}
\end{equation}
The proof is thus complete. 
\end{proof}

\begin{proof}[Proof of Proposition~\ref{prop:LocLip}]
\label{proof:LocLip}
Recall that the second component of $\chi_n$ (resp. $\chi$) equals 
$\ov X_T -\Delta_n^\SB$ (resp. $\ov X_T$). Recall from Theorem~\ref{thm:error} 
that $|\delta_n^\SB|\leq L_n$. Since $0\leq\Delta^\SB_n\leq\Delta_n$, the 
locally Lipschitz property of $g$ implies:
\[ |g(\chi)-g(\chi_n)|\leq K (\Delta_n+L_n)e^{\lambda\ov{X}_T}.\]
From the definition of $q'$ we get $1/q' + 1/q=1$. Thus H\"older's 
inequality gives:
\begin{equation}\label{eq:prod_Exp}
\E\big[|g(\chi)-g(\chi_n)|^p\big]
\leq K^p \E\big[\big(\Delta_n+L_n\big)^{p q'}\big]^{\frac{1}{q'}}
\E\left[e^{\lambda pq\ov{X}_{T}}\right]^{\frac{1}{q}},
\end{equation}
where the second expectation on the right-hand side of~\eqref{eq:prod_Exp} is finite by assumption 
$E^{\lambda pq}_+<\infty$ and the argument in the first paragraph of 
Subsection~\ref{subsec:proofs-Lp} above.

We now estimate both expectations on the right-hand side of~\eqref{eq:prod_Exp}. 
Note that $I_+^r<\infty$ for all $r>0$ as $E^{\lambda pq}_+<\infty$. 
By~\eqref{eq:P_mean}, we have $\E\big[\big(\Delta_n+L_n\big)^{pq'}\big]
\leq 2^{(pq'-1)^+}\E\big[\Delta_n^{pq'}+L_n^{pq'}\big]$. Hence
Theorem~\ref{thm:Lp},~\eqref{eq:explicit_bounds} 
and the inequality 
$(x+y)^{1/q'}\leq x^{1/q'}+y^{1/q'}$ for $x,y\geq 0$ 
imply
\[ 
	\begin{split}	
		\E\big[\big(\Delta_n+L_n\big)^{pq'}\big]^{1/q'} & 
\leq 2^{(p-1/q')^+}
\big(C_{pq'}(X)T^{\eta_{pq'}-1}\eta_{pq'}^{-n}+T^{pq'}(1+pq')^{-n}\big)^{1/q'} \\
& \leq 2^{(p-1/q')^+}
\big(C_{pq'}(X)^{1/q'}T^{(\eta_{pq'}-1)/q'}\eta_{pq'}^{-n/q'}+T^{p}(1+pq')^{-n/q'}\big). 
	\end{split}	
\]

It remains to obtain an explicit bound for the expectation 
$\E[\exp(\lambda pq\ov{X}_T)]$. By removing all jumps 
smaller than $-1$ from $X$, we obtain a L\'evy process $Z$ 
with triplet $(\sigma^2, \nu|_{[-1,\infty)},b)$ that dominates
$X$ path-wise. 
Set $Z_t^\ast=\sup_{s\in [0,t]}|Z_s|$ 
and note
$Z_T^\ast\geq \ov{Z}_T\geq \ov{X}_T$.
Define the function $h:x\mapsto e^{\lambda pqx}-1$ on $\R$. Then, 
for any $c>0$, 
by Fubini's theorem  we have
\[\begin{split}\E[h(Z_T^\ast-c)]& \leq \E[h(Z_T^\ast-c)\1{Z_T^\ast>c}]
= \int_c^\infty \P(Z_T^\ast>z)h'(z-c)dz \\
&=\int_0^\infty \P(Z_T^\ast>z+c)h'(z)dz 
		\leq \int_0^\infty \frac{\P(|Z_T|>z)}{\P[Z_T^\ast\leq c/2]}h'(z)dz 
= \frac{\E[h(|Z_T|)]}{\P[Z_T^\ast\leq c/2]},
\end{split}\] 
where the second inequality holds by~\cite[p.~167,~Eq.~(25.15)]{MR3185174}.
Hence, we get
\begin{equation}	
\label{eq:Exp_bound}
	\E\left[e^{\lambda pq\ov{X}_T}\right]
\leq \E\left[e^{\lambda pqZ^*_T}\right]
= e^{\lambda pq c}\E[1+h(Z_T^\ast-c)]
\leq e^{\lambda pq c}\left(1+
\frac{\E\left[e^{\lambda pq|Z_T|}\right]-1}{
\P[Z_T^\ast\leq c/2]}\right).
	\end{equation}	
Using the L\'evy-Khintchine formula~\cite[Thm~25.17]{MR3185174} for the L\'evy process $Z$
we get 
$$\E[e^{\lambda pq|Z_T|}]\leq \E[e^{\lambda pqZ_T}]+
\E[e^{-\lambda pqZ_T}]=e^{T\Psi_Z(\lambda pq)}+e^{T\Psi_Z(-\lambda pq)},$$ 
where $\Psi_Z(u)=bu+\sigma^2u^2/2+\int_{[-1,\infty)}(e^{ux}-1-ux\1{x<1})\nu(dx)$ for $u\in(-\infty,\lambda p q]$. 
Markov's inequality implies
$\P[Z_T^\ast\leq c/2]\geq 1-(2/c)\E[Z_T^\ast]$.
Moreover, by Lemma~\ref{lem:LevyMomBound}, we have 
$$\E[Z_T^\ast]\leq
\E\left[\ov{Z}_T-\inf_{s\in[0,T]}Z_s\right]\leq m^\1Z(T)+m^\1{-Z}(T).$$
Hence, from~\eqref{eq:Exp_bound}, for any $c>(m^\1Z(T)+m^\1{-Z}(T))/2$ we get
\begin{equation*}
\E\left[e^{\lambda pq\ov{X}_T}\right]
\leq e^{\lambda pq c}\left(1+
\frac{e^{T\Psi_Z(\lambda pq)}+e^{T\Psi_Z(-\lambda pq)}-1}{
1-\frac{2}{c}(m^\1Z(T)+m^\1{-Z}(T))}\right).
\end{equation*}
Therefore, using~\eqref{eq:prod_Exp} and the inequalities
$\eta_{pq'}\leq 2\leq 1+pq'$ (as $pq'\geq 1$), 
we obtain the bound 
$\E\big[|g(\chi)-g(\chi_n)|^p\big]\leq C'\eta_{pq'}^{-n/q'}$, where
\begin{equation}\label{eq:LocLipBound}
C'= \frac{C_{pq'}(X)^{1/q'}T^{(\eta_{pq'}-1)/q'}+T^p}
{2^{-(p-1/q')^+}K^{-p}e^{-\lambda p c}}
\bigg(1+\frac{e^{T\Psi_Z(\lambda pq)}+e^{T\Psi_Z(-\lambda pq)}-1}
{1-\frac{2}{c}(m^\1Z(T)+m^\1{-Z}(T))}\bigg)^{1/q},
\end{equation}
the constant $C_{pq'}(X)$ is defined in~\eqref{eq:C_pX} and
$m^\1{Z}(T)$ and $m^\1{-Z}(T)$ are given in Lemma~\ref{lem:LevyMomBound}.
\end{proof}

\begin{rem}\label{rem:LocLip}
The rate $\eta_{pq'}^{-1/q'}$ in the bound of Proposition~\ref{prop:LocLip}
is smallest (as a function of $q$) for the largest $q$ satisfying the exponential moment condition
in Proposition~\ref{prop:LocLip}. Indeed,
let $r=pq'$ and note that, since $p$ is fixed, 
minimising $\eta_{pq'}^{-1/q'}$ in $q$ 
is equivalent to maximising $\eta_r^{1/r}$ in $r$. 
By~\eqref{eq:eta}, 
the function $r\mapsto \eta_r^{1/r}$ 
is decreasing and hence takes its maximal value 
at the smallest possible $r$
(i.e. largest possible $q$). 
\end{rem}

\begin{proof}[Proof of Proposition~\ref{prop:barrier}]
\label{proof:barrier}
Recall from Theorem~\ref{thm:error} that 
$0\leq\Delta^\SB_n\leq\Delta_n$. Let $\epsilon_n=
\eta_q^{-n/(\gamma+q)}$ and note 
\[\begin{split}
\E \left[\frac{|h(X_{t})|^p}{\| h\|_{\infty}^p}
\big|\1{\ov X_T-\Delta_n^\SB\leq x}-\1{\ov X_T\leq x}\big|^p\right] 
&\le \P(\ov X_T-\Delta_n^\SB\leq x<\ov{X}_T)\\
&\le \P(\ov X_T-\Delta_n\leq x<\ov{X}_T)\\
& = \P(\ov X_T-\Delta_n\leq x<\ov{X}_T-\epsilon_n)\\
&\qquad + \P(\ov X_T-\Delta_n\leq x<\ov{X}_T\leq x+\epsilon_n)\\
& \leq \P(\epsilon_n<\Delta_n)+ 
\P(x<\ov{X}_T\leq x+\epsilon_n).
\end{split}\]
By~\eqref{eq:ErrorTailBound} in Lemma~\ref{lem:bound} 
we have
\[\P(\epsilon_{n}<\Delta_n)\leq 
\ov{\nu}(1)T2^{-n}+\epsilon_n^{-q}C_q(Z)T^{\eta_q-1}\eta_q^{-n}
=\ov{\nu}(1)T2^{-n}+C_q(Z)T^{\eta_q-1}\eta_q^{-n\gamma/(\gamma+q)}.\]
The assumed H\"older continuity of the distribution function of $\ov X_T$ in 
Assumption~\ref{asm:H} implies that $\P(x<\ov{X}_T\leq x+\epsilon_n)
\leq K\epsilon_n^\gamma$. Given the formula for $C_q(Z)$ in~\eqref{eq:C_pX}, 
the constant 
\begin{equation}\label{eq:barrierBound}
C'=\|h\|_\infty^p (\ov\nu(1)T+C_q(Z)T^{\eta_q-1}+K),
\end{equation}
is explicit and satisfies $\E[|g(\chi)-g(\chi_n)|^p]\leq 
C' \eta_q^{-n\gamma/(\gamma+q)}$.
\end{proof}

\begin{rem}\label{rem:optimal_q}
Minimising the rate $\eta_q^{-\gamma/(\gamma+q)}$ as a function of $q$
in Proposition~\ref{prop:barrier} is somewhat involved. 
On the interval $(\alpha_+,\infty)$, the rate
$q\mapsto \eta_q^{-\gamma/(\gamma+q)}=2^{-\gamma/(\gamma+q)}$ is strictly increasing, so the optimal $q$ always lies in $(0,\alpha_+]$. On the interval $(0,\alpha_+]$
the problem is equivalent to maximising the map 
$r\mapsto e^{f(r)}=\eta_q^{\gamma/(\gamma+q)}$ on the interval $(0,1]$, where $r=\frac{q}{\alpha_+}\in(0,1]$ 
and $f:x\mapsto\log(1+x)/\big(1+\frac{\alpha_+}{\gamma}x\big)$. 
Since 
\[
\frac{\gamma}{\alpha_+}
	\Big(1+\frac{\alpha_+}{\gamma}x\Big)^2\frac{d}{dx}f(x)
=\frac{\frac{\gamma}{\alpha_+}-1}{1+x}-(\log(1+x)-1),
\] 
the critical point of $f$, obtained by solving for $s=\log(1+x)-1$ in 
$se^s=e^{-1}(\frac{\gamma}{\alpha_+}-1)$, is given by 
$r_0=e^{W(e^{-1}(\gamma/\alpha_+-1))+1}-1$, 
where $W$ is the Lambert $W$ function, defined as the inverse of $x\mapsto xe^x$. 
Since $f$ is increasing on $[0,r_0]$ and decreasing on $(r_0,\infty)$, then 
$r=\min\{r_0,1\}$ maximises $f|_{(0,1]}$, implying that the optimal $q$ equals 
\[q=\alpha_+\min\Big\{1,e^{W(e^{-1}(\gamma/\alpha_+-1))+1}-1\Big\}.\] 
In particular, the choice $q=\alpha_+$ is optimal if and only if 
$\gamma/\alpha_+\geq 2\log(2)-1=0.38629\ldots$, 
and leads to the bound $\O(2^{-n/(1+\alpha_+/\gamma)})$. Hence, if $\gamma=1$, the best 
bound in Proposition~\ref{prop:barrier} is $\O(2^{-n/(1+\alpha_+)})$. 
\end{rem}

\subsection{The proof of the central limit theorem}
\begin{proof}[Proof of Theorem~\ref{thm:CLT}]
\label{proof:CLT}
Recall $n_N=\lceil \log N/ \log(\eta_g^{2})\rceil$ and note that 
$1\geq\sqrt{N}\eta_g^{-n_N}\geq\eta_g^{-1}$. Hence Assumption~(b) yields 
\begin{equation}
\label{eq:Deterministic_remainder}
\sqrt{N} \E \Delta_{n_N,N}^g\to 0\qquad \text{as }N\to\infty.
\end{equation}

The coupling in Subsection~\ref{subsec:coupling}, used in Theorem~\ref{thm:error}, 
implies that for all $n\in\N$ the following relations between $\chi$ and the 
SBA $\chi_n$ in~\eqref{eq:SBA} hold a.s.: $Y_T=X_T$, $\ov X_T-\Delta_n^\SB\leq 
\ov X_T$ and $\tau_T-\delta_n^\SB\leq T$. Hence parts~(i) and~(ii) of 
Assumption~(a) imply that $g(\chi_n)$ and $g(\chi_n)^2$ are dominated by $\zeta
=G(X_T,\ov X_T,T)$ and $\zeta^2$, respectively. Since $\zeta$ and $\zeta^2$ are 
integrable by assumption, the dominated convergence theorem yields 
\begin{equation}\label{eq:Var_conv}
\V[g(\chi_n)]=\E[g(\chi_n)^2]-[\E g(\chi_n)]^2\to 
\E[g(\chi)^2]-[\E g(\chi)]^2=\V[g(\chi)]\qquad\text{as $n\to\infty$.}
\end{equation}

Recall that $(\chi^i_n)_{i\in\{1,\ldots,N\}}$ is the output produced by $N$ 
independent runs of~\nameref{alg:SBA} using $n$ steps. Define the 
normalised centred random variables 
\[
\zeta_{i,N}=\left(g\big(\chi^i_{n_N}\big)-
\E g\big(\chi^i_{n_N}\big)\right)/\sqrt{N \V[g(\chi)]}, 
\quad\text{where }i\in\{1,\ldots,N\}.
\]
Hence~\eqref{eq:Var_conv} implies $\sum_{i=1}^N\E \zeta_{i,N}^2
= \V[g(\chi)]^{-1} (1/N)\sum_{i=1}^N \V [g(\chi^i_{n_N})]\to 1$ as $N\to\infty$. 
Moreover, we have 
\[
\sum_{i=1}^N \zeta_{i,N} = \sqrt{N/\V[g(\chi)]} \Delta^g_{n_N,N}+o(1) 
\qquad \text{as }N\to\infty,
\]
where $o(1)$ is a deterministic sequence, proportional to the one 
in~\eqref{eq:Deterministic_remainder}. Hence, \eqref{eq:CLT} holds if and only 
if $\sum_{i=1}^N\zeta_{i,N}\overset{d}{\to} N(0,1)$ as $N\to\infty$. 

To conclude the proof, we shall use Lindeberg's CLT~\cite[Thm~5.12]{MR1876169}, 
for which it remains to prove that Lindeberg's condition holds, i.e. 
$\sum_{i=1}^N\E[\zeta_{i,N}^2\1{\zeta_{i,N}>r}]\to0$ as $N\to\infty$ for all 
$r>0$. By the coupling from the second paragraph of this proof, we find 
$|g(\chi_n^i)|\leq |\zeta_i|$ for all $i\in\{1,\ldots,N\}$ and $n\in\N$, where 
$(\zeta_i)_{i\in\{1,\ldots,N\}}$ are iid with the law equal to $G(X_T,\ov 
X_T,T)$. Crucially, $\zeta_i$ does not depend on the number of steps $n_N$ in 
the~\nameref{alg:SBA}. Moreover, note that iid random variables 
$\xi_i=(|\zeta_i| +\E |\zeta_i|)$ satisfy $\E \xi_i^2 <\infty$ and 
$|\zeta_{i,N}|\leq \xi_i/\sqrt{N \V[g(\chi)]}$ for any $i\in\{1,\ldots,N\}$.
Hence we find
\[\V [g(\chi)]\sum_{i=1}^N \E [\zeta_{i,N}^2\1{\zeta_{i,N}>r}]\leq
\sum_{i=1}^N\frac{1}{N} \E \left[\xi_i^2 \1{\xi_i>rN\V (g(\chi))}\right]=
\E \left[\xi_1^2 \1{\xi_1>rN\V (g(\chi))}\right]\to0\]
as $N\to\infty$, implying Lindeberg's condition and our theorem. 
\end{proof}

\begin{rem}\label{rem:CLT-G}
Identifying the appropriate $G$ in Theorem~\ref{thm:CLT} is usually simple. 
For instance, the following choices of $G$  can be made in the contexts of interest. 
\begin{itemize}
\item[(a)] Let $g$ be Lipschitz (as in Proposition~\ref{prop:Lipschitz}). 
Then we can take 
\begin{itemize}
\item[(i)] $G(x,y,t)=\left\Vert g\right\Vert_\infty$, if $\left\Vert g\right\Vert_\infty<\infty$; 
\item[(ii)] $G(x,y,t)=|g(x,y,t)|+2K(y+t)$, if $I_+^2<\infty$. 
\end{itemize}
\item[(b)] Let $g$ be locally Lipschitz with the Lipschitz constant exponentially increasing at rate $\lambda>0$ (as in Proposition~\ref{prop:LocLip}).
Then we can take 
\begin{itemize}
\item[(i)] $G(x,y,t)=Ke^{\lambda y}$, if $g(x,y,t)\leq Ke^{\lambda y}$ and $E_+^{2\lambda}<\infty$ (lookback and hindsight options fall in this category);
\item[(ii)] $G(x,y,t)=|g(x,y,t)|+2K(y+t)e^{\lambda y}$ if $E_+^{2\lambda q}<\infty$ for some $q>1$. 
\end{itemize}
\item[(c)] If $g$ is a barrier option (as in Proposition~\ref{prop:barrier}), then take $G(x,y,t)=\left\Vert g\right\Vert_\infty$.
\end{itemize}
\end{rem}

\begin{rem}
\label{rem:CLT_fixed_n}
If we are prepared to centre, it is possible to apply the standard 
iid CLT to the estimator based on~\nameref{alg:SBA}.
Indeed, for fixed $n$, assuming  $\V[P_n]<\infty$ where 
$P_n=g(\chi_n)$, the classical CLT yields 
\[
\frac{1}{\sqrt{N\V[P_n]}}\sum_{i=1}^N (P_n^i-\E P_n)\overset{d}{\to}N(0,1)\qquad\text{as $N\to\infty$.}
\]
In contrast, the gist of Theorem~\ref{thm:CLT} is that one need not centre 
the sample with a function of $n$, which itself depends on the sample. 
\end{rem}

\appendix

\section{
MLMC and the debiasing}

\subsection{$\O$ and $o$}\label{sec:O_o}
The following standard notation is used throughout the paper: for functions $f,g:\N\to(0,\infty)$  
we write 
$f(n)=\O(g(n))$ (resp. $f(n)=o(g(n))$) 
as $n\to\infty$ 
if 
$\limsup_{n\to\infty}f(n)/g(n)$ is finite (resp. $0$).
Put differently, $f(n)=\O(g(n))$ is equivalent to $f(n)$ being bounded above by $C_0 g(n)$
for some constant $C_0>0$ and all $n\in\N$. In particular, $f(n)=\O(g(n))$ does \textit{not} imply that $f$ and $g$ decay at the same rate.
We also write 
$f(\epsilon)=\O(g(\epsilon))$ (resp. $f(\epsilon)=o(g(\epsilon))$) as 
$\epsilon\downarrow0$, for functions $f,g:(0,\infty)\to(0,\infty)$ if
$\limsup_{\epsilon\downarrow0}f(\epsilon)/g(\epsilon)$ is finite (resp. $0$).

\subsection{ML}\label{sec:MLMC_1}
We start by recalling a version of~\cite[Thm~1]{MR2835612}.

\begin{thm}\label{thm:MLMC}
Consider a family of square integrable random variables $P,P_1,P_2,\ldots$ and $P_0=0$. 
Let $\{D_k^i\}_{k,i\in\N}$ be independent with 
$D_k^i\overset{d}{=}P_k-P_{k-1}$ for all $k,i\in\N$. 
Assume that for some $q_1\geq (q_2\wedge q_3)/2>0$ and all $n\in\N$ we have
\begin{itemize}
\item[\normalfont(a)] $|\E P-\E P_n|\leq c_1 2^{-n q_1}$,
\item[\normalfont(b)] $\V [P_{n+1}-P_n]\leq c_2 2^{-n q_2}$, 
\item[\normalfont(c)] the expected computational cost $\C(n)$ of constructing a single 
sample of $(P_n,P_{n-1})$ is bounded by $c_3 2^{n q_3}$,
\end{itemize}
where $c_1,c_2,c_3$ are positive constants.
Then for every $\epsilon>0$ there exist $n,N_1,\ldots,N_n\in\N$ (see Remark~\ref{rem:MLMC}(i) below for explicit formulae) such that
the estimator
\begin{equation}\label{eq:MLMC_estim}
\hat{P} = \sum_{k=1}^n \frac{1}{N_k}
\sum_{i=1}^{N_k}D_k^i\quad \text{is $L^2$-accurate at level $\epsilon$,}\quad
\E \big[(\hat{P} - \E P)^2\big]<\epsilon^2,
\end{equation}
and the computational complexity is of order
\[\C_{ML} (\epsilon)=\begin{cases}
\O(\epsilon^{-2}) &\text{ if }q_2>q_3,\\
\O(\epsilon^{-2}\log^2\epsilon) &\text{ if }q_2=q_3,\\
\O(\epsilon^{-2-(q_3-q_2)/q_1}) &\text{ if }q_2<q_3.
\end{cases}
\]
\end{thm}

\begin{rem}
\label{rem:MLMC}
(i) In~\cite{MR2835612}, the number of levels equals
$n = \lceil \log_2(\sqrt{2}c_1\epsilon^{-1})/q_1 \rceil$ and the number of samples at level 
for $k\in \{1,\ldots,n\}$
is
\begin{equation}\label{eq:MLMC_const}
N_k=
\begin{cases}
\lceil 2c_2\epsilon^{-2}2^{-(q_2+q_3)k/2}/(1-2^{-(q_2-q_3)/2}) \rceil &\text{ if }q_2>q_3,\\
\lceil 2c_2\epsilon^{-2}n2^{-q_3k} \rceil &\text{ if }q_2=q_3,\\
\lceil 2c_2\epsilon^{-2}2^{n(q_3-q_2)/2-(q_2+q_3)k/2}/(1-2^{-(q_3-q_2)/2}) \rceil &\text{ if }q_2<q_3.
\end{cases}
\end{equation}
Clearly, the number of levels $n$ is obtained from the bound on the bias in Assumption~(a), while the number of samples~\eqref{eq:MLMC_const} at levels 
$k\in\{1,\ldots,n\}$ are obtained from a simple constrained optimisation using 
the bounds on the variances and computational costs. In practice, if one has 
no access to the constants involved in the bounds in Assumptions~(a)--(c), 
one estimates them via Monte Carlo simulation for small $n$. In the setting 
of this paper this is the case for barrier options, see 
Proposition~\ref{prop:barrier} and the paragraphs succeeding it.\\
\noindent (ii) The coupling $(P_n,P_{n-1})$ that can be simulated, implicit in
Assumptions~(b) and~(c) of Theorem~\ref{thm:MLMC}, constitutes a crucial
extension of any MC algorithm necessary for an MLMC estimator to be define. It 
is clear from~(b) that a trivial independent coupling is undesirable in this 
context. In fact, typically, the optimal coupling (the one where 
$\V[P_{n+1}-P_n]$ equals the $L^2$-Wasserstein distance between the laws of 
$P_n-\E P_n$ and $P_{n+1}-\E P_{n+1}$, cf.~\eqref{eq:Lp_Wd} above) is very 
expensive (resp. impossible) to simulate, making the bound in~(c) very large 
(resp. infeasible). Hence a ``compromise'' coupling is needed. This is,
however, not the case for the problems analysed in this paper as the cost
scales only linearly in $n$. In contrast, Assumption (a) requires no specific 
coupling since $|\E P_n -\E P|$ only compares $P$ and $P_n$ through their means. 
Thus, $q_1$ may be computed using the optimal coupling, even if unavailable for 
simulation.
\end{rem}

\subsection{The debiasing}\label{sec:MLMC_2}
A certain random selection of the variables $\{D^k_n\}_{n,k\in\N}$
in Theorem~\ref{thm:MLMC} leads to an unbiased estimator for $\E P$ 
(see~\cite{MR2890424,MR3422533}). More precisely,
following~\cite[Thm~7]{MR3782809}, define the estimator
\begin{equation}\label{eq:Matti_estimator}
\hat{P}=\sum_{k=1}^\infty\frac{1}{\E N_k}\sum_{n=1}^{N_k}D_k^n,
\end{equation}
where the sequence of nonnegative random integers $(N_k)_{k\in\N}$, independent 
of $\{D_n^k\}_{n,k\in\N}$, satisfies $\E N_k>0$ for all $k\in\N$ and  
$\sum_{k=1}^\infty N_k<\infty$, i.e. $N_k=0$ for all sufficiently large indices. 
The sequence $(N_k)_{k\in\N}$ can be constructed as a deterministic functional 
of a finite sample of positive integers $(R_j)_{j=1}^N$ as follows:
(a) \emph{single term estimator} (STE): $N_k=\sum_{j=1}^N \1{R_j=k}$; and
(b) \emph{independent sum estimator} (ISE): $N_k=\sum_{j=1}^N\1{R_j\geq k}$ (see~\cite[Thms~3~\&~5]{MR3782809}). For instance, one may take 
$(R_n)_{n=1}^N$ to be iid with common distribution $p_n=\P[R=n]>0$, $n\in\N$. 
The computational complexities of STE and ISE are linked with the optimal choice 
for the law of $R$~\cite[Sec.~6]{MR3782809}. One of the choices analysed 
in~\cite{MR3782809} is that of the \emph{Uniform Stratified Estimator} (USE), 
described in Theorem~\ref{thm:RE_AV} below. 
Let $F_R:x\mapsto\sum_{n=1}^{\lfloor x \rfloor}p_n$, $x>0$, be the distribution 
function of $R$ (where we denote $\lfloor x\rfloor=\sup\{k\in\Z:
k\leq x\}$), let $F_R^{-1}:u\mapsto\inf\{k\in\N:F_R(k)\geq u\}$, $u\in(0,1)$, 
be its generalised inverse. Put $\ov p_n = 1-F_R(n-1)$ for $n\in\N$ and recall 
$\C(n)$ defined in Theorem~\ref{thm:MLMC} above.

\begin{thm}[{\cite[Thm~19]{MR3782809}}]\label{thm:RE_AV} 
For some fixed $N\in\N$ let $(U_k)_{k\in\{1,\ldots,N\}}$ be independent with
$U_{k}\sim \U(\frac{k-1}{N},\frac{k}{N})$ and put $R_{k}=F_R^{-1}(U_{k})$ 
for $k\in\{1,\ldots,N\}$.\\
\quad{\normalfont(a)} Assume $\sum_{n=1}^\infty\E[(P_n-P_{n-1})^2]/p_n<\infty$ 
and define $N_j=\sum_{k=1}^N\1{R_k=j}$ whose mean is $\E N_j=Np_j$. Then $\hat 
P_{\STE,N}$ in~\eqref{eq:Matti_estimator} is the uniform stratified STE satisfying 
$\E\hat P_{\STE,N}=\E P$ and $\lim_{N\to\infty}N\V[\hat{P}_{\STE,N}]=\sum_{n=1}^\infty 
\V[P_n-P_{n-1}]/p_n$ with cost $N\sum_{n=1}^\infty p_n\C(n)$.\\
\quad{\normalfont(b)} Assume $\sum_{n=1}^\infty \E[(P-P_{n-1})^2]/\ov{p}_n
<\infty$ and define $N_j=\sum_{k=1}^N \1{R_k\geq j}$ whose mean is $\E N_j=
N\ov p_j$. Then $\hat{P}_{\ISE,N}$ in~\eqref{eq:Matti_estimator} is the 
uniform stratified ISE satisfying $\E\hat{P}_{\ISE,N}=\E P$ and 
$\lim_{N\to\infty}N\V[\hat{P}_{\ISE,N}]=\sum_{n=1}^\infty
(\V[P-P_{n-1}]-\V[P-P_{n}])/\ov p_n$ with cost $N\sum_{n=1}^\infty\ov p_n\C(n)$.
\end{thm}

\begin{rem}\label{rem:Matti}
The asymptotic inverse relative efficiencies 
(see~\cite[Sec.~6,~p.~12]{MR3782809} for definition) of STE and ISE, denoted 
by $IRE_\STE$ and $IRE_\ISE$, respectively, are given by
\begin{equation*}\begin{split}
IRE_\STE & =  \left(\sum_{n=1}^\infty \frac{\V[P_n-P_{n-1}]}{p_n} \right)
\left( \sum_{n=1}^\infty p_n \C(n)\right) 
\geq \sum_{n=1}^\infty \sqrt{V_\STE(n)\C(n)}\\
IRE_\ISE & =  \left(\sum_{n=1}^\infty 
\frac{\V[P-P_{n-1}]-\V[P-P_{n}]}{\ov{p}_n} \right)
\left( \sum_{n=1}^\infty \ov{p}_n \C(n)\right)
\geq \sum_{n=1}^\infty \sqrt{V_\ISE(n)\C(n)},
\end{split}\end{equation*}
where $V_\STE(n)=\V[P_n-P_{n-1}]$, $V_\ISE(n)=\V[P-P_{n-1}]-\V[P-P_n]$. The 
lower bounds follow from the Cauchy-Schwarz inequality, do not depend on the 
choice of the law $(p_n)_{n\in\N}$ and are attained by taking 
\begin{equation}
\label{eq:optimal_prob_unbiased}
p_n^\STE = \frac{ \sqrt{V_\STE(n)/\C(n)}}{\sum_{k=1}^\infty\sqrt{V_\STE(k)/\C(k)}}
\quad\text{and}\quad
\ov p_n^\ISE = \frac{ \sqrt{V_\ISE(n)/\C(n)}}{\sum_{k=1}^\infty\sqrt{V_\ISE(k)/\C(k)}}.
\end{equation}
Hence these choices are clearly optimal. 
\end{rem}

\section{\label{sec:Reg}Regularity of the density of the supremum $\ov X_T$}
In this appendix we discuss the necessity of the  Assumption~\ref{asm:H} 
in Proposition~\ref{prop:barrier}.

\begin{ex}\label{ex:NotLip}
For any $\gamma\in(0,1)$ there exists a L\'evy process $X$ with an absolutely continuous 
L\'evy measure $\nu$ such that 
$\liminf_{u\downarrow 0}u^{\alpha-2}\ov{\sigma}^2(u)>0$ 
holds for some $\alpha\in(0,1)$ and
Assumption~\ref{asm:H} fails for $\gamma$ at countably many $M>0$.
\end{ex}

Recall $\ov{\sigma}^2(\kappa)=\int_{(-\kappa,\kappa)}x^2\nu(dx)$ for 
$\kappa\in(0,1)$ and note that $X$ in Example~\ref{ex:NotLip} has smooth 
transition densities by~\cite[Prop.~28.3]{MR3185174}.

\begin{proof}
The essence of the proof is to construct any such $M$ as a 
singularity of the density of $\nu$. For simplicity and to make 
things explicit, we shall prove it for a single and fixed $M>0$. 
To that end, let $S$ be an $\alpha$-stable process with positivity 
parameter $\rho=\P(S_1>0)\in(0,1)$ satisfying 
$\alpha\rho+\alpha+\rho<\gamma$. Let $Z$ be an independent 
L\'evy process with finite L\'evy measure $\nu_Z$ given by 
$\nu_Z((-\infty,x]\setminus\{0\}) = \min\{1,(\max\{x,M\}-M)^\rho\}$ 
and put $X=S+Z$. Hereafter consider only small enough 
$\epsilon>0$, namely, 
$\epsilon<\min\{(T/2)^{1/\alpha},\min\{M,1\}/2\}$. Our goal is to 
bound from below the probability 
$\P(\ov{X}_T\in [M,M+3\epsilon))$. To do this, we consider the 
event where $Z$ jumps exactly once, $S$ is small, $\ov{S}\leq M$ 
at the time of that jump and $S$ does not increase too much after 
the jump.

Since the density of $S_1$ is positive, continuous and bounded, 
it follows from the scaling property that there is some constant 
$K_1>0$ (not depending on $\epsilon$) such that for all 
$t\leq \epsilon^\alpha$, 
\[
\P (S_{t}\in[0,\epsilon),\ov{S}_{t}\leq M)=
\P (S_{1}\in [0,t^{-1/\alpha}\epsilon),\ov{S}_1\leq t^{-1/\alpha}M)\geq K_1.
\]
From~\cite[Thm~4A]{MR0415780}, we also know that $\P
(\ov{S}_t\leq\epsilon)\geq K_2 \epsilon^{\alpha\rho}$ for some constant
$K_2>0$ and all $t>T-\epsilon^\alpha/2$. Now, $Z_T\in [M,M+\epsilon)$ has
probability $e^{-T}T\epsilon^\rho$ since it can only happen if $Z$ had a
single jump on $[0,T]$, whose time $U$ is then conditionally distributed
$\U(0,T)$. For fixed $t\in(0,T)$, the Markov property gives 
\[
\P\bigg[\sup_{s\in [0,T-t]}S_{s+t}-S_t\in A,(S_t,\ov{S}_t)\in B\times C\bigg]
=\P[\ov{S}_{T-t}\in A]\P[(S_t,\ov{S}_t)\in B\times C],
\]
for all measurable $A,B,C\subset\R$. Hence, multiplying by the density of $U$ at $t$, integrating and using the independence of $(U,Z)$ and $S$, we obtain
\[\begin{split}
& \P(\ov{X}_T\in [M,M+3\epsilon))\geq \P(Z_T\in [M,M+\epsilon),
S_{U}\in[0,\epsilon),\ov{S}_{U}\leq M,\ov{X}_T\in [M,M+3\epsilon)) \\
&\geq e^{-T}T\epsilon^\rho\int_0^T\P\bigg(
\sup_{s\in[0,T-t]}S_{s+t}-S_t\leq\epsilon,S_t\in[0,\epsilon),\ov{S}_t\leq M
\Big| Z_T\in[M,M+\epsilon),U=t\bigg)\frac{dt}{T} \\
&\geq  e^{-T}\epsilon^\rho\int_0^{\epsilon^\alpha}
\P(\ov{S}_{T-t}\leq\epsilon)\P(S_{t}\in[0,\epsilon),\ov{S}_t\leq M)dt 
\geq e^{-T}K_1K_2\epsilon^{\alpha\rho + \alpha+\rho}.
\end{split}\]
This implies that $x\mapsto\P(\ov X_T\leq x)$ is not locally $\gamma$-H\"older 
continuous at $M$. 
\end{proof}


\bibliographystyle{amsalpha}
\bibliography{References}

\section*{Acknowledgement} 
\noindent
JGC and AM are supported by The Alan Turing Institute under the EPSRC grant EP/N510129/1;
AM supported by EPSRC grant EP/P003818/1 and the Turing Fellowship funded by the Programme on Data-Centric Engineering of Lloyd's Register Foundation;
GUB supported by CoNaCyT grant FC-2016-1946 and UNAM-DGAPA-PAPIIT grant IN115217; 
JGC supported by CoNaCyT scholarship 2018-000009-01EXTF-00624 CVU 699336.
\end{document}